\documentclass[
a4paper,
11pt, 
DIV=14,
toc=bibliography, 
headsepline, 
parskip=half-,
abstract=true,
]{scrartcl}

\usepackage[english]{babel}
\usepackage[utf8]{inputenc}
\usepackage[T1]{fontenc}
\usepackage[english,cleanlook]{isodate}
\usepackage{lmodern}
\usepackage[babel=true]{microtype}
\usepackage{amsmath}
\usepackage{mathtools}
\usepackage{amssymb}
\usepackage{amsthm}
\usepackage[markcase=noupper]{scrlayer-scrpage} 
\usepackage{enumitem}
\usepackage{tikz} 
\usepackage{tikz-3dplot} 
\usepackage{mathrsfs}

\setkomafont{pageheadfoot}{\footnotesize}
\automark[section]{section}
\ohead{A.\,Carlotto, M.\,B.\,Schulz, D.\,Wiygul}
\ihead{\headmark}

\providecommand{\R}{\mathbb{R}}
\providecommand{\C}{\mathbb{C}}
\providecommand{\N}{\mathbb{N}}
\providecommand{\Z}{\mathbb{Z}}
\providecommand{\Sp}{\mathbb{S}}
\providecommand{\B}{\mathbb{B}}
\providecommand{\pri}{\mathbb{P}}
\providecommand{\apr}{\mathbb{A}}
\providecommand{\G}{\mathbb{G}}
\providecommand{\refl}{\underline{\mathsf{R}}}

\providecommand{\pyr}{\mathbb{Y}}
\providecommand{\rot}{\mathsf{R}} 
\providecommand{\st}{\, :\ } 
\providecommand{\dir}{\mathrm{D}}
\providecommand{\neum}{\mathrm{N}}

\providecommand{\indexform}[1]{Q^{#1}}
\providecommand{\hsd}{\mathscr{H}} 
\providecommand{\twsthom}{\sigma}  
\providecommand{\cycgrp}[1]{\boldsymbol{#1}}

\providecommand{\lipdom}{\Omega}
\providecommand{\potential}{q}
\providecommand{\robinpotential}{r}
\providecommand{\closure}[1]{\overline{#1}}  
\providecommand{\bdy}[1]{\partial_{#1}}
\providecommand{\dbdy}{\bdy{\dir}}
\providecommand{\nbdy}{\bdy{\neum}}
\providecommand{\rbdy}{\bdy{\mathrm R}}
\providecommand{\ltwosob}[1]{H^{#1}}
\providecommand{\sob}{\ltwosob{1}}
\providecommand{\sobd}[1]{\sob_{#1}}
\providecommand{\twoff}[1]{A^{#1}}
\providecommand{\met}[1]{g^{#1}}
\providecommand{\auxmet}{g_{\mathrm{aux}}}
\providecommand{\hausmeas}[2]{\mathscr{H}^{#1}(#2)}
\providecommand{\hausint}[2]{d\hausmeas{#1}{#2}}
\providecommand{\weakbf}{T}
\providecommand{\eigenvalc}[3]{\lambda_{#2}^{#3}(#1)}

\providecommand{\eigenvalsym}[4]{\eigenvalc{#1}{#2}{#3,#4}}
\providecommand{\eigenspc}[3]{E_{#3}^{#2}(#1)}

\providecommand{\eigenspsym}[4]{\eigenspc{#1}{#2}{#3,#4}}
\providecommand{\proj}{\pi}

\providecommand{\equivind}[1]{\ind_{#1}}
\providecommand{\equivnul}[1]{\nul_{#1}}
\providecommand{\symind}[2]{\ind_{#1}^{#2}}
\providecommand{\symnul}[2]{\nul_{#1}^{#2}}
\providecommand{\grp}{G}

\providecommand{\invproj}[2]{\proj_{#1,#2}}
\providecommand{\nrmlsgn}[1]{\sgn_{#1}}
\providecommand{\inttext}{\mathrm{int}}
\providecommand{\exttext}{\mathrm{ext}}
\providecommand{\intbdy}{\partial_\inttext}
\providecommand{\extbdy}{\partial_{\exttext}}
\providecommand{\dint}[1]{#1^{\mathrm{D}_\inttext}}
\providecommand{\nint}[1]{#1^{\mathrm{N}_\inttext}}

\providecommand{\subspace}{\underset{\clap{\scriptsize subspace}}{\subset}}
\providecommand{\taubar}{\underline{\tau}}
\providecommand{\force}{\mathcal{F}}
\providecommand{\disloc}{\mathcal{D}}
\providecommand{\ptofd}{\mathcal{E}}
\providecommand{\catr}{\Gamma}
\providecommand{\discr}{\Delta}
\providecommand{\interr}{\Theta}
\providecommand{\catdom}{\Lambda^K}
\providecommand{\margindom}{\Lambda^M}
\providecommand{\sunder}{{\underline{\sigma}}}
\providecommand{\thunder}{{\underline{\theta}}}

\providecommand{\wbar}{\overline{w}}
\providecommand{\mubar}{\overline{\mu}}

\DeclarePairedDelimiter\abs{\lvert}{\rvert}
\DeclarePairedDelimiter\nm{\lVert}{\rVert}
\DeclarePairedDelimiter\sk{\langle}{\rangle}
\DeclarePairedDelimiter\interval{]}{[}
\DeclarePairedDelimiter\Interval{[}{[}
\DeclarePairedDelimiter\intervaL{]}{]}
\DeclarePairedDelimiter\IntervaL{[}{]}
\DeclarePairedDelimiter\floor{\lfloor}{\rfloor}
\DeclarePairedDelimiter\ceil{\lceil}{\rceil}

\DeclareMathOperator{\Ogroup}{O} 
\DeclareMathOperator{\SOgroup}{SO} 
\DeclareMathOperator{\ind}{ind} 
\DeclareMathOperator{\sgn}{sgn}
\DeclareMathOperator{\nul}{nul}
\DeclareMathOperator{\ord}{ord}
\DeclareMathOperator{\dom}{dom}
\DeclareMathOperator{\arcosh}{arcosh}
\DeclareMathOperator{\sech}{sech}
\DeclareMathOperator{\csch}{csch}

\usepackage[initials]{amsrefs}

\definecolor{Blau}{cmyk}{1, 0.5, 0, 0}
\definecolor{Orange}{cmyk}{ 0,.5, 1, 0}
\usetikzlibrary{calc,plotmarks}
\tikzset{declare function={
step(\x)=min(max(\x,1.0001),1.9999);
cutoff(\x)=exp(-1/(2-step(\x)))/(exp(-1/(2-step(\x)))+exp(-1/(step(\x)-1)));
}}
\makeatletter
\tikzset{
    scale plot marks/.is choice,
    scale plot marks/true/.style={},	
    scale plot marks/false/.code={
        \def\pgfuseplotmark##1{\pgftransformresetnontranslations\csname pgf@plot@mark@##1\endcsname}
    },
every mark/.append style={scale plot marks=false},
plus/.style={mark=+,mark size=2.25pt},
vdash/.style={mark=|,mark size=2.25pt},
hdash/.style={mark=-,mark size=2.25pt},
bullet/.style={mark=*,mark size=1.125pt},
}
\newenvironment{leqnoalign}{%
 \tagsleft@true
 \align%
 }
 {\endalign}
\let\TagsLeftOn\tagsleft@true
\let\TagsLeftOff\tagsleft@false
\makeatother 
\pgfmathsetmacro{\unitscale}{3.6}  
\usepackage[
pdftitle={Disc stackings and their Morse index},
pdfauthor={Alessandro Carlotto, Mario B. Schulz and David Wiygul}, 
pdfborder={0 0 0},
]{hyperref}

\theoremstyle{plain}
\newtheorem{theorem}{Theorem}[section]
\newtheorem{lemma}[theorem]{Lemma}
\newtheorem{corollary}[theorem]{Corollary} 
\newtheorem{proposition}[theorem]{Proposition}

\theoremstyle{definition}

\theoremstyle{remark}
\newtheorem{remark}[theorem]{Remark}

\numberwithin{equation}{section}

\title{Disc stackings and their Morse index}
\author{Alessandro Carlotto, Mario B. Schulz, David Wiygul}
\date{
\vspace*{-3ex}
} 

\newcommand\printaddress{{
\setlength{\parindent}{17pt}
\bigskip 
\par
\vbox{
{\scshape Alessandro Carlotto}
\newline Universit\`a di Trento, 
Dipartimento di Matematica,
via Sommarive 14, 
38123 Povo di Trento, 
Italy
\newline
\textit{E-mail address:} 
\texttt{alessandro.carlotto@unitn.it}
\par\medskip
{\scshape Mario B. Schulz}
\newline Universit\`a di Trento, 
Dipartimento di Matematica,
via Sommarive 14, 
38123 Povo di Trento, 
Italy
\newline
\textit{E-mail address:} 
\texttt{mario.schulz@unitn.it}
\par\medskip
{\scshape David Wiygul}
\newline Universit\`a di Trento, 
Dipartimento di Matematica,
via Sommarive 14, 
38123 Povo di Trento, 
Italy
\newline
\textit{E-mail address:} 
\texttt{davidjames.wiygul@unitn.it}
\par
}
}}

\begin{document}

\maketitle

\begin{abstract}
We construct free boundary minimal disc stackings, with any number of strata, in the three-dimensional Euclidean unit ball, and prove uniform, linear lower and upper bounds on the Morse index of all such surfaces. 
Among other things, our work implies for any positive integer $k$ the existence of $k$-tuples of distinct, pairwise non-congruent, embedded free boundary minimal surfaces all having the same topological type.
In addition, since we prove that the equivariant Morse index of any such free boundary minimal stacking, with respect to its maximal symmetry group, is bounded from below by (the integer part of) half the number of layers, it follows that any possible realization of such surfaces via an equivariant min-max method would need to employ sweepouts with an arbitrarily large number of parameters. This also shows that it is only for $N=2$ and $N=3$ layers that free boundary minimal disc stackings can be obtained by means of one-dimensional mountain pass schemes.
\end{abstract}


\section{Introduction}

We collect and summarize in the following statement the main results contained in the present article:

\begin{theorem}\label{thm:Main}
Let $N\geq 2$. 
There exists $m_0=m_0(N)$ such that for every integer $m>m_0$ there exist a (properly embedded) free boundary minimal surface $\Sigma_{N,m}$ in $\B^3$ having genus and number of boundary components respectively equal to
\begin{align*}
\gamma &=\begin{cases}
\frac{1}{2}(m-1)(N-2)& \text{ if $N$ is even,}\\[.5ex]
\frac{1}{2}(m-1)(N-1)& \text{ if $N$ is odd,}
\end{cases} & 
\beta &=\begin{cases}
m& \text{ if $N$ is even,}\\[.5ex]
1& \text{ if $N$ is odd.}
\end{cases}	
\end{align*}
and symmetry group
\begin{equation*}
\G=
\begin{cases}
\pri_m  \text{ (the prismatic group of order $4m$) } &\text{ if $N$ is even, }
\\
\apr_m  \text{ (the antiprismatic group of order $4m$) } &\text{ if $N$ is odd; } 
\end{cases}
\end{equation*}
as $m\to\infty$ the sequence $(\Sigma_{N,m})_{m\geq m_0}$ converges in the sense of varifolds with multiplicity $N$ (and smoothly away from the boundary circle) to the equatorial disc.
Furthermore the $\G$-equivariant Morse index and nullity of $\Sigma_{N,m}$ satisfy the following two-sided bounds:
\begin{align*}
\ind_{\G}(\Sigma_{N,m})
&\geq
\floor{N/2}
\shortintertext{and}
\ind_{\G}(\Sigma_{N,m}) + \nul_{\G}(\Sigma_{N,m})
&\leq
\begin{cases}
2N-1 &\mbox{if } N \ \text{is even},
\\
2N-2 &\mbox{if } N \ \text{is odd}.
\end{cases}
\end{align*}
\end{theorem}

In particular, if we consider the case when $N$ (the number of layers) is even, this statement ensures the solvability of the ``realization problem'' for infinitely many new topological types; in fact, it provides us with plenty of new examples of free boundary minimal surfaces in $\B^3$ with many boundary components \emph{and} correspondingly large genus, that is to say infinitely many ``diagonal sequences'' of solutions to the aforementioned topological realization problem, the very first such instances to the best of our knowledge. 

The previous theorem follows by combining Proposition \ref{pro:final_existence_with_estimates} for the core existence result, Proposition~\ref{pro:final_embd} for the embeddedness claim as well as for the determination \emph{a posteriori} of the maximal symmetry group, Lemma \ref{lem:TopologicalType} for the topological type, and the discussion in Section \ref{sec:Index} (see in particular Proposition \ref{pro:Index_Equiv_UpperBounds} and Proposition \ref{pro:Index_Equiv_LowerBounds}) for the Morse index estimates; the reader is referred to these statements, especially the first one, for further details and a more accurate description of the free boundary minimal surfaces we construct. 

An additional, striking consequence of such a construction (together with the characterization of the maximal symmetry group) is given in the following statement:

\begin{corollary}[Polymorphism]\label{cor:Polym}
For any positive integer $k$ there exists in the unit ball of the three-dimensional Euclidean space a collection of at least $k$ (connected, properly embedded) non-congruent free boundary minimal surfaces all having the same topological type.
\end{corollary}

\begin{proof}
Given $k$ a positive integer, let $p_1<p_2<\ldots< p_k$ be the first $k$ odd primes and let $m_0(\cdot)$ be as provided by Theorem \ref{thm:Main}. There exists a positive integer $q$ such that 
\[
q\prod_{i=1}^{k-1}p_i>\max_{i=1,\ldots, k} \ m_0(1+2p_i).
\]
That being said, we take 
\(
\gamma=q\prod_{i=1}^k p_i.
\)
Thanks to Theorem \ref{thm:Main} we thus get a collection of exactly $k$ free boundary minimal surfaces $\Sigma_{N_1,m_1},\ldots, \Sigma_{N_k,m_k}$ in $\B^3$, all with genus $\gamma$ and connected boundary, that are indexed by taking for $i\in\{1,\ldots,k\}$
\begin{align*}
N_i&=1+2p_i, &
m_i&=1+q\prod_{j\neq i}p_j.
\end{align*}
Such surfaces are pairwise non-congruent because each given pair can be distinguished by means of their maximal symmetry group, since in particular such a group will contain, for $\Sigma_{N_i,m_i}$, exactly $4m_i$ elements.
\end{proof}

This result should be contrasted with the uniqueness theorem obtained in 1986 by Nitsche \cite{Nitsche1985} (which is, in this specific context, the only unconditional topological uniqueness theorem to date) and compared to the main theorem obtained by the authors in \cite{CSWnonuniqueness}, where we constructed infinitely many \emph{pairs} ($k=2$) of free boundary minimal surfaces in $\B^3$ each having same topology and symmetry group.

Our construction has PDE-theoretic character and, roughly speaking, consists of firstly designing (for any given $N$ and $m$) properly embedded surfaces in $\B^3$ that meet the unit sphere orthogonally and are almost minimal (that is to say: have mean curvature close to zero after appropriately rescaling) and then perturbing such objects to exact minimality while preserving the other features, in particular the free boundary property.
To confront the cokernel one eventually encounters, such \emph{intial surfaces}, which are then deformed to exact solutions to our problem, actually come in finite-dimensional families; if we neglect, for expository convenience, the description of the actual degrees of freedom one needs to maintain (for which we refer the reader to the first part of Section~\ref{sec:Stackings}) any such initial surface consists of $N$ parallel horizontal discs, with a certain number of thin half-catenoidal ribbons joining, according to a certain symmetry pattern depending on the parity of $N$, pairs of adjacent discs. Unsurprisingly, the second step above -- that can be described as a perturbative deformation -- turns out to be feasible only when the parameter $m$ is large enough (depending on $N$).

Constructions of this type have already been implemented, in the specific setting of $\B^3$, in the special case $N=2$ (often referred to as \emph{doubling}) by work of Folha--Pacard--Zolotareva (see \cite{FolPacZol17}) and then when $N=3$ (often referred to as \emph{tripling}) by the third-named author and Kapouleas in \cite{KapouleasWiygul17}, where the possibility of extending the result to any positive $N$ is also mentioned (cf. Remark~1.2 therein).
Our methods here are certainly closer in spirit and substance to the latter of such references, although -- setting aside the predictable, yet significant challenges imposed by handling the case of general $N\geq 2$ layers -- they differ from it by some specific aspects, as we now briefly describe.
First, rather than constructing the initial surfaces
in two stages
(starting with so-called pre-initial surfaces
and then later arranging orthogonality
of the intersection with the sphere)
as in \cite{KapouleasWiygul17}*{Sections 3--4},
we instead apply the map
\eqref{def:Phi},
as in \cite{CSWnonuniqueness},
to enforce orthogonality at the boundary
from the start.
We find this choice,
although inessential,
more convenient and direct
for setting up the comparisons we make,
in the course of carrying out our index estimates,
between regions
of the surfaces we produce
and various models.
Second,
and more inevitably,
the incorporation of arbitrarily many discs
significantly complicates
the balancing analysis
(that is the object of Appendix \ref{app:forces}, to be
compared to \cite{KapouleasWiygul17}*{Lemma 7.15})
needed to control some of the obstructions
encountered in solving the linearized problem. The latter study, concerning the balancing problem, should also be compared to the stacking construction presented by the third-named author, in his earlier work \cite{Wiygul2020}, for nearly-Clifford tori in the round three-dimensional sphere $\Sp^3$.

From another perspective, it is also to be recalled here that stacking constructions are, at least in principle, also tractable by means of variational methods, i.\,e. by constructing unstable critical points of the area functional in the category of relative cycles of the Euclidean unit ball through equivariant min-max schemes. As is well-known, such methods necessarily (for compactness purposes) rely upon suitably weak notions of ``surfaces'' and ``convergence'' and thus pose very serious problems in trying to fully control the topological type of the critical points they produce. In the case when $N=2$ an approach, along such lines, to construct free boundary minimal surfaces having genus zero and an assigned number $\beta$ of boundary components, and converging in the sense of varifolds (with multiplicity two and smoothly away from the boundary) to the horizontal disc as $\beta\to\infty$, has first been presented by Ketover in \cite{Ketover16FB}; however, the methodology to properly control the topology of the resulting surfaces has only very recently been developed by the second-named author and Franz (see \cite{FraSch23}), thus resulting -- among other things -- in Theorem 5.1 therein. Similarly, the case $N=3$, i.\,e. the construction of free boundary minimal surfaces with connected boundary and any pre-assigned genus, and converging in the sense of varifolds (with multiplicity three and smoothly away from the boundary) to the horizontal disc, had been obtained in earlier work by the first- and second-named author with Franz \cite{CarFraSch20}, thereby settling -- in particular -- the long-standing question about the existence of genus-one free boundary minimal surfaces in $\B^3$ having connected boundary.  

Concerning the (equivariant and absolute) Morse index bounds, our results here (as presented in the statement of Theorem \ref{thm:Main} and in Remark \ref{rem:IndexTop} below) crucially build on the methodology developed by the authors in their most recent work \cite{CSWSpectral}, which in turn was inspired by the approach pioneered by Montiel and Ros in \cite{MontielRos} and by later work of the third-named author and Kapouleas \cite{KapouleasWiygulIndex}. 
Loosely speaking, the idea is that the symmetries we enforce throughout the gluing construction, as well as the corresponding building blocks, suggest natural partitions of the free boundary minimal surfaces we construct into a certain number of pairwise congruent domains, and the knowledge (or partial knowledge) of the index of the domains with interior Dirichlet or, respectively, Neumann boundary conditions allows us to suitably bound from below or, respectively, from above the actual Morse index of the surfaces, possibly in the presence of an underlying group action and a twisting homomorphism.
A point of novelty we wish to stress is the following one. 
We recall here that among the direct applications of such methodology is the possibility of proving linear lower bounds for free boundary minimal surfaces in $\B^3$ having prismatic (or, more generally, pyramidal) symmetry group; the exact statement is recalled below, see Proposition \ref{indBelowPyr}. 
Such a statement accounts for the order of the natural cyclic subgroup (of rotations) inside the symmetry group in question, but does not account for the number of layers of the surface. Our lower bounds here, that are Proposition \ref{pro:Index_LowerBounds} (absolute) and Proposition~\ref{pro:Index_Equiv_LowerBounds} (maximally equivariant), do provide such refinements, taking both factors into careful account.

That background information being provided, the following two remarks focus on additional key implications of our work, and are of special importance within the general framework of index estimates for minimal surfaces, and their variational (re)-construction.

\begin{remark}[Implications for min-max stacking constructions]\label{rem:MinMax}
It follows, by simply combining the equivariant index estimate stated above with the main (variational) index estimate in \cite{FranzIndex}, that:
\begin{itemize}
\item it is \emph{only} for $N=2$ (doubling) and $N=3$ (tripling) that one can realize free boundary minimal stacking constructions in $\B^3$ by employing $1$-parameter sweepouts (which have indeed been successfully implemented in \cite{FraSch23} and earlier in \cite{CarFraSch20}, respectively);
\item the (least) number $k$ of parameters so that $k$-parameter sweepouts would generate a mountain pass scheme that can, at least in principle (that is to say: modulo the delicate problem of controlling the topology of the resulting surfaces), produce free boundary minimal stacking constructions in $\B^3$ with $N$ layers, albeit not uniquely determined by the estimates above, shall grow exactly linearly in $N$; the family above seems to be the first -- in whichever setting -- explicitly displaying this type of behavior.
\end{itemize}
\end{remark}

 \begin{remark}[Index growth in terms of topological data]\label{rem:IndexTop}
Our methods also allow us to get effective two-sided bounds on the absolute Morse index and nullity of $\Sigma_{N,m}$. In particular, in Proposition \ref{pro:Index_LowerBounds} and Proposition \ref{pro:Index_UpperBounds} we prove the following estimates:
\begin{align*}
\ind(\Sigma_{N,m})
&\geq\max\{(N-1), 2\} m,
\\
\ind(\Sigma_{N,m}) + \nul(\Sigma_{N,m})
&\leq m(5N-3)+N.
\end{align*}
Such bounds, when expressed in terms of the topological data $\gamma(N,m)$ and $\beta(N,m)$, take the form:
\[
\ind(\Sigma_{N,m}) \geq 2 \gamma(N,m)+\beta(N,m)+N-2 = N-\chi(N,m)
\]
where $\chi(N,m)$ denotes the Euler characteristic of the surface in question, and
\begin{equation*}
\ind(\Sigma_{N,m}) + \nul(\Sigma_{N,m})
\leq
\begin{cases}
10\gamma(N,m)+7\beta(N,m)+6(N-1)-4
 &\mbox{if } N \ \text{is even},
\\
10\gamma(N,m)+\beta(N,m)+6(N-1)+2m &\mbox{if } N \ \text{is odd}.
\end{cases}
\end{equation*}
These result significantly improve the general index estimate obtained in \cite{AmbCarSha18-Index}, as well as the ineffective linear upper bound in \cite{Lim17}.
\end{remark}

On yet another front, we also wish to add one more comment, of special interest from the ``spectral analysis'' perspective, and to be considered in light of \cite{FraSch11}, \cite{FraSch13}, \cite{FraSch16}.
Indeed, we point out that,
by a result
\cite[Theorem 5.2]{KusnerMcGrathFirstSteklov}
of Kusner and McGrath,
each $\Sigma_{N,m}$ has $1$ as its first Steklov eigenvalue
and in fact 
the span of the coordinate functions
as its corresponding eigenspace,
which follows by observing that
each symmetry group $\G$
and each quotient $\Sigma_{N,m} / \G$
indeed satisfy the description
given in Section 6 therein. 

Lastly, recalling the ample, well-established (albeit partial) parallel between the theory of closed minimal surfaces in $\Sp^3$ and that of free boundary minimal surfaces in $\B^3$, it would definitely be of interest to study the analogous stacking problem for nearly equatorial spheres in $\Sp^3$. Such a construction appears -- in its full generality -- non-trivial already at the linear level, when it comes to determining the actual degrees of freedom in placing the catenoidal tunnels connecting adjacent spheres; in the doubling case ($N=2$) these aspects have been thoroughly investigated by Kapouleas in \cite{KapouleasSphereDoublingI} and then in later joint work with McGrath (see \cite{KapouleasMcGrathSphereDoublingII}).

The present work is meant to be a continuation of the earlier articles \cite{CSWnonuniqueness} and \cite{CSWSpectral}, as well as -- to some extent -- of \cite{CarFraSch20} by the first two authors with Franz, and for this reason, we have decided to privilege the analysis in $\B^3$. 
In retrospect, what emerges from the present work is perhaps not only, or primarily, the construction in itself, but the first sketch of a comparative picture of variational vs. perturbative methods, which in turn crucially builds on combined information coming from the fine analysis, in both contexts, of the (equivariant or absolute) Morse index. A deeper understanding of such a landscape is a fascinating matter for future investigation.

\paragraph{Acknowledgments.} 
This project has received funding from the European Research Council (ERC) under the European Union’s Horizon 2020 research and innovation programme (grant agreement No. 947923). 
The research of M.\,S. was funded by the Deutsche Forschungsgemeinschaft (DFG, German Research Foundation) under Germany's Excellence Strategy EXC 2044 -- 390685587, Mathematics M\"unster: Dynamics--Geometry--Structure, and the Collaborative Research Centre CRC 1442, Geometry: Deformations and Rigidity. 
Part of this article was finalized while A.\,C. was visiting the ETH-FIM and then while A.\,C. and D.\,W. were visting the Mathematisches Forschungsinstitut Oberwolfach on the occasion of the program \emph{Differentialgeometrie im Grossen}; the excellent working conditions at both institutes are gratefully acknowledged.
Finally, the authors thank the anonymous referee for their valuable suggestions.

\section{Free boundary minimal stackings}\label{sec:Stackings}

\paragraph{Notation for symmetries.} In order to define the symmetry groups we will work with, we need to consider two classes of generators:
\begin{itemize}
\item given a plane $\Pi \subset \R^3$ through the origin,
we write $\refl_\Pi \in \Ogroup(3)$
for reflection through $\Pi$;
\item given a directed line $\xi \subset \R^3$ through the origin and an angle $\theta \in \R$, we write $\rot_\xi^\theta$ for rotation about $\xi$ through angle $\alpha$ in the usual right-handed sense. 
\end{itemize}
That said, given symmetries $\mathsf{T}_1, \ldots, \mathsf{T}_k \in \Ogroup(3)$, we write $\sk{\mathsf{T}_1, \ldots, \mathsf{T}_k}$ for the subgroup they generate.
The order-$2$ groups generated by reflections through planes will figure repeatedly in the sequel, 
so for succinctness of notation we agree that whenever the label of a given plane is marked in bold we mean instead the generated group of reflections: 
for example, given a plane $\Pi \subset \R^3$ through the origin, we set $\cycgrp{\Pi}\vcentcolon=\sk{\refl_\Pi}$.  
We will employ the apex $+$ (respectively: $-$) to denote functions that are even (respectively: odd) with respect to the reflection through $\Pi$.
In this article, we are particularly interested in four symmetry groups:
\begin{alignat}{3}
\label{eqn:cyclic}&\text{the \emph{cyclic group} }&
\Z_k&\vcentcolon=\sk[\Big]{\rot_{\{x=y=0\}}^{2\pi/k}} 
&&\text{ of order } k, \qquad 
\\ 
\label{eqn:pyramidal}&\text{the \emph{pyramidal group} }&
\pyr_k&\vcentcolon=\sk[\Big]{\rot_{\{x=y=0\}}^{2\pi/k},~\refl_{\{y=x \tan \frac{\pi}{2k}\}}}
&&\text{ of order } 2k, 
\\ 
\label{eqn:prismatic}&\text{the \emph{prismatic group} } & 
\pri_k&\vcentcolon=\sk[\Big]{\rot_{\{x=y=0\}}^{2\pi/k},~\refl_{\{y=x \tan \frac{\pi}{2k}\}},~\refl_{\{z=0\}}}
&&\text{ of order } 4k,
\intertext{which is the maximal symmetry group of a right prism over a regular $k$-gon as visualized e.\,g. in Section~2 of \cite{CSWnonuniqueness}, }
\label{eqn:antiprismatic}&\text{the \emph{antiprismatic group} }&
\apr_k&\vcentcolon=\sk[\Big]{\rot_{\{x=y=0\}}^{2\pi/k},~\refl_{\{y=x \tan \frac{\pi}{2k}\}},~\rot_{\{y=z=0\}}^\pi}
&&\text{ of order } 4k,
\end{alignat}
which is the maximal symmetry group of a standard antiprism over a regular $k$-gon as visualized e.\,g. in Section 2 of \cite{CSWnonuniqueness}. 

\paragraph{Formal disc stackings and their topology.}
The surfaces we construct will be indexed by integers $N \geq 2$ and $m \gg 1$, 
the former indicating the number of copies of $\B^2$ to be considered and the latter specifying the pyramidal group $\pyr_m$ to be imposed on the construction;
correspondingly the construction features $(N-1)m$ half-catenoidal ribbons, i.\,e. $m$ ribbons between each pair of adjacent discs. 
In this section we agree to employ the word ``ribbon'' to denote a compact surface that is homeomorphic to $[0,1]\times [0,1]$; further specifications, of geometric character, will be added in the sequel of the article. 

One also imposes reflectional symmetry
through the plane $\{z=0\}$ in the case of even $N$
and rather through the line $\{y=z=0\}$ in the case of odd $N$,
so that the full symmetry group imposed
is correspondingly
\begin{equation*}
\G
=
\G_{N, m}
\vcentcolon=
\begin{cases}
\pri_m & \text{ for $N$ even,} 
\\
\apr_m & \text{ for $N$ odd,} 
\end{cases}
\end{equation*}
as defined in \eqref{eqn:prismatic}--\eqref{eqn:antiprismatic}. 
In particular, we stress that in either case the cyclic group $\Z_m$ of order $m$ is a subgroup of $\G$, as also apparent from the equations above. 

As anticipated in the introduction, the free boundary minimal surface
corresponding to given data $(N,m)$
is produced in two stages.
First an approximate free boundary minimal surface,
called an initial surface,
is assembled by taking unions of the sets
in \eqref{building_blocks} below
for suitably chosen values of the parameters
appearing there.
Then the initial surface is perturbed
to an exact free boundary minimal surface
by solving the corresponding nonlinear boundary value problem.
Actually, in this last step,
one encounters an obstruction space,
at the linear level,
of dimension $N-1$.
For this reason,
corresponding to each pair of data $(N,m)$
one constructs in the first step
not just a single initial surface
but rather
an $(N-1)$-parameter family
of initial surfaces.
The $N-1$ parameters will be associated
with the sizes and heights
of the catenoidal waists.

It will be convenient to write $n\vcentcolon=\floor{N/2}=\max\{k\in\Z\st k\leq N/2\}$ for the greatest integer no larger than $N/2$, so that
\begin{equation*}
N=\begin{cases}
2n &\text{ for $N$ even,} \\
2n+1 &\text{ for $N$ odd.}
\end{cases}
\end{equation*}
Before proceeding to define the initial surfaces
let us pause to discuss the topology of the disc stackings, as combinatorially described above. In fact, we prove a more general result.

\begin{lemma}\label{lem:TopologicalType}
Given $3\leq m\in\N$ let $\Z_m$ be the cyclic group of order $m$ as defined in \eqref{eqn:cyclic}. 
Let $\Gamma\subset\B^3$ be a properly embedded, $\Z_m$-equivariant surface which is obtained by taking the union of $N\in\N$ disjoint, parallel, horizontal discs and connecting each adjacent pair through $m$ ribbons. 
Let $\gamma$ denote the genus of $\Gamma$ and $\beta$ the number of its boundary components. 
Then 
\begin{align*}
\gamma &=\begin{cases}
\frac{1}{2}(m-1)(N-2)& \text{ if $N$ is even,}\\[.5ex]
\frac{1}{2}(m-1)(N-1)& \text{ if $N$ is odd,}
\end{cases} & 
\beta &=\begin{cases}
m& \text{ if $N$ is even,}\\[.5ex]
1& \text{ if $N$ is odd.}
\end{cases}	
\end{align*}
\end{lemma}

\begin{figure}%
\pgfmathsetmacro{\rib}{5mm}
\pgfmathsetmacro{\lw}{3pt}
\begin{tikzpicture}[line join=round,scale=2.5,baseline={(0,0)},thick]
\foreach\k in {1,2,3}{
\begin{scope}[shift={(\k-1,0)}]
\path(0,0.5)coordinate(d0)--++(1,0)coordinate(d3)
node[pos=0.15,sloped,minimum size=\rib,inner sep=0](d1){}
node[pos=0.44,sloped,minimum size=\rib,inner sep=0](d2){}
;
\path[Blau](0,0)to[out=60,in=-120,looseness=3]
node[pos=0.09,sloped,minimum size=\rib,inner sep=0](b1){}
node[pos=0.37,sloped,minimum size=\rib,inner sep=0](b2){}
(1,0)
;
\draw[line width=\lw,black!33]
(d0)--(d1.center)..controls(d1.east)and(b1.east)..(b1.center)
(b2.center)..controls(b2.west)and(d2.west)..(d2.center)--(d3);
\path(d1)to[bend right=45]coordinate[midway](tip)(d2);
\path(tip)node[above=4ex,inner sep=1pt](rib){$\mathrm{rib}_\k$};
\begin{scope}
\clip(b1.center)--(d0)to[bend right=120,looseness=4](d3)--(b2.center)to[bend left]cycle;
\draw[line width=\lw,black!33] (0,0)to[out=60,in=-120,looseness=3](1,0);
\end{scope}
\draw[Blau](d0)--(d3) 
 (0,0)to[out=60,in=-120,looseness=3](1,0);
\end{scope}
\draw[latex-](tip)to[bend left=4](rib);
}
\draw[Blau](2,0.5)node[above left]{$\partial D$}(2,-0.4)node[above]{$\partial\Gamma'$};
\pgfresetboundingbox
\useasboundingbox(0,-0.5)rectangle(3,0.9);
\end{tikzpicture}
\hfill
\begin{tikzpicture}[line join=round,scale=2.5,baseline={(0,0)},thick]
\foreach\k in {1,2,3}{
\begin{scope}[shift={(\k-1,0)}]
\path(0,0.5)coordinate(d0)--++(1,0)coordinate(d3)
node[pos=0.28,sloped,minimum size=\rib,inner sep=0](d1){}
node[pos=0.62,sloped,minimum size=\rib,inner sep=0](d2){}
;
\path[Blau](0.9,0)arc(0:360:0.43 and 0.3)
node[pos=0.18,sloped,minimum size=\rib,inner sep=0](b2){}
node[pos=0.33,sloped,minimum size=\rib,inner sep=0](b1){}
node[pos=0.75,below](beta){$\beta_\k$}
;
\draw[line width=\lw,black!33]
(d0)--(d1.center)..controls(d1.east)and(b1.east)..(b1.center)
(b2.center)..controls(b2.west)and(d2.west)..(d2.center)--(d3);
\path(d1)to[bend right=45]coordinate[midway](tip)(d2);
\path(tip)node[above=4ex,inner sep=1pt](rib){$\mathrm{rib}_\k$};
\begin{scope}
\clip(b1.center)to[bend right=45](d0)|-(beta.center)-|(d3)to[bend right=45](b2.center)to[bend left=60]cycle;
\draw[line width=\lw,black!33](0.9,0)arc(0:360:0.43 and 0.3);
\end{scope}
\draw[Blau](d0)--(d3)
(0.9,0)arc(0:360:0.43 and 0.3);
\end{scope}
\draw[latex-](tip)to[bend left=4](rib);
}
\draw[Blau](2,0.5)node[above]{$\partial D$};
\pgfresetboundingbox
\useasboundingbox(0,-0.5)rectangle(3,0.9);
\end{tikzpicture}
\caption{The two cases for the structure of the boundary when attaching a horizontal disc $D$ equivariantly through vertical half-neck ribbons $\mathrm{rib}_k$.}%
\label{fig:boundary_components}%
\end{figure}

\begin{proof}
We determine the number $\beta$ of boundary components by induction in $N$.  
For $N=1$ we simply have one disc (without any ribbons) which clearly has connected boundary. 
Then suppose $\Gamma'\subset\R^3$ is any embedded, $\Z_m$-equivariant surface with connected boundary and let $\Gamma''$ be the embedded surface which we obtain from $\Gamma'$ by $\Z_m$-equivariantly attaching a horizontal disc $D$ through $m$ vertical half-neck ribbons $\mathrm{rib}_1,\ldots,\mathrm{rib}_m$. 
Then $\Gamma''$ has exactly $m$ pairwise isometric boundary components. 
Indeed, the $k$-th boundary component formally consists of four segments (see Figure \ref{fig:boundary_components} left image):
\begin{itemize}[nosep]
\item the right vertical boundary component of $\mathrm{rib}_{k}$, 
\item the (short) segment on $\partial\Gamma'$ between the $\mathrm{rib}_{k}$ and $\mathrm{rib}_{k+1}$,
\item the left vertical boundary component of $\mathrm{rib}_{k+1}$, 
\item the (short) segment on $\partial D$ between the $\mathrm{rib}_k$ and $\mathrm{rib}_{k+1}$.
\end{itemize}
Conversely, suppose that $\Gamma'\subset\R^3$ is any embedded, $\Z_m$-equivariant surface with $m$ pairwise isometric boundary components $\beta_1,\ldots,\beta_m$ and let $\Gamma''$ be the embedded surface which we obtain from $\Gamma'$ by $\Z_m$-equivariantly attaching a horizontal disc $D$ as above. 
Then $\partial\Gamma''$ is obtained from $\partial D$ by taking the ``connected sum'' with $\beta_k$ through the vertical boundary segments of the ribbon $\mathrm{rib}_{k}$; more precisely, index $i=1,\ldots, k$ we remove two half-discs and attach a half cylinder (a ribbon) inbetween. In particular, $\partial\Gamma''$ is connected (see Figure \ref{fig:boundary_components} right image).

\begin{figure}%
\centering 
\pgfmathsetmacro{\thetaO}{72}
\pgfmathsetmacro{\phiO}{90+1*45/2} 
\tdplotsetmaincoords{\thetaO}{\phiO} 
\begin{tikzpicture}[tdplot_main_coords,line cap=round,line join=round,scale=\unitscale,baseline={(0,0,0)}]
\draw(0,0,0)node[inner sep=0,scale={\unitscale/4.015}](S){\includegraphics[page=2]{figures}};
\end{tikzpicture}
\hfill
\begin{tikzpicture}[tdplot_main_coords,line cap=round,line join=round,scale=\unitscale,baseline={(0,0,0)}]
\draw(0,0,0)node[inner sep=0,scale={\unitscale/4.015}](S){\includegraphics[page=1]{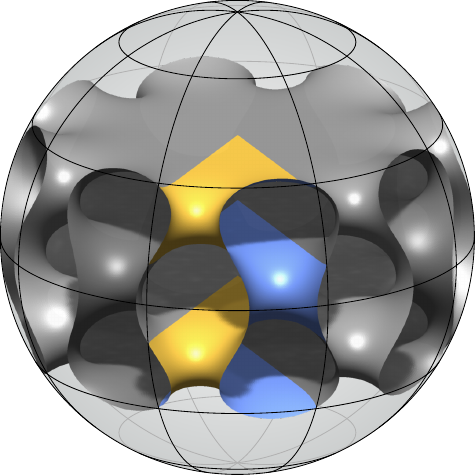}};
\end{tikzpicture}
\caption{A free boundary minimal disc stacking with $N=5$ (left image) respectively $N=4$ (right image) and $m=8$. 
Highlighted is a fundamental domain with respect to the $\Z_m$-action. 
}%
\label{fig:fundamentaldomain-Zm-action}%
\end{figure} 

It remains to compute the genus $\gamma$.  
Given any $\Z_m$-equivariant surface $\Sigma$, we call a closed subset $\Omega\subset\Sigma$ \emph{fundamental domain} of $\Sigma$ if 
\(
\Sigma=\bigcup_{\varphi\in\Z_m}\varphi(\Omega)	
\) 
and if the set $\varphi_1(\Omega)\cap \varphi_2(\Omega)$ has no interior in $\Sigma$ for any $\varphi_1,\varphi_2\in\Z_m$ with $\varphi_1\neq\varphi_2$. 
We claim that we can choose a connected fundamental domain $\Omega$ of $\Gamma$ which is a topological disc with piecewise smooth boundary.  
By assumption, $\Gamma$ contains a disjoint union of $\Z_m$-equivariant discs $D_1,\ldots,D_N$ such that the closure of the complement $\Gamma\setminus(D_1\cup\ldots\cup D_N)$ has $(N-1)m$ connected components $K_j^i$ for $j=1,\ldots,N-1$ and $i=1,\ldots,m$, each of which is homeomorphic to $[0,1]\times[0,1]$.  
These ribbons are labelled such that $K_j^i$ connects $D_j$ and $D_{j+1}$ and such that 
$K_j^1$ and $K_{j+1}^1$ are adjacent along a boundary segment of $\Gamma$.  
There clearly exists a connected fundamental domain $D_1^1$ of $D_1$ intersecting $K_1^1$ and 
 a connected fundamental domain $D_N^1$ of $D_N$ intersecting $K_{N-1}^1$. 
Since $K_j^1$ and $K_{j+1}^1$ are adjacent with disjoint $\Z_m$-orbits, there also exists a connected fundamental domain $D_j^1$ of $D_j$ intersecting both $K_j^1$ and $K_{j+1}^1$ for any $j\in\{2,\ldots,N-1\}$. 
Note that all these sets can be chosen to have piecewise smooth boundary. 
Hence, the set $\Omega\vcentcolon=K_1^1\cup\ldots\cup K_{N-1}^1\cup D_1^1\cup\ldots\cup D_{N}^1$ has all the desired properties, in particular, it is a topological disc being an acyclic concatenation of topological discs.  
Two examples are visualized in Figure \ref{fig:fundamentaldomain-Zm-action}.  

Given the disc-type fundamental domain $\Omega$ of $\Gamma$, the quotient $\Gamma'\vcentcolon=\Gamma/\Z_m$ can be constructed by identifying certain pairs of consecutive segments in $\partial\Omega$ (i.\,e. the ones in the interior of $\Gamma$) which means that $\Gamma/\Z_m$ is still a topological disc. 
By assumption, $\Gamma$ intersects the singular locus of the corresponding group action 
(i.\,e. the vertical axis of rotation), exactly $N$ times: once for each horizontal disc. 
The Riemann--Hurwitz formula (see e.\,g. \cite[Chapter IV.3]{Freitag2011} and \cite[Remark B.2]{CarFraSch20}) implies 
\begin{align*}
\chi(\Gamma)&=m\chi(\Gamma')-(m-1)N. 
\end{align*}
Substituting $\chi(\Gamma)=2-2\gamma-\beta$ and $\chi(\Gamma')=1$ we obtain 
$2-2\gamma-\beta=m-(m-1)N$, or equivalently 
\begin{align*}
2\gamma&=(m-1)(N-1)-(\beta-1). 
\end{align*}
Since $\beta=1$ if $N$ is odd and $\beta=m$ if $N$ is even as shown above, the claim follows.  
\end{proof}

\paragraph{Building blocks of the construction.}
Let us now turn to the accurate geometric description of the initial surfaces.
Outside a neighborhood of $\Sp^2$
each initial surface will be a union
of $N$ exact horizontal discs.
To construct the portions of the initial surfaces
near $\Sp^2$
we find it convenient to work
with a variant of standard spherical coordinates
(which in particular make it simple
to enforce orthogonality
of the initial surfaces to $\Sp^2$).
Specifically, we define the map
\begin{equation}
\label{def:Phi}
\begin{aligned}
\Phi
\colon
\IntervaL{0,\tfrac{1}{3}}
  \times
  \R
  \times
  \IntervaL{-\tfrac{\pi}{4},\tfrac{\pi}{4}}
&\to
  \B^3
\\
(\sigma,\theta,\omega)
&\mapsto
(1-\sigma)\bigl(
  \cos \theta \, \cos \omega, ~
  \sin \theta \, \cos \omega, ~
  \sin \omega
  \bigr),
\end{aligned}
\end{equation}
a covering of a
(large) neighborhood of the equator,
where we denote Cartesian coordinates in the domain by
$(\sigma,\theta,\omega)$
and in the target by $(x,y,z)$.

We will define the initial surfaces
by first assembling certain surfaces with boundary
in the domain of $\Phi$.
It will be clear later that by taking the datum $m$
sufficiently large
(in terms of other pieces of data, to be explained)
we ensure that these last surfaces
are contained in the above box
on which $\Phi$ is injective with smooth inverse.
We then apply $\Phi$
and extend the image surfaces in $\B^3$
by further applying the symmetry group $\G$.
The initial surfaces are then completed
by attaching the $N$ exact horizontal discs
mentioned above.
We turn to the details.

Given an integer $m>0$, we define the domain
\begin{align}\label{eqn:Lambda}
\Lambda^\Phi_m
&\vcentcolon=
  \bigl\{ 
    (\sigma,\theta) \in \R^2
    \st
    0 \leq \sigma \leq \tfrac{1}{3}
    \mbox{ and }
    \abs{\theta} \leq \tfrac{\pi}{2m}
  \bigr\}.
\intertext{Given also $\tau>0$
(later to represent the waist radius
of an exact catenoid in the domain of $\Phi$),
we in turn define the two subdomains
$\Lambda^{\Phi,+}_{m,\tau},
 \Lambda^{\Phi,-}_{m,\tau}
 \subset
 \Lambda^\Phi_m$ by}
\notag
\Lambda^{\Phi,\pm}_{m,\tau}
&\vcentcolon=
    \Lambda^\Phi_m
  \setminus
  \bigl\{
    \sigma^2
    +(\theta \mp \tfrac{\pi}{2m})^2
    <
    \tau^2
  \bigr\},
\intertext{and given as well $\tau_+,\tau_->0$, we set (see Figure \ref{fig:subdomainsLambda}, left pair) }
\label{eqn:Lambdatau}
\Lambda^\Phi_{m,\tau_+,\tau_-}
&\vcentcolon=
  \Lambda^{\Phi,+}_{m,\tau_+}
  \cap
  \Lambda^{\Phi,-}_{m,\tau_-}.
\end{align}%
\begin{figure}
\pgfmathsetmacro{\mpar}{5} 
\hfill
\begin{tikzpicture}[line cap=round,line join=round,scale=1.5*\unitscale,baseline={(0,0)}]
\fill[black!33](0,-pi/2/\mpar)rectangle(1/3,pi/2/\mpar);	
\draw[->](0,0)--(0.4,0)node[right]{$\sigma$};	
\draw[->](0,0)--(0,1/2)node[below right,inner sep=0]{~$\theta$};	
\draw plot[plus](0,0)node[below left]{$0$};
\draw plot[vdash](1/3,0)node[below]{$\frac{1}{3}$};
\draw plot[hdash](0,pi/2/\mpar)node[left]{$\frac{\pi}{2m}$};
\path(1/6,0)node[above=1ex]{$\Lambda^\Phi_m$};
\end{tikzpicture}
\hfill
\pgfmathsetmacro{\tauplus}{0.125}
\pgfmathsetmacro{\tauminus}{0.15}
\begin{tikzpicture}[line cap=round,line join=round,scale=1.5*\unitscale,baseline={(0,0)}]
\fill[black!33](\tauminus,-pi/2/\mpar)arc(0:90:\tauminus)
--(0,pi/2/\mpar-\tauplus)arc(-90:0:\tauplus)
-|(1/3,0)|-cycle;
\draw[->](0,0)--(0.4,0)node[right]{$\sigma$};	
\draw[->](0,0)--(0,1/2)node[below right,inner sep=0]{~$\theta$};	
\draw plot[plus](0,0)node[below left]{$0$};
\draw plot[vdash](1/3,0)node[below]{$\frac{1}{3}$};
\draw plot[hdash](0,pi/2/\mpar)node[left]{$\frac{\pi}{2m}$};
\path(1/6,0)node[above=1ex]{$\Lambda^\Phi_{m,\tau_+,\tau_-}$};
\draw[latex-latex](0, pi/2/\mpar)--++(0:\tauplus) node[midway,below]{$\tau_+$};
\draw[latex-latex](0,-pi/2/\mpar)--++(0:\tauminus)node[midway,above]{$\tau_-$};
\end{tikzpicture}
\hfill
\pgfmathsetmacro{\phiO}{120} 
\pgfmathsetmacro{\hpar}{0.27} 
\tdplotsetmaincoords{72}{\phiO} 
\begin{tikzpicture}[scale=\unitscale,line cap=round,line join=round,tdplot_main_coords,baseline={(0,0,0)}]
\shade[tdplot_screen_coords,ball color=white] (0,0) circle (1);
\tdplotdrawarc[black!50]{(0,0,0)}{1}{\phiO}{\phiO+180}{}{} 
\draw[tdplot_screen_coords] (0,0) circle (1);
\draw[black!50](0,0,0)--(1,0,0)(0,0,0)--(0,1,0)(0,0,0)--(0,0,1);
\draw[->](1,0,0)--(1.2,0,0);
\draw[->](0,1,0)--(0,1.2,0);
\draw[->](0,0,1)--(0,0,1.1);
\tdplotdrawarc[]{(0,0,0)}{1}{\phiO}{\phiO-180}{}{}
\begin{scope}[dotted]
\tdplotdrawarc{(0,0,\hpar)}{{sqrt(4/9-\hpar^2)}}{90/\mpar}{360-90/\mpar}{}{} 
\tdplotdrawarc{(0,0,\hpar)}{{sqrt(1-\hpar^2)}}{90/\mpar}{360-90/\mpar}{}{} 
\draw(0,0,\hpar)--({sqrt(4/9-\hpar^2)*cos(90/\mpar)},{sqrt(4/9-\hpar^2)*sin(90/\mpar)},\hpar);
\draw(0,0,\hpar)--({sqrt(4/9-\hpar^2)*cos(-90/\mpar)},{sqrt(4/9-\hpar^2)*sin(-90/\mpar)},\hpar);
\end{scope}
\fill[black!90]({sqrt(1-\hpar^2)*cos(-90/\mpar)},{sqrt(1-\hpar^2)*sin(-90/\mpar)},\hpar)
arc(-90/\mpar:90/\mpar:{sqrt(1-\hpar^2)})
--({sqrt(4/9-\hpar^2)*cos(90/\mpar)},{sqrt(4/9-\hpar^2)*sin(90/\mpar)},\hpar)
arc(90/\mpar:-90/\mpar:{sqrt(4/9-\hpar^2)})--cycle
;
\draw[latex-latex,tdplot_screen_coords](0,0)--++(0,\hpar)node[midway,right]{$h^B$};
\tdplotdrawarc[latex-latex]{(0,0,0)}{1}{0}{90/\mpar}{below}{$\frac{\pi}{2m}$} 
\begin{scope}[white,very thin]
\foreach\spar in {0,2,...,10}{
\tdplotdrawarc{(0,0,\hpar)}{{sqrt((1-\spar/30)^2-\hpar*\hpar)}}{-90/\mpar}{90/\mpar}{}{} 
} 
\foreach\tpar in {-90,-70,...,90}{
\draw({sqrt(4/9-\hpar^2)*cos(\tpar/\mpar)},{sqrt(4/9-\hpar^2)*sin(\tpar/\mpar)},\hpar)--({sqrt(1-\hpar^2)*cos(\tpar/\mpar)},{sqrt(1-\hpar^2)*sin(\tpar/\mpar)},\hpar);
} 
\end{scope}
\draw[black!90]({sqrt(1-\hpar^2)*cos(-90/\mpar)},{sqrt(1-\hpar^2)*sin(-90/\mpar)},\hpar)
arc(-90/\mpar:90/\mpar:{sqrt(1-\hpar^2)})
--({sqrt(4/9-\hpar^2)*cos(90/\mpar)},{sqrt(4/9-\hpar^2)*sin(90/\mpar)},\hpar)
arc(90/\mpar:-90/\mpar:{sqrt(4/9-\hpar^2)})--cycle
;
\end{tikzpicture}
\caption{The  domains defined in \eqref{eqn:Lambda}--\eqref{eqn:Lambdatau} and the image of $\Phi\bigl(\sigma,\theta,\omega^B_{m,h^B}(\sigma,\theta)\bigr)$ defined on $\Lambda^\Phi_m$.}%
\label{fig:subdomainsLambda}%
\end{figure}%
Next suppose we are given additionally
$h^B\in\interval{-2/3,2/3}$
to represent the height of the center
of a disc layer in the initial surface
 -- frequently 
we will attach the label $B$
to objects associated  with disc layers
(approximate copies of $\B^2$)
and the label $K$
to objects
associated with catenoidal ribbons
(images under $\Phi$ of subsets
of exact catenoids).
We define the function 
$
 \omega^B_{m,h^B}
 \colon
 \Lambda^\Phi_m
 \to
 \R
$
by
\begin{equation}
\label{eqn:disc_def_fun_Phi}
\omega^B_{m,h^B}(\sigma,\theta)
  \vcentcolon=
  \arcsin \frac{h^B}{1-\sigma},
\end{equation}
whose graph has image under $\Phi$
a neighborhood of a circumferential segment
in the horizontal disc in $\B^3$ at height $h^B$ (see Figure \ref{fig:subdomainsLambda}, right image). 
Given also
$h^K \in \R$
(to represent the height of the center
of an exact catenoid in the domain of $\Phi$)
in the initial surface)
and defining $\sgn \colon \R \to \R$ by
\begin{equation*}
\sgn x
\vcentcolon=
\begin{cases}
\hphantom{-}1 &\mbox{if} \quad x \geq 0
\\
-1 &\mbox{if} \quad x<0,
\end{cases}
\end{equation*}
we further define
the two functions
$
 \omega^{K,\pm}_{m,\tau,h^K,h^B}
 \colon
 \Lambda^{\Phi,\pm}_{m,\tau}
 \to
 \R
$
by
\begin{equation}
\label{eqn:cat_def_fun_Phi}
\omega^{K,\pm}_{m,\tau,h^K,h^B}(\sigma,\theta)
\vcentcolon=
h^K + \sgn (h^B-h^K) \, \tau
  \arcosh
  \frac{
    \sqrt{
      \sigma^2
      +\left(
        \theta \mp \frac{\pi}{2m}
      \right)^2
    }
  }
  {\tau},
\end{equation}
whose graph is a subset of the exact catenoid
(in the domain of $\Phi$)
with center
$\bigl(0, \pm \pi/(2m), h^K\bigr)$
and waist radius $\tau$.

\begin{figure}%
\pgfmathsetmacro{\xscalepar}{6}
\pgfmathsetmacro{\mpar}{10}
\begin{tikzpicture}[semithick,line cap=round,line join=round,scale=3,xscale=\xscalepar,baseline={(0,0)},declare function={
 psithetaplus(\t)=cutoff( abs( \t-pi/(2*\mpar) )*4*\mpar/pi );
psithetaminus(\t)=cutoff( abs( \t+pi/(2*\mpar) )*4*\mpar/pi );
}] 
\pgfmathsetmacro{\tmax}{pi/2/\mpar}
\draw[->](-\tmax-0.15/\xscalepar,0)--(\tmax+0.15/\xscalepar,0)node[above]{$\theta$};
\draw[->](0,-0)--(0,1.15);
\draw plot[plus] (0,0)node[below]{$0$};
\draw plot[hdash](0,1)node[left]{$1$};
\draw plot[vdash] (pi/2/\mpar,0)node[below]{$\frac{\pi}{2m}$};
\draw plot[vdash] (pi/4/\mpar,0)node[below]{$\frac{\pi}{4m}$};
\draw plot[vdash](-pi/2/\mpar,0)node[below]{$\mathllap-\frac{\pi}{2m}$};
\draw plot[vdash](-pi/4/\mpar,0)node[below]{$\mathllap-\frac{\pi}{4m}$};
\draw[dotted]
 (pi/2/\mpar,0)--++(0,1)
 (pi/4/\mpar,0)--++(0,1)
(-pi/2/\mpar,0)--++(0,1)
(-pi/4/\mpar,0)--++(0,1)
;
\draw[Orange,thick,domain=-\tmax:0,samples=200,smooth,variable=\x]plot({\x},{psithetaminus(\x)})--(\tmax,0)
(-\tmax,1)node[above]{$\psi^-_m$};
\draw[Blau,thick,domain=0:\tmax,samples=200,smooth,variable=\x]
(-\tmax,0)--
plot({\x},{psithetaplus(\x)})
(\tmax,1)node[above]{$\psi^+_m$}; 
\draw[dashed,xscale=1/\xscalepar,<-](0,0,2/3)--(0,0,0)node[pos=0,right=1pt]{$\sigma$};
\end{tikzpicture}
\hfill
\begin{tikzpicture}[semithick,line cap=round,line join=round,scale=3,xscale=\xscalepar,baseline={(0,0)},declare function={
psisigma(\s)=cutoff(\mpar*\s);
}] 
\draw[->](0,0)--(1/3+0.15/\xscalepar,0)node[above]{$\sigma$};
\draw[->](0,-0)--(0,1.15);
\draw plot[plus] (0,0)node[below]{$0$};
\draw plot[hdash](0,1)node[left]{$1$};
\draw plot[vdash] (1/\mpar,0)node[below]{$\frac{1}{m}$};
\draw plot[vdash] (2/\mpar,0)node[below]{$\frac{2}{m}$};
\draw plot[vdash] (1/3,0)node[below]{$\frac{1}{3}$}; 
\draw[dotted](1/\mpar,0)--++(0,1); 
\draw[thick,domain=0:1/3,samples=200,smooth,variable=\x]plot({\x},{psisigma(\x)})(0,1)
node[above right]{~$\psi^\sigma_m$}; 
\draw[dashed,xscale=1/\xscalepar,<-](0,0,-2/3)--(0,0,2/3)node[pos=0,right]{$\theta$};
\end{tikzpicture}
\caption{Profiles of the cutoff functions $\psi^\pm_m,\psi^\sigma_m\colon\Lambda^\Phi_m\to\R$ defined in \eqref{eqn:cat_cutoffs}.}%
\label{fig:cat_cutoffs}%
\end{figure}

With the aim of gluing together the preceding functions
we fix a smooth function $\Psi \colon \R \to \R$
that is constantly $1$ on $\{x \leq 1\}$
and constantly $0$ on $\{x \geq 2\}$,
and we introduce in turn the three cutoff functions
$
 \psi^\pm_m, \psi^\sigma_m
 \colon
 \Lambda^\Phi_m
 \to \R
$
given by
\begin{align}
\label{eqn:cat_cutoffs}
\psi^\pm_m(\sigma,\theta)
&\vcentcolon=
\Psi\Bigl(
    \tfrac{4m}{\pi}
    \abs[\big]{
      \theta \mp \tfrac{\pi}{2m}
    }
  \Bigr),
&
\psi^\sigma_m(\sigma,\theta)
&\vcentcolon=
\Psi(m\sigma)
\end{align}
as shown in Figure \ref{fig:cat_cutoffs}. 
Given as well $h^K_+,h^K_- \in \R$ we define the function
$
 \omega_{m,h^B,\tau_+,h^K_+}
 \colon
 \Lambda^{\Phi,+}_{m,\tau_+}
 \to
 \R
$
by
\begin{align}
\label{angular_height_function_one_bridge}
\omega_{m,h^B,\tau_+,h^K_+}
&\vcentcolon=
\psi^\sigma_m\psi^+_m \omega^{K,+}_{m,\tau_+,h^K_+,h^B}
  +\psi^\sigma_m(1-\psi^+_m)h^B
  +(1-\psi^\sigma_m)\omega^B_{m,h^B}
\intertext{and the function
$
 \omega_{m, h^B,\tau_+,h^K_+,\tau_-,h^K_-}
 \colon
 \Lambda^\Phi_{m,\tau_+,\tau_-}
 \to
 \R
$
by}
\label{angular_height_function_two_bridges}
\omega_{m, h^B,\tau_+,h^K_+,\tau_-,h^K_-}
&\vcentcolon=
\psi^\sigma_m\psi^+_m \omega^{K,+}_{m,\tau_+,h^K_+,h^B}
  +\psi^\sigma_m(1-\psi^+_m-\psi^-_m)h^B
  +(1-\psi^\sigma_m)\omega^B_{m,h^B}
\\\notag
&\hphantom{{}\vcentcolon={}}
  +\psi^\sigma_m\psi^-_m \omega^{K,-}_{m,\tau_-,h^K_-,h^B}
\end{align}
as plotted in Figure \ref{fig:graphGamma}. 
We denote the graphs of these last two functions by
\begin{equation*}
\begin{aligned}
\Gamma_{m,h^B,\tau_+,h^K_+}
  &\vcentcolon=
  \left\{
    \left(
      \sigma,\theta,
      \omega_{m,h^B,\tau_+,h^K_+}
        (\sigma,\theta)
    \right)
    \st
    (\sigma,\theta) \in \Lambda^{\Phi,+}_{m,\tau_+}
  \right\},
\\
\Gamma_{m,h^B,\tau_+,h^K_+,\tau_-,h^K_-}
  &\vcentcolon=
  \left\{
    \left(
      \sigma,\theta,
      \omega_{m,h^B,\tau_+,h^K_+,\tau_-,h^K_-}
        (\sigma,\theta)
    \right)
    \st
    (\sigma,\theta) \in \Lambda^\Phi_{m,\tau_+,\tau_-}
  \right\}
\end{aligned}
\end{equation*}%
\begin{figure}%
\pgfmathsetmacro{\globalscale}{15}
\pgfmathsetmacro{\mpar}{8}
\pgfmathsetmacro{\taupar}{0.05}
\pgfmathsetmacro{\hB}{0.2}
\pgfmathsetmacro{\hKplus}{0.35} 
\pgfmathsetmacro{\hKminus}{0.1} 
\pgfmathsetmacro{\thetaO}{72.0000}
\pgfmathsetmacro{\phiO}{-30.0000}
\tdplotsetmaincoords{\thetaO}{\phiO} 
\tdplottransformmainscreen{1}{0}{0}
\pgfmathsetmacro{\VecHx}{\tdplotresx}
\pgfmathsetmacro{\VecHy}{\tdplotresy}
\tdplottransformmainscreen{0}{1}{0}
\pgfmathsetmacro{\VecVx}{\tdplotresx}
\pgfmathsetmacro{\VecVy}{\tdplotresy}
\begin{tikzpicture}[scale=\globalscale,tdplot_main_coords,line cap=round,line join=round,baseline={(0,0,0)}]
\fill[black!30](0,pi/2/\mpar-\taupar,0)arc(-90:0:\taupar)
--(1/3, pi/2/\mpar,0)
--(1/3,-pi/2/\mpar,0)
--(0,-pi/2/\mpar,0) 
--cycle;
\draw(0,0,0)node[scale=\globalscale]{\includegraphics[page=3]{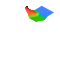}};
\pgfresetboundingbox
\path(0,pi/2/\mpar-\taupar,0)arc(-90:0:\taupar)--(1/3, pi/2/\mpar,0)--(1/3,-pi/2/\mpar,0)--(0,-pi/2/\mpar,0)--cycle;
\draw[->](0,0,0)--(1/3+0.1,0,0)node[below]{$\sigma$};
\draw[->](0,0,0)--(0,pi/2/\mpar+0.05,0)node[below]{$\theta$};
\draw[->](0,0,0)--(0,0,0.45)node[below right,inner sep=0]{~$\omega_{m,h^B,\tau_+,h^K_+}$};
\begin{scope}[semithick]
\draw[latex-latex](0,-0,0)--++(0,0,\hB)node[pos=0.6,right,inner sep=1pt]{$h^B$};
\draw[latex-latex](0,-\taupar+pi/2/\mpar,0)--++(0,0,\hKplus)node[pos=0.33,left,inner sep=1pt]{$h^K_+$};
\end{scope}
\begin{scope}[very thin]
\draw[latex-latex](0,-\taupar+pi/2/\mpar,0)
--++(0,\taupar,0)node[pos=0.5,below]{$\tau_+$};
\draw[densely dotted](0,pi/2/\mpar,0)--++(\taupar,0,0);
\end{scope}
\draw(1/6,-pi/4/\mpar,0)node[scale=2,cm={\VecHx,\VecHy,\VecVx,\VecVy,(0,0)}]{$\Lambda^{\Phi,+}_{m,\tau_+}$};
\end{tikzpicture}
\hfill
\begin{tikzpicture}[scale=\globalscale,tdplot_main_coords,line cap=round,line join=round,baseline={(0,0,0)}]
\fill[black!30](0,pi/2/\mpar-\taupar,0)arc(-90:0:\taupar)
--(1/3, pi/2/\mpar,0)
--(1/3,-pi/2/\mpar,0)
--(\taupar,-pi/2/\mpar,0)arc(0:90:\taupar)
--cycle;
\draw(0,0,0)node[scale=\globalscale]{\includegraphics[page=4]{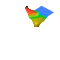}};
\pgfresetboundingbox
\path(0,pi/2/\mpar-\taupar,0)arc(-90:0:\taupar)--(1/3, pi/2/\mpar,0)--(1/3,-pi/2/\mpar,0)--(\taupar,-pi/2/\mpar,0)arc(0:90:\taupar)--cycle;
\draw[->](0,0,0)--(1/3+0.1,0,0)node[below]{$\sigma$};
\draw[->](0,0,0)--(0,pi/2/\mpar+0.05,0)node[below]{$\theta$};
\draw[->](0,0,0)--(0,0,0.45)node[below right,inner sep=0]{~$\omega_{m,h^B,\tau_+,h^K_+,\tau_-,h^K_-}$};
\begin{scope}[semithick]
\draw[latex-latex](0,-0,0)--++(0,0,\hB)node[pos=0.6,right,inner sep=1pt]{$h^B$};
\draw[latex-latex](0,-\taupar+pi/2/\mpar,0)--++(0,0,\hKplus)node[pos=0.33,left,inner sep=1pt]{$h^K_+$};
\draw[latex-latex](0, \taupar-pi/2/\mpar,0)--++(0,0,\hKminus)node[pos=0.33,right,inner sep=1pt]{$h^K_-$};
\end{scope}
\begin{scope}[very thin]
\draw[latex-latex](0,-\taupar+pi/2/\mpar,0)
--++(0,\taupar,0)node[pos=0.5,below]{$\tau_+$};
\draw[latex-latex](0,+\taupar-pi/2/\mpar,0)
--++(0,-\taupar,0)node[pos=0.5,below]{$\tau_-$};
\draw[densely dotted](0,pi/2/\mpar,0)--++(\taupar,0,0)(0,-pi/2/\mpar,0)--++(\taupar,0,0);
\end{scope}
\draw(0.2,-pi/4/\mpar,0)node[scale=2,cm={\VecHx,\VecHy,\VecVx,\VecVy,(0,0)}]{$\Lambda^\Phi_{m,\tau_+,\tau_-}$};
\end{tikzpicture}
\caption{The graphs $\Gamma_{m,h^B,\tau_+,h^K_+}$ (left) and $\Gamma_{m,h^B,\tau_+,h^K_+,\tau_-,h^K_-}$ (right). 
Roughly speaking, 
the red, green, respectively blue regions are shaped by the first, second respectively third summand of \eqref{angular_height_function_one_bridge} and \eqref{angular_height_function_two_bridges}; 
the yellow region is shaped by the fourth summand of \eqref{angular_height_function_two_bridges}.
}%
\label{fig:graphGamma}%
\bigskip
\bigskip
\pgfmathsetmacro{\thetaO}{72}
\pgfmathsetmacro{\phiO}{90+1*45/2} 
\tdplotsetmaincoords{\thetaO}{\phiO} 
\begin{tikzpicture}[tdplot_main_coords,line cap=round,line join=round,scale=\unitscale,baseline={(0,0,0)}]
\draw(0,0,0)node[inner sep=0,scale={\unitscale/4.015}](S){\includegraphics[page=5]{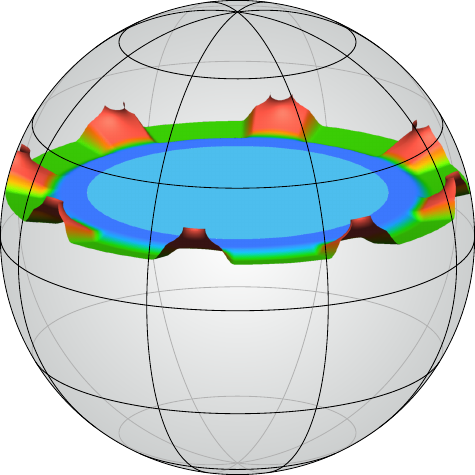}};
\end{tikzpicture}
\hfill
\begin{tikzpicture}[tdplot_main_coords,line cap=round,line join=round,scale=\unitscale,baseline={(0,0,0)}]
\draw(0,0,0)node[inner sep=0,scale={\unitscale/4.015}](S){\includegraphics[page=6]{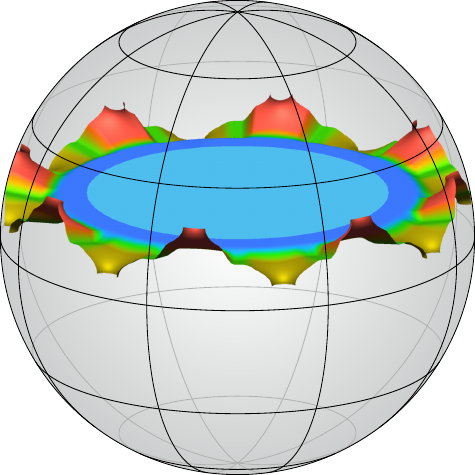}};
\end{tikzpicture}
\caption{The sets $D_{m,h^B,\tau_+,h^K_+}$ (left) and $D_{m,h^B,\tau_+,h^K_+,\tau_-,h^K_-}$ (right) defined in \eqref{building_blocks} with the same choice of parameters as in Figure \ref{fig:graphGamma}; 
appropriate heights $h^B$, $h^K_\pm$ and waist radii $\tau_\pm$ still need to be chosen (cf. Figure \ref{fig:Initialsurface}). 
}
\label{fig:DinBall}%
\end{figure}%
Finally, we map these regions into $\B^3$ using $\Phi$,
extend them by applying the pyramidal symmetries of the construction,
and adjoin the discs
\begin{equation*}
B_{h^B}\vcentcolon=\bigl\{(x,y,h^B) \st x^2+y^2\leq\tfrac{4}{9}-(h^B)^2\bigr\}
\end{equation*}
at height $h^B$ with radius 
$(1-1/3)\cos\omega^B_{m,h^B}(1/3,0)=(2/3)\sqrt{1-(3h^B/2)^2}$, 
to obtain the sets
\begin{equation}
\label{building_blocks}
\begin{aligned}
D_{m,h^B,\tau_+,h^K_+}
  &\vcentcolon=
  B_{h^B}
  \cup
  \pyr_m
  \Phi
  \left(
    \Gamma_{m,h^B,\tau_+,h^K_+}
  \right),
\\
D_{m,h^B,\tau_+,h^K_+,\tau_-,h^K_-}
  &\vcentcolon=
  B_{h^B}
  \cup
  \pyr_m
  \Phi
  \left(
    \Gamma_{m,h^B,\tau_+,h^K_+,\tau_-,h^K_-}
  \right)
\end{aligned}
\end{equation}
visualized in Figure \ref{fig:DinBall}. 
Note that all of the above regions are invariant
under rotation through angle $2\pi/m$
about the $z$-axis (in either direction)
and that
\begin{equation*}
D_{m,h^B,\tau_+,h^K_+,\tau_-,h^K_-}
=
\rot_{\widehat{z}}^{\pi/m}
D_{m,h^B,\tau_-,h^K_-,\tau_+,h^K_+},
\end{equation*}
where $\rot_{\widehat{z}}^{\pi/m} \in \Ogroup(3)$
is rotation about the positively directed $z$-axis
through angle $\pi/m$ in the conventional positive sense.

\paragraph{Specifications of the initial surfaces.}
Each initial surface will be a union of $N$
surfaces derived from \eqref{building_blocks}, two of the first type (to form the top
and the bottom of the corresponding
initial surface),
and $N-2$ of the second type (in between),
suitably aligned and
with appropriately selected parameters.
To make the definition we must specify
$N$ disc heights and $N-1$ waist heights
\begin{equation*}
-1 \ll
h^B_1 < h^K_1 < h^B_2 < h^K_2 < \cdots < h^K_{N-1} < h^B_N
\ll 1,
\end{equation*}
as well as $N-1$ waist radii
$0<\tau_1, \tau_2, \ldots, \tau_{N-1} \ll 1$
associated with the corresponding
$h^K_i$.
Once all these quantities
have been specified, each initial surface
will be defined as the union of the
$N$ surfaces in \eqref{layers_with_half_cats}; let us proceed step by step towards that definition.

The symmetry group $\G$ we impose
subjects these quantities
to constraints
that we express with the aid of the following notation
(which will be useful for related purposes
throughout the construction).
For any integer $d \geq 0$
we define the linear subspaces
$
 \R^d_{\updownarrow+},
 \R^d_{\updownarrow-}
 \subseteq
 \R^d
$
by
\begin{equation*}
\R^d_{\updownarrow\pm}
\vcentcolon=
\{
  \vec{v} \in \R^d
  \st
    v_i
    =
    \pm
    v_{d+1-i}
    \text{ for each }
    i=1,\ldots,d
\}.
\end{equation*}
For any $S \subset \R^d$
(not necessarily a subspace)
we in turn define
the two sets
\(
S_{\updownarrow \pm}
\vcentcolon=
S \cap \R^d_{\updownarrow \pm}.
\) 
Note that
\begin{align*}
\dim \R^d_{\updownarrow+}
&=
  \ceil{d/2},
&
\dim \R^d_{\updownarrow-}
&=
  \floor{d/2},
&
\R^d
&=
  \R^d_{\updownarrow+}
    \oplus
    \R^d_{\updownarrow-}.
\end{align*}
In enforcing $\G$-invariance
of the initial surfaces
we thus accordingly require
\begin{equation*}
\vec{h}^B \in \R^N_{\updownarrow-},
\quad
\vec{h}^K \in \R^{N-1}_{\updownarrow-},
\quad
\vec{\tau} \in \R^{N-1}_{\updownarrow+},
\end{equation*}
where we have bundled the above
heights and radii into vectors in the obvious way.
Rather than prescribing these quantities directly
we will instead specify them in terms of
data
\begin{equation*}
\vec{\zeta} \in \R^{N-1}_{\updownarrow+},
\quad
\vec{\xi} \in \R^{N-1}_{\updownarrow-}.
\end{equation*}
Each initial surface will then be uniquely determined
by a choice of data $(N,m,\vec{\zeta},\vec{\xi}\,)$.

Conceptually the specification
of $(\vec{\tau},\vec{h}^K,\vec{h}^B)$
in terms of $(N,m,\vec{\zeta},\vec{\xi}\,)$
is made in stages,
which we summarize with the chain of dependence
\begin{align*}
(N,m) &\leadsto \vec{\taubar},
&
(m,\vec{\taubar},\vec{\zeta}\,)
&\leadsto \vec{\tau},
&
(N,m,\vec{\tau},\vec{\xi}\,) &\leadsto (\vec{h}^K,\vec{h}^B).
\end{align*}
In more detail, we first take
\(
\vec{x}=(x_1,\ldots,x_{N-1}) \in \R^{N-1}_{\updownarrow+}
\)
as in Lemma \ref{lem:limiting_waist_ratios} 
and we then define
$\vec{\taubar} \in \R^{N-1}_{\updownarrow+}$
by
\begin{equation}
\label{eqn:taubar_def}
\taubar_i=\taubar_{i;N,m}
=
\frac{x_i}{m}
  \exp
  \left(
    \frac{x_{n-1}-1}
    {1 + N \bmod 2}
    \cdot
    \frac{m}{2}
  \right),
\end{equation}
where we understand $x_0=0$ (in case $n=1$).
We emphasize that $x_{n-1}<1$,
so that the argument of the exponential
is strictly negative.
The above choice of $\vec{\taubar}$
is made to ensure that the initial surface
is approximately balanced in a certain sense
(see Appendix \ref{app:forces} for full details);
this geometric condition in turn implies,
roughly speaking,
that a certain component of the projection
of the initial mean curvature onto
the cokernel is not too large.

Next, to help achieve finer control
over the cokernel,
we define
$\vec{\tau} \in \R^{N-1}_{\updownarrow+}$
by
\begin{equation}
\label{eqn:tau_def}
\tau_i=\tau_{i;N,m,\vec{\zeta}}
\vcentcolon=
\taubar_i
\exp
  \Bigl(
    \zeta_n
    +\frac{\zeta_i}{m}
  \Bigr).
\end{equation}
Thus variation of $\vec{\zeta}$
adjusts the waist radii.
We pause to point out that
for each integer $N \geq 2$,
for each $i=1,\ldots,N$,
for each
$\vec{\zeta} \in \R^{N-1}_{\updownarrow+}$,
and for each $q>0$
we have, for some $C(N)>0$,
\begin{align}
\label{eqn:tau_size}
\lim_{m \to \infty}
\frac{\tau_i}{\tau_1}
&\leq
C(N),
&
\lim_{m \to \infty} m^q \tau_1 &= 0.
\end{align}
Finally,
we specify the heights $\vec{h}^K, \vec{h}^B$.
Conceptually,
the data $(N,m,\tau)$ determine
heights $\vec{\underline{h}}{}^K$
and $\vec{\underline{h}}{}^B$
by matching conditions
(essentially by the requirement that the spacing
between disc layers should
agree with the vertical extent
of each intermediate catenoidal ribbon,
which is set,
via the definitions leading up to
\eqref{building_blocks},
by $m$ and
the corresponding $\tau$).
To control the cokernel, however,
it is necessary to be able to adjust
the heights of the centers of the catenoids
via the $\vec{\xi}$ parameters,
affecting also $\vec{h}^B$
through modified matching conditions
that average the contributions 
on a given $h^B$ from neighboring catenoids.
Since we will never need to refer
to either
$\vec{\underline{h}}{}^K$
or $\vec{\underline{h}}{}^B$
directly,
we do not distinguish them explicitly
in the following specification.

To be precise, we define the vector 
$\vec{\delta h}^K \in \R^{N-1}_{\updownarrow+}$
whose entries 
\begin{equation}
\label{eqn:cat_extent}
(\delta h^K)_i
=
(\delta h^K)_{i;N,m,\vec{\zeta}}
\vcentcolon=
2\tau_i \arcosh \frac{1}{m\tau_i}
\end{equation}
determine the vertical extent 
of the catenoidal ribbon at height $h^K_i$,
and
$
 \vec{\disloc}
 =
 \vec{\disloc}_{N,m,\vec{\zeta},\vec{\xi}}
 \in
 \R^N_{\updownarrow+}
$
by
\begin{equation}
\label{eqn:disloc_def}
\disloc_i=\disloc_{i;N,m,\vec{\zeta},\vec{\xi}}
\vcentcolon=
\begin{cases}
0
  &\mbox{for } i \in \{1,N\}
\\
\frac{1}{2}(\xi_i - \xi_{i-1})\tau_n
  &\mbox{for } 1<i<N,
\end{cases}
\end{equation}
which we call the dislocation
at disc layer $i$.
For future convenience
we define in $\R^d$,
for any integer $d \geq 2$,
the three subspaces
\begin{align*}
\R^{d,0,0}
&\vcentcolon=
 \{
    \vec{v} \in \R^d
    \st
    v_1=v_d=0
  \},
&
\R^{d,0,0}_{\updownarrow\pm}
&\vcentcolon=
  \R^{d,0,0} \cap \R^d_{\updownarrow\pm}.
\end{align*}
In particular we then have
$\vec{\disloc} \in \R^{N,0,0}_{\updownarrow+}$.
Now $\vec{h}^K$ and $\vec{h}^B$
are uniquely determined
by the symmetry constraints
$\vec{h}^K \in \R^{N-1}_{\updownarrow-}$,
$\vec{h}^B \in \R^N_{\updownarrow-}$
and the matching conditions
\begin{equation}
\label{eqn:dislocated_matching}
\begin{aligned}
h^K_i - h^B_i
  &=
  \tfrac{1}{2}(\delta h^K)_i
  +\disloc_i
&\quad
  (1 \leq i &\leq N-1),
\\
h^B_i - h^K_{i-1}
  &=
  \tfrac{1}{2}(\delta h^K)_{i-1}
  +\disloc_i
&
  (2 \leq i&\leq N).
\end{aligned}
\end{equation}
Indeed, the symmetry constraints
force
\begin{equation*}
N \mbox{ even} \Rightarrow h^K_n = 0,
\qquad
N \mbox{ odd} \Rightarrow h^B_{n+1} = 0,
\end{equation*}
anchoring the configurations,
and then (in both cases)
the remaining $h^K_i$ and $h^B_i$
are uniquely determined
from these anchors and the already specified
$(\delta h^K)_i$ and $\disloc_i$
via \eqref{eqn:dislocated_matching};
the required symmetries then follow from those
already enforced for the latter quantities.
Before continuing
we observe,
for repeated future use,
that
for each integer $N \geq 2$,
for each integer $i=1,\ldots,N-1$,
for each
$\vec{\zeta} \in \R^{N-1}_{\updownarrow+}$,
and for each
$\vec{\xi} \in \R^{N-1}_{\updownarrow-}$
we have
\begin{equation}
\label{eqn:height_bound}
\lim_{m \to \infty}
  \frac{\abs{h^B_i}+\abs{h^K_i}+\abs{h^B_{i+1}}}
      {m\tau_1}
\leq
C(N)
\end{equation}
for some $C(N)>0$,
having made use of
\eqref{eqn:taubar_def},
\eqref{eqn:tau_def},
\eqref{eqn:tau_size},
\eqref{eqn:cat_extent},
\eqref{eqn:disloc_def},
and \eqref{eqn:dislocated_matching},
and assuming
$m > \abs{\vec{\zeta}}+\abs{\vec{\xi}}$.

\begin{figure}
\begin{tikzpicture}[scale=\unitscale,baseline={(0,0)}]
\draw(0,0)node[inner sep=0,scale={\unitscale/4.015}](S0){\includegraphics[page=7]{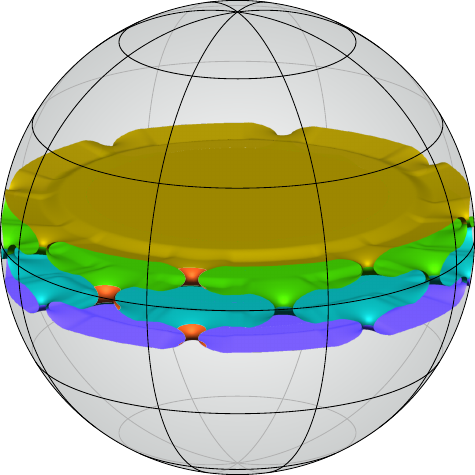}};
\draw[orange]
(-0.05,-0.13-0.3)node{$K_1$}
(-0.05, 0.13-0.3)node{$K_3$}
(-0.4, 0.00-0.3)node{$K_2$};
\pgfmathsetmacro{\hpar}{0.08}
\draw[color={rgb,255:red,158; green,136; blue,0  }](1, 3*\hpar-0.03)node[right]{$D_4$};
\draw[color={rgb,255:red,59 ; green,183; blue,0  }](1,   \hpar-0.03)node[right]{$D_3$};
\draw[color={rgb,255:red,6  ; green,167; blue,166}](1,  -\hpar-0.03)node[right]{$D_2$};
\draw[color={rgb,255:red,109; green,84 ; blue,254}](1,-3*\hpar-0.03)node[right]{$D_1$};
\end{tikzpicture}
\hfill
\begin{tikzpicture}[scale=\unitscale,baseline={(0,0)}]
\draw(0,0)node[inner sep=0,scale={\unitscale/4.015}](S1){\includegraphics[page=8]{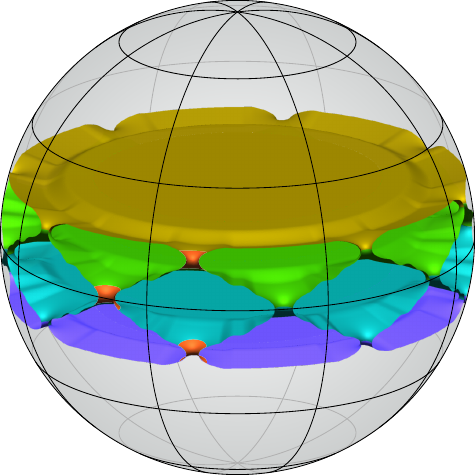}};
\end{tikzpicture}
\caption{Visualisation of the initial surface $\Sigma_{N,m,\vec{\zeta},\vec{\xi}}$ 
for $N=4$, $m=8$, $\vec{\zeta}=(0,0,0)$ and 
$\vec{\xi}=(0,0,0)$ (left image) respectively 
$\vec{\xi}=(-2,0,2)$ (right image). 
Highlighted are the regions $K_1,K_2,K_3$ defined in \eqref{eqn:catr_init_def} and the regions $D_1,\ldots,D_4$ defined in \eqref{layers_with_half_cats}. 
}
\label{fig:Initialsurface}%
\end{figure}

\paragraph{Assembly of the initial surfaces.}
Referring to \eqref{building_blocks} and the specifications of the preceding subsection,
we define the $N$ regions
\begin{align}
\notag
D_1=D_{1;N,m,\vec{\zeta},\vec{\xi}}
  &\vcentcolon=
  D_{m,h^B_1,\tau^K_1,h^K_1},
\\
\label{layers_with_half_cats}
D_i=D_{i;N,m,\vec{\zeta},\vec{\xi}}
  &\vcentcolon=
  \rot_{\widehat{z}}^{(i-1)\pi/m}
  D_{m,h^B_i,\tau^K_i,h^K_i,\tau^K_{i-1},h^K_{i-1}}
  \quad \mbox{for } 2 \leq i \leq N-1,
\\\notag
D_N=D_{N;N,m,\vec{\zeta},\vec{\xi}}
  &\vcentcolon=
  \rot_{\widehat{z}}^{N\pi/m}
  D_{m,h^B_N,\tau^K_{N-1},h^K_{N-1}}.
\intertext{The initial surfaces are defined as the unions of the above regions
as visualized in Figure \ref{fig:Initialsurface}:}
\label{eqn:init_surf_def}
\Sigma=\Sigma_{N,m,\vec{\zeta},\vec{\xi}}
&\vcentcolon=
\bigcup_{i=1}^N
  D_{i;N,m,\vec{\zeta},\vec{\xi}}.
\end{align}
We write $\met{\Sigma}$
for the induced metric on $\Sigma$,
$\eta^\Sigma$ for its outward unit conormal
on $\partial_\Sigma$,
and $\twoff{\Sigma}$ and $H^\Sigma$
for the scalar projections
of, respectively,
the vector-valued
second fundamental form and mean curvature
of $\Sigma$
onto the unit normal which points upward
at the point $\{x=y=0\} \cap D_1$.

\paragraph{Estimates of the initial geometry.}

For each $i=1,\ldots, N-1$ we set
\begin{equation*}
a_i\vcentcolon=\arcosh \frac{1}{2m\tau_i}
\end{equation*}
and define the map $\kappa_i=\kappa_{i; N, m, \vec{\zeta}, \vec{\xi}}$ by 
\begin{equation}
\label{def:kappa}
\begin{aligned}
\kappa_i
  \colon
  \IntervaL{-a_i, a_i} \times \IntervaL{-\tfrac{\pi}{2},\tfrac{\pi}{2}}
  &\to
  \Sigma=\Sigma_{N, m, \vec{\zeta}, \vec{\xi}}
\\
(t,\vartheta)
  &\mapsto
  \Phi
  \left(
        \tau_i \cosh t \cos \vartheta,~
        (-1)^{i-1}\tfrac{\pi}{2m}+\tau_i \cosh t \sin \vartheta,~
        h^K_i+\tau_i t
  \right),
\end{aligned}
\end{equation}
whose image we denote by
\begin{equation*}
K_i
=
K_{i; N, m, \vec{\zeta}, \vec{\xi}}
\vcentcolon=
\kappa_i\bigl(\IntervaL{-a_i,a_i} \times \IntervaL{-\tfrac{\pi}{2},\tfrac{\pi}{2}}\bigr)
\subset
\Sigma
\end{equation*}
as visualized in Figure \ref{fig:Initialsurface}; note that  each $\kappa_i$ is a diffeomorphism onto its image. 
Given $0 \leq R \leq \frac{1}{2m}$, we further define the regions
\begin{equation}
\label{eqn:catr_init_def}
K_i(R)
=
K_{i;N, m, \vec{\zeta}, \vec{\xi}\,}(R)
\vcentcolon=
\bigl\{
  \kappa_i(t,\vartheta)
  \st
  \tau_i \cosh t \leq R
  \mbox{ and } \abs{\vartheta} \leq \pi/2
\bigr\},
\end{equation}
so we have in particular
$K_i=K_i\bigl(1/(2m)\bigr)$.
We also define on each initial surface
$\Sigma = \Sigma_{N,m,\vec{\zeta},\vec{\xi}}$
the function
\begin{equation*}
\rho \colon \Sigma \to \interval{0,\infty}
\end{equation*}
to be the unique $\G$-invariant
function on $\Sigma$
whose restriction to each $D_i$
is as follows.
Recalling from \eqref{def:Phi}
the domain of the map $\Phi$,
let
\begin{equation}
\label{eqn:axes_def}
L_i
\vcentcolon=
\Biggl\{
  (\sigma,\theta,\omega) \in \dom \Phi
  \st
  \sigma = 0
  \mbox{ and }
  \begin{cases}
    \frac{m\theta}{\pi} \in \frac{1}{2} + 2\Z
      &\mbox{ for } i=1
  \\
    \frac{m\theta}{\pi} \in \frac{1}{2} + \Z
      &\mbox{ for } 2 \leq i \leq N-1
  \\
    \frac{m\theta}{\pi} \in (-1)^N\frac{1}{2} + 2\Z
      &\mbox{ for } i=N
  \end{cases}
\Biggr\},
\end{equation}
let $d^{L_i} \colon \R^3 \to \R$
be the distance-to-$L_i$ function,
and let
\begin{equation}
\label{eqn:rho_def}
\begin{aligned}
\rho|_{D_i \setminus \Phi(\{d^{L_i} < 2/m\})}
  &\vcentcolon=
  m,
\\
\Phi^*\rho|_{\Phi^{-1}(D_i) \cap \{d^{L_i} < 2/m\}}
  &\vcentcolon=
\frac{1}{d^{L_i}} \cdot (\Psi \circ md^{L_i})
  +m \cdot (1-\Psi \circ md^{L_i}).
\end{aligned}
\end{equation}
The function $\rho$ is introduced
for the purpose of defining on the initial surface
the conformal metric $\rho^2\met{\Sigma}$
(whose analogue in each of
\cites{KapouleasYangTorusDoubling, KapouleasWiygul17}
is called the $\chi$ metric),
which will be used to study the linearized problem.
In particular we have
on each $K_i$
\begin{align}
\label{eqn:rho_on_cat}
\kappa_i^*\rho
&=
\tau_i^{-1} \sech t,
&
d(\kappa_i^*\rho)
&=
-(\kappa_i^*\rho) \tanh t \, dt.
\end{align}
while elsewhere
\begin{align}
\label{eqn:rho_off_cat}
\frac{m}{3}
 &\leq \rho|_{\Sigma \setminus \pyr_m \bigcup K_j}
  \leq 3m,
&
\nm{\rho}_{
    C^k\bigl(
      \overline{\Sigma \setminus \pyr_m \bigcup K_j}, ~
      \rho^2\met{\Sigma}
    \bigr)
}
&\leq
C(k)m
\end{align}
for each integer $k \geq 0$
and some corresponding constant $C(k)>0$
independent of $m$.

\begin{lemma}[Geometric estimates on the $K_i$ regions]
\label{lem:geo_ests_cat}
Given data $(N,\vec{\zeta},\vec{\xi}\,)$ as above 
there exist $C=C(N)>0$ and 
$m_0=m_0(N,\abs{\vec{\zeta}},\abs{\vec{\xi}})>0$
such that for every integer $m>m_0$
the following items hold for each
$K_i\subset\Sigma=\Sigma_{N,m,\vec{\zeta},\vec{\xi}}$:
\begin{enumerate}[label={\normalfont(\roman*)}]
\item \label{itm:geo_ests_cat:rho}
$\nm{\rho^{-1}}_{C^0(K_i(R))}\leq CR$
for all $\tau_i < R \leq \frac{1}{2m}$,

\item \label{itm:geo_ests_cat:met}
$
 \nm[\big]{
   \kappa_i^*\rho^2\met{\Sigma}
   -(dt^2+d\vartheta^2)
   }_{
     C^2(
       \kappa_i^{-1}(K_i(R)),~
       dt^2+d\vartheta^2
     )
    }
 \leq CR$
for all $\tau_i < R \leq \frac{1}{2m}$,

\item \label{itm:geo_ests_cat:twoff}
$
 \nm[\big]{
   \kappa_i^*\rho^{-2}\abs{\twoff{\Sigma}}_{\met{\Sigma}}^2
   -2 \sech^2 t
 }_{C^1(\kappa_i^{-1}(K_i),~ dt^2+d\vartheta^2)}
 \leq
 C\tau_1
$,

\item \label{itm:geo_ests_cat:mc}
$
 \nm[\big]{\rho^{-1}H^\Sigma|_{K_i}}_{C^1(K_i, ~
     \rho^2\met{\Sigma})}
 \leq
 C\tau_1
$.
\end{enumerate}
\end{lemma}

\begin{proof}
Item \ref{itm:geo_ests_cat:rho}
is clear from
\eqref{eqn:catr_init_def}
and \eqref{eqn:rho_on_cat}.
For the remaining items
we can appeal to the computations
in Appendix \ref{app:geo_ests}.
Item \ref{itm:geo_ests_cat:met}
follows from
\eqref{eqn:cat_met_calc}
and
\eqref{eqn:cat_aux_fun_ests},
recalling \eqref{eqn:tau_size}.
Using \eqref{eqn:cat_aux_fun_ests} again
but now in combination with \eqref{eqn:cat_twoff_calc},
we find
\begin{equation*}
\nm[\big]{
  \tau_i^{-1}\kappa_i^*\twoff{\Sigma}
  -(-1)^{i-1}(dt^2-d\vartheta^2)
}_{C^1(\kappa_i^{-1}(K_i(R)), ~ dt^2+d\vartheta^2)}
\leq
CR,
\end{equation*}
whence follow, using also item \ref{itm:geo_ests_cat:met}, items \ref{itm:geo_ests_cat:twoff} and \ref{itm:geo_ests_cat:mc}.
\end{proof}

\begin{figure}\centering
\pgfmathsetmacro{\rmax}{1.15}
\pgfmathsetmacro{\mpar}{18}
\begin{tikzpicture}[semithick,line cap=round,line join=round,scale=3,declare function={psi(\r)=cutoff( (1-\r)*\mpar/5 );}] 
\draw[->](0,0)--(\rmax,0)node[above]{$r$};
\draw[->](0,-0)--(0,1.2);
\draw plot[plus] (0,0)node[below]{$0$};
\draw plot[hdash](0,1)node[left]{$1$};
\draw plot[vdash] (1,0)node[below]{$1$};
\draw plot[vdash] (1-5/\mpar,0); 
\draw plot[vdash] (1-10/\mpar,0); 
\draw[dotted](1,0)--++(0,1.2)(1-5/\mpar,0)--++(0,1)(1-10/\mpar,0)--++(0,1.2); 
\path (1-5/\mpar,1/2)--(1,1/2)node[midway,inner sep=1pt](5/m){$\frac{5}{m}$}coordinate[pos=0](a)coordinate[pos=1](b);
\path (1-10/\mpar,1.2)--++(10/\mpar,0)node[midway,inner sep=1pt](10/m){$\frac{10}{m}$}
coordinate[pos=0](c)coordinate[pos=1](d);
\draw[-latex](5/m)--(a);
\draw[-latex](5/m)--(b);
\draw[-latex](10/m)--(c);
\draw[-latex](10/m)--(d);
\draw[thick,domain=0:\rmax,samples=200,smooth,variable=\x]plot({\x},{psi(\x)});
\end{tikzpicture} 
\caption{Graph of the cutoff function $r\mapsto\Psi\left(\frac{m}{5}\bigl(1-r\bigr)\right)$.}%
\label{fig:varpi}%
\end{figure}

Recalling that $\Psi \colon \R \to \R$
is a fixed smooth cutoff
function that is constantly $1$
on $\{x \leq 1\}$
and constantly $0$
on $\{x \geq 2\}$,
we define
$\psi^\varpi_m \colon \R^2 \to \R$ by
\begin{align*}
\psi^\varpi_m(x,y)
&\vcentcolon=
\Psi\left(\frac{m}{5}\Bigl(1-\sqrt{x^2+y^2}\Bigr)\right)
\intertext{(see Figure \ref{fig:varpi}) and
$\varpi^m \colon \R^3 \to \R^2 \times \{0\}$ by}
\varpi^m(x,y,z)
  &\vcentcolon=
  \frac{(x,y,0)}
  {
    1-\psi^\varpi_m(x,y)
    +\psi^\varpi_m(x,y)
      \sqrt{\frac{x^2+y^2}{x^2+y^2+z^2}}
  }.
\end{align*}
Note that
\begin{alignat}{3}
\label{eqn:varpi_on_core}
\varpi^m(x,y,z)
  &=(x,y,0)
  &&\text{ on } &
  \bigl\{\textstyle\sqrt{x^2+y^2} &\leq 1-10m^{-1}\bigr\},
\\
\label{eqn:varpi_on_Phi}
 \varpi^m(\Phi(\sigma,\theta,\omega))
 &=
 \Phi(\sigma,\theta,0) ~
  &&\text{ on } &
  \Phi^{-1}\bigl(\bigl\{\textstyle\sqrt{x^2+y^2}
  &\geq
  1-5m^{-1}\bigr\}\bigr),
\end{alignat}
$\varpi^m|_{\{z=0\}}$ is the identity,
and
$\varpi^m(\B^3)=\B^2=\B^3 \cap \{z=0\}$.
Next,
for each region $D_i$
in \eqref{layers_with_half_cats}
we define
\begin{equation}
\label{eqn:varpi_i_def}
\varpi_i
=
\varpi_{i; N, m, \vec{\zeta}, \vec{\xi}}
\vcentcolon=
\varpi^m|_{D_{i; N, m, \vec{\zeta}, \vec{\xi}}}
  \colon
  D_{i; N, m, \vec{\zeta}, \vec{\xi}}
  \to
  \B^2.
\end{equation}
Given $0 \leq R \leq \frac{1}{2m}$,
we further define the regions (see Figure \ref{fig:Initialsurface})
\begin{equation}
\label{eqn:discr_init_def}
D_i(R)
=
D_{i; N, m, \vec{\zeta}, \vec{\xi}\,}(R)
\vcentcolon=
\closure{
  D_i
  \setminus
  \pyr_m \bigcup_j K_j(R)
}
\end{equation}
where the overline indicates the closure
in the initial surface $\Sigma$;
in particular we have $D_i=D_i(0)$.

Because of obstructions we will confront
to solving the linearized problem
it is necessary to isolate a certain
component of the initial mean curvature.
For this purpose
(and later for expressing these obstructions)
we define for each $D_i$
a function
$\wbar_i \in C^\infty(\Sigma)$
with support contained in $D_i$
by
\begin{equation}
\label{eqn:wbar_def}
\wbar_i
\vcentcolon=
\rho^{-2}\varpi_i^*
  [
    (\partial_x^2+\partial_y^2)
    \widehat{\overline{v}}
  ],
\end{equation}
where $\widehat{\overline{v}}$
is the unique
$\pyr_m$-invariant function
on $\B^2$
with support contained in
$\{\sqrt{x^2+y^2} \geq 1-2/m\}$
and satisfying
\begin{equation}
\label{eqn:vbarhat_ess_def}
\widehat{\overline{v}}(\Phi(\sigma,\theta,0))
=
\psi_m^\sigma(\psi_m^+ - \psi_m^-)(\sigma,\theta)
\quad \forall (\sigma,\theta) \in \Lambda^\Phi_m,
\end{equation}
for which we recall
the three cutoff functions
$\psi_m^\sigma$, $\psi_m^{\pm}$
from \eqref{eqn:cat_cutoffs}. We do not really need the functions
$\wbar_1$ and $\wbar_N$,
but for purposes of notation
it is somewhat convenient to have them defined.
For any given $\vec{V} \in \R^{N,0,0}$
we write
\begin{equation*}
\vec{V} \cdot \vec{\wbar}
\vcentcolon=
\sum_{i=1}^N V_i \wbar_i
=
\sum_{i=2}^{N-1} V_i \wbar_i.
\end{equation*}
Note that
$\vec{V} \in \R^{N,0,0}_{\updownarrow+}$ implies that 
$\vec{V} \cdot \vec{\wbar}$ is $\G$-equivariant.

\begin{lemma}
[Geometric estimates on the $D_i$ regions]
\label{lem:geo_ests_disc}
Given data $(N,\vec{\zeta},\vec{\xi})$
as above
there exist
$C=C(N)>0$
and
$m_0=m_0(N,\abs{\vec{\zeta}},\abs{\vec{\xi}})>0$
such that for every integer $m>m_0$
the following items hold
for each
$
D_i 
  \subset 
  \Sigma
  =
  \Sigma_{N,m,\vec{\zeta},\vec{\xi}}
$:

\begin{enumerate}[label={\normalfont(\roman*)}]
\item \label{itm:geo_ests_disc:diffeo}
  $\varpi_i$ is a diffeomorphism onto its image,

\item \label{itm:geo_ests_disc:flat}
  $
  \twoff{\Sigma}|_{\Sigma \cap (1-\frac{3}{m})\B^3}
  =
  0
  $,

\item \label{itm:geo_ests_disc:met_flat}
$
\nm[\big]{
  \met{\Sigma}
  -\varpi_i^*(dx^2+dy^2)
}_{
  C^2(
    \Sigma \cap (1-\frac{3}{m})\B^3, ~
    \met{\Sigma}
  )
}
\leq
\tau_1^{3/2}
$,

\item \label{itm:geo_ests_disc:met_Phullback}
$
\nm[\big]{
\Phi^*\met{\Sigma}
  -
  (d\sigma^2+(1-\sigma)^2 \, d\theta^2)
}_{
  C^2(
    \Phi^{-1}(D_i(m^{-4}))
      \cap \{\sigma \leq \frac{5}{m}\}, ~
    d\sigma^2 + d\theta^2
  )
}
\leq
\tau_1^{3/2}
$,

\item \label{itm:geo_ests_disc:twoff_Phullback}
$
\nm[\big]{
\Phi^*\abs{\twoff{\Sigma}}_{\met{\Sigma}}^2
}_{
  C^1(
    \Phi^{-1}(D_i(m^{-4}))
      \cap \{\sigma \leq \frac{5}{m}\}, ~
    d\sigma^2 + d\theta^2
  )
}
\leq
\tau_1^{3/2}
$,

\item \label{itm:geo_ests_disc:conf_met}
  $
  \nm[\big]{\rho^2\met{\Sigma}
  -\rho^2\varpi_i^*(dx^2+dy^2)}_{
    C^2(D_i(m^{-4}), ~\rho^2\met{\Sigma})
  }
  \leq
  Cm^{-1}
  $,

\item \label{itm:geo_ests_disc:twoff}
  $
  \nm[\big]{\abs{\twoff{\Sigma}}_{\met{\Sigma}}^2}_{
    C^1(D_i(m^{-4}), ~\rho^2\met{\Sigma})
  }
  \leq C\tau_1^{3/2}
  $,

\item
\label{itm:geo_ests_disc:mc}
$
 \nm[\big]{H^\Sigma-\rho^2\disloc_i \wbar_i}_{
   C^1(D_i((4m)^{-1}), ~\rho^2\met{\Sigma})
   }
 \leq
 Cm^2\tau_1
$,
\end{enumerate}
where, we recall, $H^\Sigma$
is the scalar projection of the vector-valued
mean curvature of $\Sigma$
onto the unit normal
which points upward at $D_1 \cap \{x=y=0\}$.
\end{lemma}

\begin{proof}
By construction
(see in particular
\eqref{building_blocks}
and the prior supporting definitions)
$\Sigma \cap (1-\frac{3}{m})\B^3$
consists of $N$ horizontal Euclidean discs,
verifying item \ref{itm:geo_ests_disc:flat}
and also the equality
\begin{equation}
\label{flat_metric}
\met{\Sigma}|_{\Sigma \cap (1-\frac{3}{m})\B^3}
=
d(x|_\Sigma)^2 + d(y|_\Sigma)^2.
\end{equation}
We will next check items \ref{itm:geo_ests_disc:diffeo}
and \ref{itm:geo_ests_disc:met_flat}.
The only point of item \ref{itm:geo_ests_disc:diffeo}
that may not be obvious is the injectivity.
In view of \eqref{eqn:varpi_on_Phi}
and the fact that near $\partial \B^3$ each $D_i$
(recall \eqref{layers_with_half_cats} and \eqref{building_blocks})
was constructed as a graph in the domain of the map $\Phi$,
it suffices to check injectivity on
$\Sigma \cap (1-\frac{3}{m})\B^3$
(where, for the purposes of comparing
the values of $\sqrt{x^2+y^2}$ and $\sqrt{x^2+y^2+z^2}$,
we bear in mind the height estimate \eqref{eqn:height_bound}
and the comparison of $m$ and $\tau_1$ in \eqref{eqn:tau_size}).
For this we will compare $\varpi_i$ to orthogonal projection onto $\{z=0\}$,
which comparison will also establish \ref{itm:geo_ests_disc:met_flat},
in light of \eqref{flat_metric}.

Defining
$s_i \colon \R^2 \to \R$ by
\begin{equation*}
\frac{1}{s_i(x,y)}
\vcentcolon=
1
  -\frac{
    (h^B_i)^2
      \psi_m^\varpi(x,y)
   }
   {
      \sqrt{x^2+y^2+(h^B_i)^2}
      \bigl(
        \sqrt{x^2+y^2}
        +\sqrt{x^2+y^2+(h^B_i)^2}
      \bigr)
    },
\end{equation*}
we have
\begin{equation*}
\forall
    (x,y,h^B_i)
    \in
    D_i \cap \bigl(1-\tfrac{3}{m}\bigr)\B^3
\quad
\varpi_i(x,y,h^B_i)=s_i(x,y)(x,y,0)
\end{equation*}
and
\begin{equation*}
\nm{s_i-1}_{C^3(\B^2,dx^2+dy^2)}
\leq
Cm^3(h_i^B)^2
\end{equation*}
for some $C>0$
(independent of $N,m,\vec{\zeta},\vec{\xi}\,$),
which,
recalling \eqref{eqn:height_bound} and \eqref{eqn:tau_size}
and taking $m$ sufficiently large
in terms of a universal constant,
completes the proof of items
\ref{itm:geo_ests_disc:diffeo}
and
\ref{itm:geo_ests_disc:met_flat}.

For the remainder of the proof
we refer to calculations in Appendix \ref{app:geo_ests}.
Item \ref{itm:geo_ests_disc:met_Phullback}
follows by applying 
\eqref{eqn:dsc_def_fun_est}
in \eqref{eqn:dsc_met_calc},
while \ref{itm:geo_ests_disc:twoff_Phullback}
follows by applying
\eqref{dsc_nu_norm}
and
\eqref{eqn:dsc_def_fun_est}
in  \eqref{eqn:dsc_twoff_calc}
(and in both cases again appealing to \eqref{eqn:tau_size}).

From item \ref{itm:geo_ests_disc:met_flat}
and conjoint application of
item \ref{itm:geo_ests_disc:met_Phullback},
\eqref{eqn:varpi_on_Phi},
and \eqref{eqn:Phi_pullback_metric}
we ensure
\begin{equation*}
\nm[\big]{\met{\Sigma}-\varpi_i^*(dx^2+dy^2)}_{
  C^2(D_i(m^{-4}), ~\met{\Sigma})
}
\leq
Cm^{-1}
\end{equation*}
(the error of order $m^{-1}$ arising from the comparison
of the metrics $d\sigma^2+d\theta^2$ and $d\sigma^2+(1-\sigma)^2 d\theta^2$
on $\{\sigma \leq \frac{5}{m}\}$),
while from items
\ref{itm:geo_ests_disc:flat}
and \ref{itm:geo_ests_disc:twoff_Phullback},
along with \eqref{eqn:Phi_pullback_metric},
we ensure
\begin{equation*}
\nm[\big]{\abs{\twoff{\Sigma}}_{\met{\Sigma}}^2}_{
  C^1(
    D_i(m^{-4}),
    ~\met{\Sigma}
  )
}
\leq
C\tau_1^{3/2},
\end{equation*}
both by taking $m$ sufficiently large in terms of
$N,\abs{\vec{\zeta}},\abs{\vec{\xi}}$,
and with $C>0$ a constant
independent of $m,\vec{\zeta},\vec{\xi}$.
Using \eqref{eqn:rho_on_cat},
\eqref{eqn:rho_off_cat},
and item \ref{itm:geo_ests_cat:met}
of Lemma \ref{lem:geo_ests_cat},
the above two displayed estimates imply
items \ref{itm:geo_ests_disc:conf_met}
and \ref{itm:geo_ests_disc:twoff}.

Finally,
noting the choice of direction
in defining the (scalar-valued) mean curvature
at the end of the statement
of the lemma
and referring to the definition
\eqref{eqn:wbar_def}
of $\wbar_i$,
we obtain
item \ref{itm:geo_ests_disc:mc}
from the estimates
\eqref{eqn:dsc_mc_est}
and \eqref{eqn:disloc_mc_est}
(and from item \ref{itm:geo_ests_disc:flat}).
In order
to pass from the estimates
with respect to $m^2(d\sunder^2+d\thunder^2)$
made in \eqref{eqn:dsc_mc_est} and \eqref{eqn:disloc_mc_est}
to the desired estimate with respect
to $\rho^2\met{\Sigma}$
we apply
item \ref{itm:geo_ests_disc:conf_met},
\eqref{eqn:varpi_on_Phi},
and \eqref{eqn:Phi_pullback_metric}.
\end{proof}

\paragraph{On the construction of free boundary minimal stackings of any order -- linear theory.}

For any initial surface $\Sigma$,
each integer $k \geq 0$,
each $\alpha,\beta \in \interval{0,1}$,
any $u \in C^{k,\alpha}(\Sigma)$,
and any $f \in C^{k,\alpha}(\partial\Sigma)$
we define
\begin{equation}
\label{eqn:global_norms}
\begin{aligned}
\nm{u}_{k,\alpha,\beta}
&\vcentcolon=
  \nm{m^{-\beta}\rho^{\beta}u}_{
    C^{k,\alpha}(\Sigma,\rho^2\met{\Sigma})
    },
\\
\nm{f}_{k,\alpha,\beta}
&\vcentcolon=
  \nm{m^{-\beta}\rho^{\beta}f}_{
    C^{k,\alpha}(\partial\Sigma,\rho^2\met{\Sigma})
    },
\\
\nm{u}'_{\alpha,\beta}
&\vcentcolon=
  \nm{u}_{0,\alpha,\beta}
  +
  m^2\nm{u}_{
      C^{0,\alpha}(
        \Sigma
        \cap 
        \{
          \sqrt{x^2+y^2} \leq 1-10/m
        \}, ~
        m^2\met{\Sigma} 
        )
      },
\\
\nm{(u,f)}'_{\alpha,\beta}
&\vcentcolon=
  \nm{u}'_{\alpha,\beta}
    +\nm{f}_{1,\alpha,\beta}.
\end{aligned}
\end{equation}
Now we can state our global estimate of the initial mean curvature.

\begin{lemma}[Estimate of the initial mean curvature]
\label{lem:initial_mc}
For any integer $N \geq 2$,
any parameters
$\vec{\zeta} \in \R^{N-1}_{\updownarrow+}$
and
$\vec{\xi} \in \R^{N-1}_{\updownarrow-}$,
and any reals
$\alpha,\beta \in \interval{0,1}$
there exist
$m_0=m_0(N,\abs{\vec{\zeta}}+\abs{\vec{\xi}})>0$
and $C=C(N,\alpha,\beta)>0$
such that 
on every initial surface
$\Sigma_{N,m,\vec{\zeta},\vec{\xi}}$ with $m>m_0$
there holds
\begin{equation*}
\nm{
  \rho^{-2}H^\Sigma
  -\vec{\disloc} \cdot \vec{\wbar}
}'_{\alpha,\beta}
\leq
C\tau_1.
\end{equation*}
\end{lemma}

\begin{proof}
The estimate is a straightforward consequence
of
Lemma~\ref{lem:geo_ests_cat} \ref{itm:geo_ests_cat:mc},
Lemma~\ref{lem:geo_ests_disc} \ref{itm:geo_ests_disc:flat},
Lemma~\ref{lem:geo_ests_disc}  \ref{itm:geo_ests_disc:mc},
and the definition
of the $\nm{\cdot}'_{\alpha,\beta}$ norm
in \eqref{eqn:global_norms}.
\end{proof}

We shall also append the suffix $\G$
to the notation for a given function space
over $S \in \{\Sigma, \partial \Sigma\}$
to denote the corresponding subspace
of $\G$-equivariant functions:
\begin{align*}
C^{k,\alpha}_\G(S)
&\vcentcolon=
\{
  u \in C^{k,\alpha}(S)
  \st
  \forall \rot \in \G
  \quad
  u = (\nrmlsgn{\nu^{\Sigma}}(\rot))(u \circ \rot^{-1})
\}
&
(S&\in \{\Sigma, \partial \Sigma\}),
\end{align*}
where $\nrmlsgn{\nu^{\Sigma}}(\rot)=1$
if the symmetry $\rot$ preserves each global unit normal for $\Sigma$,
while $\nrmlsgn{\nu^{\Sigma}}(\rot)=-1$
if instead $\rot$ exchanges the two global unit normals.
At this point we fix also,
for each initial surface
$\Sigma_{N, m, \vec{\zeta}, \vec{\xi}}$,
a smooth diffeomorphism
\begin{equation*}
P_{\vec{\zeta},\vec{\xi}}
=
P_{N,m,\vec{\zeta},\vec{\xi}}
\colon
\Sigma_{N, m, \vec{0}, \vec{0}}
\to
\Sigma_{N, m, \vec{\zeta}, \vec{\xi}}
\end{equation*}
such that
$P_{\vec{0},\vec{0}}$ is the identity,
$P_{\vec{\zeta},\vec{\xi}}$ commutes with $\G$,
$P_{\vec{\zeta},\vec{\xi}}$ regarded as a map into $\B^3$
depends smoothly on
$(\vec{\zeta},\vec{\xi}\,)$,
and moreover 
for any integers $N \geq 2$,
$k \geq 0$
and any reals $c>0$,
$\alpha, \beta \in \interval{0,1}$
there exist $m_0=m_0(c,k,\alpha,\beta,N)>0$
and $C=C(c,k,\alpha,\beta,N)$
such that for every integer $m>m_0$,
all
$\vec{\zeta} \in \IntervaL{-c,c}^{N-1}_{\updownarrow+}$,
$\vec{\xi} \in \IntervaL{-c,c}^{N-1}_{\updownarrow-}$,
and any functions
$
 u
 \in
 C^{k,\alpha}(\Sigma_{N,m,\vec{\zeta},\vec{\xi}})
$,
$
 v
 \in
 C^{k,\alpha}(\Sigma_{N,m,\vec{0},\vec{0}})
$
\begin{align}
\label{eqn:diffeo_pullback_ests}
\nm{P_{\vec{\zeta},\vec{\xi}~}^*u}_{k,\alpha,\beta}
&\leq
  C\nm{u}_{k,\alpha,\beta},
&
\nm{P_{\vec{\zeta},\vec{\xi}}^{-1*}v}_{k,\alpha,\beta}
&\leq
  C\nm{v}_{k,\alpha,\beta}.
\end{align}

\begin{figure}%
\centering
\pgfmathsetmacro{\xscalepar}{6}
\pgfmathsetmacro{\mpar}{20}
\begin{tikzpicture}[semithick,line cap=round,line join=round,scale=3,xscale=\xscalepar,baseline={(0,0)},
declare function={wi(\s)=cutoff(2*\mpar*abs(\s-3/\mpar));}
] 
\draw[->](0,0)--(1/3+0.15/\xscalepar,0)node[above]{$\sigma$};
\draw[->](0,-0)--(0,1.15);
\draw plot[plus] (0,0)node[below]{$0$};
\draw plot[hdash](0,1)node[left]{$1$};
\draw plot[vdash] (4/\mpar,0)node[below]{$\frac{4}{m}$};
\draw plot[vdash] (3/\mpar,0)node[below]{$\frac{3}{m}$};
\draw plot[vdash] (2/\mpar,0)node[below]{$\frac{2}{m}$};
\draw plot[vdash] (1/3,0)node[below]{$\frac{1}{3}$}; 
\draw[dotted](3/\mpar,0)--++(0,1); 
\draw[thick,domain=2/\mpar:4/\mpar,samples=200,smooth,variable=\x]
(0,0)--plot({\x},{wi(\x)})--(1/3,0);
\end{tikzpicture}
\caption{Profile of the cutoff function $\Psi(2m\lvert{\sigma-\frac{3}{ms}}\rvert)$ employed in \eqref{eqn:w_ess_def}.}%
\label{fig:w_ess_def}%
\end{figure}

For each $D_i \subset \Sigma$ we define
$w_i \in C^\infty(\Sigma)$
with support contained in $D_i$
to be the unique $\pyr_m$-invariant
function on $D_i$
with support contained in
$\Phi(\{\frac{2}{m} \leq \sigma \leq \frac{4}{m}\})$
and satisfying 
\begin{equation}
\label{eqn:w_ess_def}
\forall (\sigma,\theta) \in \Lambda^\Phi_m
\quad
w_i\bigl(\varpi_i^{-1}(\Phi(\sigma,\theta,0))\bigr)
=
(-1)^{i-1}\Psi\bigl(2m\abs{\sigma-\tfrac{3}{m}}\bigr),
\end{equation}
for which we recall
that $\Psi \colon \R \to \R$
is a fixed cutoff function
that is constantly $1$
on $\intervaL{-\infty,1}$
and constantly $0$
on $\Interval{2,\infty}$ 
(see Figure \ref{fig:w_ess_def}).
For any given $\vec{V} \in \R^N$ we write
\begin{equation*}
\vec{V} \cdot \vec{w}
\vcentcolon=
\sum_{i=1}^N V_i w_i.
\end{equation*}
Note that
$
\vec{V} \in \R^N_{\updownarrow-}
$ implies that $\vec{V} \cdot \vec{w}$ is $\G$-equivariant.

Before proceeding further, let us pause for introducing an important piece of notation: given an initial surface $
 \Sigma=
 \Sigma_{N, m, \vec{\zeta}, \vec{\xi}}
$ we let $J_{\Sigma}$ denote its Jacobi operator, i.\,e. we set
$J_{\Sigma}:=\Delta_{\Sigma}+\abs{\twoff{\Sigma}}^2$; furthermore, we define the Robin operator $\eta^{\Sigma}-1$ where $\eta^{\Sigma}$, we recall, is the outward unit conormal to $\Sigma\subset\B^3$.

\begin{proposition}
[Global solution operator modulo cokernel]
\label{pro:global_inverse}
Let $N \geq 2$ be an integer,
$c>0$,
$\alpha, \beta \in \interval{0,1}$.
There exist $m_0,C>0$ such that
on every initial surface
$
 \Sigma=
 \Sigma_{N, m, \vec{\zeta}, \vec{\xi}}
$
with $\abs{\vec{\zeta}}+\abs{\vec{\xi}}<c$ 
and $m>m_0$
there is a linear map
\begin{equation*}
R_\Sigma
\colon
C^{0,\alpha}_\G(\Sigma)
  \oplus
  C^{1,\alpha}_\G(\partial\Sigma)
\to
C^{2,\alpha}_\G(\Sigma)
  \oplus
  \R^N_{\updownarrow-}
  \oplus
  \R^{N,0,0}_{\updownarrow+}
\end{equation*}
such that
for any $(E,f)$ in its domain 
the corresponding image
$(u,\vec{\mu},\vec{\mubar}) \vcentcolon= R_\Sigma(E,f)$
satisfies
\begin{align*}
\rho^{-2}J_{\Sigma}u
&=
  E
    +\vec{\mu} \cdot \vec{w}
    +\vec{\mubar} \cdot \vec{\wbar},
&
\rho^{-1}(\eta^{\Sigma}u-u)&=f,
&
\nm{u}_{2,\alpha,\beta}
    +\abs{\vec{\mu}}
    +\abs{\vec{\mubar}}
&\leq
  C\nm{(E,f)}'_{\alpha,\beta}.
\end{align*}
Furthermore,
for any data as above,
we have continuity of the map
\begin{align*} 
C_\G^{0,\alpha}(\Sigma_{N,m,\vec{0},\vec{0}}) 
\oplus
C_\G^{1,\alpha}(\partial \Sigma_{N,m,\vec{0},\vec{0}})
\oplus
[-c,c]^{N-1}_{\updownarrow+}
\oplus
[-c,c]^{N-1}_{\updownarrow-}
&\to
C_\G^{2,\alpha}(\Sigma_{N,m,\vec{0},\vec{0}})
\oplus \R^N_{\updownarrow-}
\oplus \R^{N,0,0}_{\updownarrow+} 
\\[.5ex]
\bigl(E, f, \vec{\zeta}, \vec{\xi}\,\bigr)
&\mapsto\bigl(P_{\vec{\zeta}, \vec{\xi}~}^*u,\vec{\mu}, \vec{\mubar}\bigr),
\shortintertext{where $(u, \vec{\mu}, \vec{\mubar})
  \vcentcolon=
  R_{\Sigma_{\vec{\zeta},\vec{\xi}}}
    P_{\vec{\zeta},\vec{\xi}}^{-1*}
    (E,f)$. }
\end{align*}
\end{proposition}

\begin{proof}
Suppose
$
(E,f)
\in
C^{0,\alpha}_\G(\Sigma)
  \oplus
  C^{1,\alpha}_\G(\partial\Sigma)
$,
and recall that $\Psi \colon \R \to \R$
is a fixed smooth function
that is constantly $1$
on $\{x \leq 1\}$
and $0$ on $\{x \geq 2\}$.
Set
\begin{align*}
\Psi_K
&\vcentcolon=
\Psi \circ \tfrac{8m}{\rho},
&
E_K
&\vcentcolon=
\Psi_K E,
&
f_K
&\vcentcolon=
 \Psi_K f,
\end{align*}
and, recalling Lemma \ref{lem:cat_model_sol} that concerns the construction of the resolvent $R_K$ on half-catenoids,
define
$u_{K_i} \colon \Sigma \to \R$
(for $1 \leq i \leq N-1$)
and
$u_K \colon \Sigma \to \R$
by
\begin{align*}
u_{K_i}
  &\vcentcolon=
  \kappa_i^{-1*}R_K(\kappa_i^*E_K, ~ \kappa_i^*f_K),
&
u_K|_{K_i}
  &\vcentcolon=
  \Psi_K u_{K_i}
\end{align*}
for each $K_i$.
It follows easily that
$u_K \in C^{2,\alpha}_\G(\Sigma)$
with
\begin{equation}
\label{eqn:uK_est}
\nm{u_K}_{2,\alpha,\beta}
\leq
C\nm{(E,f)}'_{\alpha,\beta};
\end{equation}
here and throughout the proof
$C>0$ is a constant whose value
may vary from instance to instance
but can be chosen
independently of $m$, $\vec{\zeta}$, $\vec{\xi}$, and $c$,
so that all asserted estimates hold.

Moreover,
defining
$E_\Psi \colon \Sigma \to \R$
and $f_\Psi \colon \partial \Sigma \to \R$
by
\begin{align*}
E_\Psi|_{K_i}
&\vcentcolon=
[\rho^{-2}J_\Sigma, \Psi_K]u_{K_i}
  +(\Psi_K-1)E,
\\
f_\Psi|_{K_i}
&\vcentcolon=
[\rho^{-1}, \Psi_K]u_{K_i}
  +(\Psi_K-1)f
\end{align*}
(the first term in each case being the commutator
of the two enclosed operators)
and using items
\ref{itm:geo_ests_cat:rho},
\ref{itm:geo_ests_cat:met},
and \ref{itm:geo_ests_cat:twoff}
of Lemma \ref{lem:geo_ests_cat}
in order to control
$
\kappa_i^*\rho^{-2}J_\Sigma
-L_K \kappa_i^*
$
and
$
\kappa_i^*\rho^{-1}(\eta^\Sigma-1)
-\eta^{\catdom_\infty} \kappa_i^*
$,
we also have
\begin{equation}
\begin{aligned}
\label{eqn:uK_error_est}
\nm[\big]{
\bigl(
    \rho^{-2}J_\Sigma u_K
      -E_K - E_\Psi, ~
    \rho^{-1}(\eta^\Sigma-1)u_K
      -f_K - f_\Psi
\bigr)
}'_{\alpha,\beta}
&\leq
Cm^{-1}\nm{(E,f)}'_{\alpha,\beta},
\\
\nm{(E_\Psi, f_\Psi)}'_{\alpha,\beta}
&\leq
C\nm{(E,f)}'_{\alpha,\beta}.
\end{aligned}
\end{equation}
Next set
\begin{align*}
E_D&\vcentcolon=E - E_K - E_\Psi,
&
f_D&\vcentcolon=F - F_K - F_\Psi,
&
\Psi_D&\vcentcolon=\Psi \circ \tfrac{m^2}{\rho}
\end{align*}
and define $u_D \in C^{2,\alpha}(\Sigma)$ by
\(
u_D|_{D_i}
  \vcentcolon=
  (1-\Psi_D) \varpi_i^*u_i
\)
for each $D_i$, with $u_i \in C^{2,\alpha}(\B^2)$ constructed as follows.
First, define
$E_i \in C^{0,\alpha}(\B^2)$ and 
$f_i \in C^{1,\alpha}(\partial \B^2)$
by
\begin{align*}
E_i
&\vcentcolon=
    \varpi_i^{-1*}(\rho^2 E_D),
&
f_i
&\vcentcolon=
  \varpi_i^{-1*}(\rho f_D),
\end{align*}
with both functions extended to
take the constant value zero
outside the image of $\varpi_i$.
Next,
working with standard Cartesian
$(x,y)$ and polar $(r,\theta)$
coordinates
on $\B^2$,
let $\mu_i$
be the unique real such that
\begin{equation}
\label{eqn:disc_bvp_solvability}
\int_{\B^2}
\Bigl(E_i + \mu_i\varpi_i^{-1*}(\rho^2 w_i)\Bigr)
  \, dx \, dy
=\int_{\Sp^1} f_i \, d\theta,
\end{equation}
and let $\widetilde{u}_i \in C^{2,\alpha}(\B^2)$
be the unique solution to the Neumann problem
\begin{equation}
\label{eqn:disc_bvp}
\left\{
\begin{aligned}
(\partial_x^2+\partial_y^2)\widetilde{u}_i
  &=E_i + \mu_i\varpi_i^{-1*}(\rho^2 w_i)
    && \text{ in } \B^2 \\
  \partial_r\widetilde{u}_i
  &=f_i
    && \text{ on } \ \Sp^1=\partial\B^2
\end{aligned}
\right.
\end{equation}
such that
\begin{equation}
\label{eqn:disc_bvp_first_vanishing}
\begin{aligned}
  \widetilde{u}_1\bigl(\cos\tfrac{\pi}{2m}, \sin\tfrac{\pi}{2m}\bigr)
  &=0
  \quad
  &&(i=1),
\\
  \widetilde{u}_i\bigl(\cos\tfrac{\pi}{2m}, \sin\tfrac{\pi}{2m}\bigr)
  +\widetilde{u}_i\bigl(\cos\tfrac{\pi}{2m}, -\sin\tfrac{\pi}{2m}\bigr)
  &=0
  \quad
  &&(2 \leq i \leq N-1),
\\
\widetilde{u}_N\bigl(\cos\tfrac{\pi}{2m}, (-1)^N\sin\tfrac{\pi}{2m}\bigr)
&=0
  \quad
  &&(i=N).
\end{aligned}
\end{equation}
Note, referring to \eqref{eqn:w_ess_def}, that
$m^{-1}\abs[\big]{\int_{\B^2} \varpi_i^{-1}(\rho^2 w_i)}$ 
is a strictly positive constant independent of $m$
and that condition \eqref{eqn:disc_bvp_solvability}
is necessary and sufficient for the existence
of a solution to \eqref{eqn:disc_bvp};
any two such solutions differ by a free constant,
so imposition of \eqref{eqn:disc_bvp_first_vanishing}
uniquely determines $\widetilde{u}_i$.
Moreover, we have the estimate
\begin{equation}
\label{eqn:first_soln_on_disc_est}
\abs{\mu_i}
+\nm{
  \widetilde{u}_i
  }_{C^0(\B^2)}
\leq
C\nm{(E,f)}'_{\alpha,\beta},
\end{equation}
for which we emphasize the presence
of the second term in the definition of the
$\nm{\cdot}'_{\alpha,\beta}$ norm
in \eqref{eqn:global_norms}.

Defining
$\vec{G}, \vec{W} \in \R^N$
by
\begin{align*}
G_i&\vcentcolon=\int_{\B^2} E_i \, dx \, dy,
&
W_i&\vcentcolon=\int_{\B^2} \varpi_i^{-1*}(\rho^2w_i) \, dx \, dy,
\shortintertext{note also that}
\vec{G}&\in
  \begin{cases}
    \R^N_{\updownarrow+}
      &\text{ for $N$ even, }
    \\[1ex]
    \R^N_{\updownarrow-}
      &\text{ for $N$ odd, }
  \end{cases}
&
\vec{W}&\in
  \begin{cases}
  \R^N_{\updownarrow-}
      &\text{ for $N$ even, }
    \\[1ex]
    \R^N_{\updownarrow+}
      &\text{ for $N$ odd, }
  \end{cases} 
\end{align*}
whence $\vec{\mu} \in \R^N_{\updownarrow-}$ for $N$ of either parity.
The definition of $u_i$ is finally concluded by taking
\begin{equation}
\label{eqn:disc_bvp_second_vanishing}
\mubar_i
\vcentcolon=
\begin{cases}
0
&\text{ for } i \in \{1,N\}
\\
-\widetilde{u}_i\bigl(\cos \tfrac{\pi}{2m}, \sin \tfrac{\pi}{2m}\bigr)
&\text{ for }  2 \leq i \leq N-1
\end{cases}
\end{equation}
and setting
\(
u_i
\vcentcolon=
\widetilde{u}_i + \mubar_i \widehat{\overline{v}},
\)
recalling \eqref{eqn:vbarhat_ess_def}
for the definition of $\widehat{\overline{v}}$.
From the above observation
that $\vec{\mu} \in \R^N_{\updownarrow-}$,
the $\G$-equivariance assumed of $(E,f)$,
the $\G$-invariance of $\Sigma$,
and the definition \eqref{eqn:varpi_i_def}
of $\varpi$,
it readily follows that
$\mubar \in \R^{N,0,0}_{\updownarrow+}$
(using also the middle condition
in \eqref{eqn:disc_bvp_first_vanishing}
for $N$ odd)
and $u_D$ is $\G$-equivariant.

Additionally,
it is clear that
\eqref{eqn:first_soln_on_disc_est}
continues to hold with
$u_i$ in place of $\widetilde{u}_i$
and with an additional term
of $\abs{\mubar_i}$ added to its left-hand side.
Using standard elliptic estimates,
we then obtain
\begin{align}
\label{eqn:uD_est}
\nm{u_D}_{C^{2,\alpha}(\Sigma)}
  +\abs{\vec{\mu}}
  +\abs{\vec{\mubar}}
&\leq
C\nm{(E,f)}'_{\alpha,\beta}.
\shortintertext{In fact we have the estimate}
\label{eqn:uD_decay}
\nm{u_D}_{2,\alpha,(1+\beta)/2}
&\leq
C\nm{(E,f)}'_{\alpha,\beta},
\end{align}
as we now explain.
First,
note that the construction of
$(E_D,f_D)$
and the definitions
\eqref{eqn:w_ess_def}
and \eqref{eqn:wbar_def}
of the $w_i$ and $\wbar_i$
ensure that $u_i$
is harmonic on
$\varpi_i(\{\rho \geq 8m\})$,
an $O(m^{-1})$-neighborhood
of the points
on $\Sp^1=\partial \B^2$
with angular coordinate
\begin{equation*}
\theta \in
\begin{cases}
\tfrac{\pi}{2m} + \tfrac{2\pi}{m}\Z
  &\mbox{ if }   i=1
\\
\tfrac{\pi}{2m} + \tfrac{\pi}{m}\Z
  &\mbox{ if }   2 \leq i \leq N-1
\\
(-1)^N\tfrac{\pi}{2m} + \tfrac{2\pi}{m}\Z
  &\mbox{ if }   i=N.
\end{cases}
\end{equation*}
Moreover, the choices of $\vec{\mu}$
and $\vec{\mubar}$
in \eqref{eqn:disc_bvp_first_vanishing}
and \eqref{eqn:disc_bvp_second_vanishing}
were made to ensure that $u_i$
vanishes at the above points.
The estimate \eqref{eqn:uD_decay}
then follows from standard estimates
for harmonic functions on $\R^2$;
in fact one has even faster decay,
but \eqref{eqn:uD_decay} suffices
for the purposes of this proof.

Using the rapid decay
\eqref{eqn:uD_decay},
in particular
to control the error induced
by the cutoff factor
in the definition of $u_D$,
and using
\eqref{eqn:varpi_on_core}
and item \ref{itm:geo_ests_disc:flat}
of Lemma \ref{lem:geo_ests_disc}
(corresponding, slightly oversimplifying, 
to the fact that
$\Sigma \cap \{\sqrt{x^2+y^2<1-10m^{-1}}\}$
consists of $N$ exactly horizontal discs)
and also items
\ref{itm:geo_ests_disc:conf_met}
and \ref{itm:geo_ests_disc:twoff}
of this same lemma,
we conclude
\begin{equation*}
\nm[\Big]{
\Bigl(
    \rho^{-2} J_\Sigma u_D
      -E_D 
      -\vec{\mu} \cdot \vec{w}
      -\vec{\mubar} \cdot \vec{\wbar}, ~
    \rho^{-1}(\eta^\Sigma-1) u_D - f_D 
\Bigr)
}'_{\alpha,\beta}
\leq
C\bigl(m^{-1} + m^{(\beta-1)/2}\bigr)
  \nm[\big]{(E,f)}'_{\alpha,\beta}.
\end{equation*}
From this last estimate along with \eqref{eqn:uK_error_est} we see that the maps
\begin{align*}
\widetilde{R}\colon
C^{0,\alpha}_\G(\Sigma)
  \oplus
  C^{1,\alpha}_\G(\partial\Sigma)
&\to
C^{2,\alpha}_\G(\Sigma)
  \oplus
  \R^N_{\updownarrow-}
  \oplus
  \R^{N,0,0}_{\updownarrow+}
\\
(E,f)
&\mapsto
\bigl(u_K+u_D, ~\vec{\mu}, ~\vec{\mubar}\bigr),
\\[1.5ex]
L\colon
C^{2,\alpha}_\G(\Sigma)
  \oplus
  \R^N_{\updownarrow-}
  \oplus
  \R^{N,0,0}_{\updownarrow+}
&\to
C^{0,\alpha}_\G(\Sigma)
  \oplus
  C^{1,\alpha}_\G(\partial\Sigma)
\\
(u', ~\vec{\mu}', ~\vec{\mubar}')
&\mapsto
\bigl(
  \rho^{-2}J_\Sigma u'
    -\vec{\mu}' \cdot \vec{w}
    -\vec{\mubar}' \cdot \vec{\wbar}, ~
  \rho^{-1}(\eta^\Sigma-1)u'
\bigr)
\end{align*}
satisfy
$
\nm{I-L\widetilde{R}}
\leq
Cm^{(\beta-1)/2}
$,
where $I$ is the identity map
on the domain of $\widetilde{R}$
and the norm is the operator norm
when the latter is endowed with
the $\nm{\cdot}'_{\alpha,\beta}$ norm.
Thus $L\widetilde{R}$ is invertible
for $m$ sufficiently large,
and we conclude the proof
by taking
$
R_\Sigma
\vcentcolon=
\widetilde{R}(L\widetilde{R})^{-1}
$.
The claimed estimates
and symmetry assertions
are then clear from the foregoing construction,
while the continuity assertions
follow from the smooth dependence
of the initial surfaces
and the maps $\kappa_i, \varpi_i^{-1}$
(as maps into $\B^3$)
on the $\vec{\zeta},\vec{\xi}$ parameters.
\end{proof}

\begin{remark}[Global solution operator with homogeneous boundary data]
In practice we will apply
Proposition \ref{pro:global_inverse}
only with boundary data $f=0$.
For brevity
we shall write
$R_\Sigma E$ in place of $R_\Sigma(E,0)$
wherever we invoke Proposition \ref{pro:global_inverse}
in the sequel.
\end{remark}

\paragraph{On the construction of free boundary minimal stackings of any order -- nonlinear theory.}
We fix on $\R^3$ a metric
$\auxmet=\Omega^2(dx^2+dy^2+dz^2)$
which is everywhere conformal
to the standard Euclidean metric,
with conformal factor
\begin{equation*}
\Omega
\vcentcolon=
\begin{cases}
1
  &\text{ on } \
  \{
    (x,y,z)\in\R^3
    \st
    \abs{1-x^2-y^2-z^2} \geq 2/3
  \}
\\
(x^2+y^2+z^2)^{-\frac{1}{2}}
  &\text{ on } \
  \{
    (x,y,z)\in\R^3
    \st
    \abs{1-x^2-y^2-z^2} \leq 1/3
  \}.
\end{cases}
\end{equation*}
On each initial surface $\Sigma$
we pick the unit normal $\nu^\Sigma$
which points upward on the lowest disc.
Given a function $u \colon \Sigma \to \R$
we then define $\iota_u \colon \Sigma \to \R^3$
by
\begin{equation*}
\iota_u(p)
\vcentcolon=
\exp^{\auxmet}_p u(p)\nu^\Sigma(p),
\end{equation*}
where $\exp^{\auxmet}_p \colon T_p\R^3 \to \R^3$
is the exponential map of $\auxmet$ at $p$.
If $u$ is $C^2$ and sufficiently small
in terms of $\Sigma$,
then $\iota_u$ is an immersion
into $\B^3$
with well-defined (scalar-valued) mean curvature
$H_u$.
(In particular $H_0=H^\Sigma$.)
We set
\begin{equation}
\label{eqn:Q_def}
Q_u=
Q^{\vec{\zeta},\vec{\xi}}_u
\vcentcolon=
H_u - H_0 - J_\Sigma u.
\end{equation}

\begin{lemma}[Nonlinear
perturbation of the mean curvature]
\label{lem:quadratic_estimate}
Let $N \geq 2$ be a given integer
and $C_0,c>0$
and
$\alpha,\beta \in \interval{0,1}$
given reals.
There exists
$m_0=m_0(N,C_0,c,\alpha,\beta)>0$
such that for every integer $m \geq m_0$,
all
$
 \vec{\zeta}
 \in
 \IntervaL{-c,c}^{N-1}_{\updownarrow+}
$,
$
 \vec{\xi}
 \in
 \IntervaL{-c,c}^{N-1}_{\updownarrow-}
$,
and any real-valued function
$u$ on
$\Sigma=\Sigma_{N, m, \vec{\zeta},\vec{\xi}}$
with $\nm{u}_{2,\alpha,\beta} \leq C_0\tau_1$
we have
\begin{equation*}
\nm[\Big]{
    \rho^{-2}Q^{\vec{\zeta},\vec{\xi}}_u
}'_{\alpha,\beta}
\leq
\taubar_1^{1+\beta/2}.
\end{equation*}
\end{lemma}

\begin{proof}
Throughout the proof
we will make repeated use of the bound
\begin{equation}
\label{eqn:rho_d_bds}
\abs{d\rho}_{\rho^2 g}
  +\abs{D_{\rho^2 g}^2\rho}_{\rho^2 g}
  +\abs{D_{\rho^2 g}^3\rho}_{\rho^2 g}
\leq
C\rho,
\end{equation}
which holds
(with $D_{\rho^2 g}$ the Levi-Civita connection 
under $\rho^2g$)
pointwise on $\Sigma$
for some universal $C>0$
(independent of $N$, $m$, $\vec{\zeta}$, and $\vec{\xi}\,$)
and follows from
\eqref{eqn:rho_on_cat},
\eqref{eqn:rho_off_cat},
Lemma \ref{lem:geo_ests_cat} \ref{itm:geo_ests_cat:met},
and Lemma~\ref{lem:geo_ests_disc} \ref{itm:geo_ests_disc:conf_met}.
In the remainder of the proof
the value of $C>0$ may change from line to line
and is allowed to depend on
$N,C_0,\alpha,\beta$
but can always be chosen independently of $m$.

Now suppose $p \in \Sigma$
and write $B \subset \Sigma$
be the geodesic disc with center $p$
and radius $1$ with respect to the metric
$\rho^2(p)\met{\Sigma}$.
Recalling the definition
\eqref{eqn:rho_def} of $\rho$,
it is clear from the construction
of the initial surfaces
that
inside $\rho(p)B$
each $k$\textsuperscript{th}
covariant derivative of
the second fundamental form
of $\rho(p)\Sigma$
is bounded by some constant $C(k)>0$.
Furthermore,
each $k$\textsuperscript{th}
covariant derivative
of the Riemann curvature tensor of
the metric $\rho^2(p)\auxmet$
is bounded by $C(k)m^{-2}$
for some $C(k)>0$. 
\pagebreak[2]
Since the mean curvature of $\iota_u$ at $p$
depends smoothly on the derivatives of $u$
at $p$ up to order two,
by scaling we obtain the bound
\begin{equation*}
 \nm[\Big]{\rho^{-1}(p)Q^{\vec{\zeta},\vec{\xi}}_u}_{
  C^{0,\alpha}(
    B,
    \rho^2(p)\met{\Sigma}
  )
}
\leq
  C\nm{\rho(p)u}_{C^{2,\alpha}(B,\rho^2(p)\met{\Sigma})}^2   
\end{equation*}
(since the mean curvature of $\iota_u$
under ambient metric $\rho^2(p)(dx^2+dy^2+dz^2)$
is $\rho^{-1}(p)H_u$,
while the defining function to represent
$\iota_u$ as a graph over $\Sigma$
under the ambient metric $\rho^2(p)\auxmet$
is not $u$ but $\rho(p)u$).

From the preceding displayed inequality
we further estimate
\begin{equation*}
\nm[\Big]{
      \rho^{-2}Q^{\vec{\zeta},\vec{\xi}}_u
    }_{C^{0,\alpha}(B,\rho^2\met{\Sigma})}
\leq
C\rho(p)\nm{u}_{C^{2,\alpha}(B,\rho^2\met{\Sigma})}^2
\leq
C m^{2\beta} \tau_1^2 \rho^{1-2\beta}(p),
\end{equation*}
using \eqref{eqn:rho_d_bds}
and for the final inequality also
the assumed bound on $u$
in the $\nm{\cdot}_{2,\alpha,\beta}$ norm,
as defined in \eqref{eqn:global_norms}.
Recalling also the definition
of the $\nm{\cdot}'_{\alpha,\beta}$ norm
in \eqref{eqn:global_norms}
and using \eqref{eqn:rho_d_bds} again,
we obtain in turn
\begin{equation*}
\nm[\Big]{
      \rho^{-2}Q^{\vec{\zeta},\vec{\xi}}_u
    }'_{\alpha,\beta}
\leq
Cm^{2\beta}\tau_1^2
  [
    m^2 \cdot m^{1-2\beta}
    +\sup_\Sigma
      (m^{-\beta}\rho^\beta \cdot \rho^{1-2\beta})
  ],
\end{equation*}
where for the first term in brackets,
which arises from the second term in the
definition of $\nm{\cdot}'_{\alpha,\beta}$,
we have used
\eqref{eqn:rho_off_cat}.
On the other hand,
using also \eqref{eqn:rho_on_cat}
(and \eqref{eqn:tau_size}),
we have the global bound
$\rho^{1-\beta} \leq C\tau_1^{\beta-1}$,
so, continuing the preceding displayed estimate,
\begin{equation*}
\nm[\Big]{
      \rho^{-2}Q^{\vec{\zeta},\vec{\xi}}_u
    }'_{\alpha,\beta}
\leq
C
  (
    m^3\tau_1^2
    + m^\beta \tau_1^{1+\beta}
  )
\leq
Ce^{2c}\taubar_1^{1+\beta/2}
  (
    m^3 \tau_1^{1-\beta/2}
    +m^\beta \tau_1^{\beta/2}
  ),
\end{equation*}
having used \eqref{eqn:tau_def}
for the final inequality.
The proof is now completed
by invoking \eqref{eqn:tau_size}
and choosing $m$ sufficiently large in terms of
the data $N$, $c$, $\beta$, and $C_0$.
\end{proof}

The next lemma
will enable us to apply
Lemma \ref{lem:coker_control}
to control the mean curvature
in the main existence result
(Proposition \ref{pro:existence_of_fbms})
below.
We refer the reader to
\eqref{eqn:force_def}
for the definition
of the vertical force $\force_i$
on the region $D_i$
of a given initial surface
$\Sigma=\Sigma_{N,m,\vec{\zeta},\vec{\xi}}$.
Given also
$u \in C^2_\G(\Sigma)$
sufficiently small
that $H_u$ is well-defined,
we correspondingly define the perturbed force
\begin{equation}
\label{eqn:perturbed_force_def}
\force_{i;u}
=
\force_{i;u,N,m,\vec{\zeta},\vec{\xi},u}
\vcentcolon=
\int_{D_i}
    \met{\R^3}(\nu^\Sigma, \partial_z|_\Sigma)H_u \,
      \hausint{2}{\met{\R^3}}.
\end{equation}
Note that we are integrating the perturbed mean curvature
over the initial surface
with its unit normal
rather than over the perturbed surface
with its own unit normal,
so $\force_{i;u}$
is not exactly the vertical force
through $\iota_u(D_i)$;
we work with this quantity
because we find it slightly
more convenient, compared to the exact
vertical force,
for the purpose
of monitoring the $\vec{w}$
component of the perturbed mean curvature.
As we do for the $\force_i$ in Appendix \ref{app:forces},
we bundle the $\force_{i;u}$ into a vector
$\vec{\force}_u=\vec{\force}_{u,N,m,\vec{\zeta},\vec{\xi}}$
in the obvious way.

\begin{lemma}[Perturbed forces]
\label{lem:perturbed_force}
Let $N \geq 2$ be a given integer
and $C_0,c>0$ and
$\alpha,\beta \in \interval{0,1}$
given reals.
There exist
$m_0=m_0(N,C_0,c,\alpha,\beta)>0$
and
$C=C(N,C_0)>0$
such that for every integer $m \geq m_0$, all
$
 \vec{\zeta}
 \in
 \IntervaL{-c,c}^{N-1}_{\updownarrow+}
$,
$
 \vec{\xi}
 \in
 \IntervaL{-c,c}^{N-1}_{\updownarrow-}
$,
and any $\G$-equivariant
real-valued function
$u$ on
$\Sigma=\Sigma_{N, m, \vec{\zeta},\vec{\xi}}$
with $\nm{u}_{2,\alpha,\beta} \leq C_0\tau_1$
and $\eta^\Sigma u = u$
we have
\begin{equation*}
 \vec{\force}_{u,N,m,\vec{\zeta},\vec{\xi}}
 \in
 \R^N_{\updownarrow-}
\quad \mbox{and} \quad
\abs{
  \force_{i;u}
  -
  \force_i
}
<
C\tau_1
\quad (1 \leq i \leq N).
\end{equation*}
Furthermore,
$\vec{\force}_{P_{\vec{\zeta},\vec{\xi}}^{-1*}v}$
depends continuously on $(v,\vec{\zeta},\vec{\xi}\,)$.
\end{lemma}

\begin{proof}
The symmetry and continuity assertions
follow easily from the definitions and assumptions,
so we focus on the estimate.
We define $f \colon \Sigma \to \R$
by
\begin{equation*}
f \vcentcolon=
\met{\R^3}(\nu^\Sigma,\partial_z|_\Sigma).
\end{equation*}
Obviously
$\abs{f} \leq 1$,
and it is easy to see from the construction
of $\Sigma$
that we may assume
area and perimeter bounds
for $D_i$ independent
of $m$ and the parameters
$\vec{\zeta}, \vec{\xi}$.
Using the definition \eqref{eqn:Q_def} of $Q_u$,
we write
\begin{equation*}
H_u=H^\Sigma + J_\Sigma u + Q_u.
\end{equation*}
Of course
\begin{equation*}
\int_{D_i} f H^\Sigma
  \, \hausint{2}{\met{\R^3}}
=
\force_i
\end{equation*}
(by applying the formula for the first
variation of area to
definition \eqref{eqn:force_def},
since $\partial_z$ is Killing),
and Lemma \ref{lem:quadratic_estimate}
delivers the necessary estimate
for the integral associated
to the $Q_u$ term:
\begin{align*}
\abs[\bigg]{\int_{D_i}
f Q_u \, \hausint{2}{\met{\R^3}}}
&\leq
C\int_{D_i \setminus K_i}
  \abs{Q_u}
  \, \hausint{2}{\met{\R^3}}
  +
  C\int_{K_i}
    \abs{Q_u}
    \, \hausint{2}{\met{\R^3}}
\\
&\leq
Cm^2\taubar_1^{1+\beta/2}
  +
  C\taubar_1^{1+\beta/2}
    \int_{K_i}
      \rho^2
      \, \hausint{2}{\met{\R^3}}
\\
&\leq
C(m^2+m)\taubar_1^{1+\beta/2}
\leq
C(m^2+m)e^{(2+\beta)c}\tau_1^{1+\beta/2}
\end{align*}
for some $C>0$ independent of $m$,
recalling the definition
in \eqref{eqn:global_norms}
of the $\nm{\cdot}'_{\alpha,\beta}$ norm
(but not even taking advantage of
the built-in decay)
and appealing to \eqref{eqn:rho_off_cat}
to bound $\rho$ on $D_i \setminus K_i$
and to item \ref{itm:geo_ests_cat:met}
of Lemma \ref{lem:geo_ests_cat}
and \eqref{eqn:catr_init_def}
to estimate the $K_i$ integral;
we thus ensure the term in question
satisfies the asserted bound
by recalling \eqref{eqn:tau_size}
and taking $m$ sufficiently large
in terms of $\beta$, $c$,
and the above $C$.
For the remaining, linear term we integrate by parts to get
\begin{align*}
\int_{D_i} fJ_\Sigma u
    \, \hausint{2}{\met{\R^3}}
&=
\int_{D_i} uJ_\Sigma f
    \, \hausint{2}{\met{\R^3}}
  +\int_{\partial D_i}
    \left( 
      f \eta^{D_i} u
      - u \eta^{D_i} f
    \right)
    \, \hausint{1}{\met{\R^3}},
\end{align*}
with $\eta^{D_i}$
the outward unit conormal on $\partial D_i$.

Considering first the boundary integral,
we make the decomposition
$
\partial D_i
=
(\partial D_i \cap \Sp^2)
  \cup
  (\partial D_i \setminus \Sp^2)
$,
where the first set 
is equivalently $\partial D_i \cap \partial \Sigma$,
while the second set consists
of $2m$ (for $2 \leq i \leq N$)
or $m$ (for $i \in \{1,N\}$)
approximate catenoidal half waists.
Using the bound and boundary condition
assumed on $u$,
we estimate
\begin{equation*}
\abs[\bigg]{\int_{\partial D_i \cap \Sp^2}
  f\eta^{D_i}u \, \hausint{1}{\met{\R^3}}}
=
\abs[\bigg]{\int_{\partial D_i \cap \Sp^2}
  fu \, \hausint{1}{\met{\R^3}}}
\leq
C\tau_1\hausmeas{1}{\met{\R^3}}(\partial\Sigma)
\leq
C\tau_1.
\end{equation*}
Since $\Sigma$ meets $\Sp^2$ orthogonally
and $\partial_z$ is parallel in $\R^3$, we have
\begin{equation*}
\eta^\Sigma f
=-\twoff{\Sigma}(\eta^\Sigma,\eta^\Sigma)z|_{\Sigma},
\end{equation*}
so, using also the bound
$\abs{\twoff{\Sigma}}_{\met{\Sigma}} \leq C\rho$
(implied by the stronger estimates in
Lemma \ref{lem:geo_ests_cat} \ref{itm:geo_ests_cat:twoff}
and
Lemma \ref{lem:geo_ests_disc} \ref{itm:geo_ests_disc:twoff}),
\begin{equation*}
\abs[\bigg]{
  \int_{\partial D_i \cap \Sp^2}
  u \eta^{D_i}f \, \hausint{1}{\met{\R^3}}
}
\leq
C\tau_1 \int_{\partial \Sigma}
  \abs{z|_\Sigma}\rho
  \, \hausint{1}{\met{\R^3}}
\leq
Cm^2\tau_1^2
\leq
\tau_1,
\end{equation*}
having used \eqref{eqn:tau_size}
for the final inequality
and \eqref{eqn:height_bound},
\eqref{eqn:rho_off_cat},
\eqref{eqn:rho_on_cat},
and item \ref{itm:geo_ests_cat:met}
of Lemma \ref{lem:geo_ests_cat}
for the integral estimate.

Turning to $\partial D_i \setminus \Sp^2$, we have 
$\hausmeas{1}{\met{\R^3}}(\partial D_i \setminus \Sp^2)\leq Cm\tau_1$ 
and on $D_i \setminus \Sp^2$: 
\begin{align*}
\abs{u}
&\leq Cm^\beta \tau_1^{1+\beta},
& 
\abs{\eta^{D_i}u}
&\leq Cm^\beta\tau_1^\beta\leq C,
\\
\abs{f}
&\leq C\max_j (\tau_j + \abs{h^K_j})\leq Cm\tau_1,
&
\abs{df}
&\leq  C\abs{\twoff{\Sigma}}_{\met{\Sigma}} \leq  C\tau_1^{-1}
\end{align*}
from the bound assumed on $u$
and the fact that each component
of $\partial D_i \setminus \Sp^2$
is the image under the map $\Phi$,
as defined in \eqref{def:Phi},
of a half waist of a catenoid
with vertical axis,
having also used
\eqref{eqn:height_bound}
and \eqref{eqn:tau_size}
for the final bound on $\abs{f}$.
Consequently,
\begin{equation*}
\abs[\bigg]{
  \int_{\partial D_i \setminus \Sp^2}
    (
      f \eta^{D_i} u
      -u \eta^{D_i} f
    ) 
    \, \hausint{1}{\met{\R^3}}
}
\leq
Cm\tau_1
  (
    m\tau_1
    +m^\beta \tau_1^\beta
  )
\leq
C\tau_1,
\end{equation*}
once again using \eqref{eqn:tau_size}.

Finally, for the bulk term,
again exploiting the fact
that $\partial_z$ is Killing
and denoting its projection
onto $\Sigma$ by $\partial_z^\top$,
we find
\begin{align*}
\abs[\bigg]{
  \int_{D_i} uJ_\Sigma f
    \, \hausint{2}{\met{\R^3}}}
&=\abs[\bigg]{
  \int_{D_i} (\partial_z^\top H^\Sigma)u
    \, \hausint{2}{\met{\R^3}}}
\\
&\leq
\int_{D_i}
  \abs{u} \, \rho
  \, \abs{dH}_{\rho^2\met{\Sigma}}
  \, \hausint{2}{\met{\R^3}}
\leq
C\tau_1 m^\beta
  \int_{D_i}
    \rho^{1-\beta}
    \abs{dH}_{\rho^2\met{\Sigma}}
    \, \hausint{2}{\met{\R^3}}.
\end{align*}
Using item \ref{itm:geo_ests_cat:mc}
of Lemma \ref{lem:geo_ests_cat}
and item \ref{itm:geo_ests_disc:mc}
of Lemma \ref{lem:geo_ests_disc}
to estimate $dH$,
along with \eqref{eqn:rho_on_cat}
and \eqref{eqn:rho_off_cat},
we have
\begin{align*}
\int_{D_i}
    \rho^{1-\beta}
    \abs{dH}_{\rho^2\met{\Sigma}}
    \,
    \hausint{2}{\met{\R^3}}
&=
  \int_{D_i \setminus K_i} 
    \rho^{1-\beta}
    \abs{dH}_{\rho^2\met{\Sigma}}
    \,
    \hausint{2}{\met{\R^3}}
  \\
  &\hphantom{{}={}}
  +
  \int_{K_i} 
    \rho^{1-\beta}
    \abs{dH}_{\rho^2\met{\Sigma}}
    \,
    \hausint{2}{\met{\R^3}}
\\
&\leq
Cm^2
  \Bigl(
    \tau_1
    +\abs{\disloc_i}
      \nm{\wbar_i}_{C^1(\rho^2 \met{\Sigma})}
  \Bigr)
    \cdot \hausmeas{2}{\met{\R^3}}(D_i \setminus K_i)
  \\
  &\hphantom{{}={}}
  +
    C\tau_1 \int_{K_i}
    \rho^{2-\beta}
    \,
    \hausint{2}{\met{\R^3}}
\\
&\leq
Cm^2\tau_1(1+c)
  +C\beta^{-1}m^\beta\tau_1,
\end{align*}
having recalled also
the definitions \eqref{eqn:disloc_def}
of $\disloc_i$
and \eqref{eqn:wbar_def} of $\wbar_i$
and having made use of
item \ref{itm:geo_ests_cat:met}
of Lemma \ref{lem:geo_ests_cat}
to estimate the $K_i$ integral.

Feeding this last estimate into the one preceding it,
we conclude
\begin{equation*}
\abs[\bigg]{\int_{D_i} uJ_\Sigma f \, \hausint{2}{\met{\R^3}}}
\leq C[
  (1+c)m^{2+\beta}\tau_1
  +\beta^{-1}m^{2\beta}\tau_1
]\tau_1.
\end{equation*}
By invoking \eqref{eqn:tau_size}
and taking $m$ sufficiently large
in terms of $c$, $\beta^{-1}$,
and the final $C$ above
we confirm the asserted bound
and so complete the proof.
\end{proof}

\paragraph{On the construction of free boundary minimal stackings of any order -- main results.} 
We shall now exploit and combine the preliminary results above to prove our main existence result. 
The actual embeddedness of the surfaces, and the characterization of their maximal symmetry group are postponed to a separate statement.

\begin{proposition}[Existence of free boundary minimal stackings]\label{pro:final_existence_with_estimates}
\label{pro:existence_of_fbms}
Let $N \geq 2$ be an integer
and $\alpha,\beta \in \interval{0,1}$.
There exist
$m_0=m_0(N,\alpha,\beta)>0$
and $C=C(N,\alpha,\beta)>0$
such that for every integer $m>m_0$
there exist an initial surface
$
 \Sigma
 =
 \Sigma_{N, m, \vec{\zeta}, \vec{\xi}}
$
and a function
$u \colon \Sigma \to \R$ 
such that
\begin{equation*}
\tau_1^{-1}\nm{u}_{2,\alpha,\beta}
  +\abs{\vec{\zeta}}
  +\abs{\vec{\xi}}
\leq
C
\end{equation*}
(recalling the definition of $\tau_1$
in \eqref{eqn:tau_def})
and 
$\iota_u \colon \Sigma \to \B^3$
is a free boundary minimal immersion. 
\end{proposition}

\begin{proof}
Let $N \geq 2$ be a given integer,
and fix $\alpha,\beta \in \interval{0,1}$.
For any
$\vec{\zeta} \in \R^{N-1}_{\updownarrow+}$
and
$\vec{\xi} \in \R^{N-1}_{\updownarrow-}$
by virtue of
Proposition \ref{pro:global_inverse}
and Lemma \ref{lem:initial_mc}
we can take $m$ large enough in terms
of $N$ and $\abs{\vec{\zeta}}+\abs{\vec{\xi}}$
that
\begin{equation*}
\bigl(u^{(1)},\vec{\mu}^{(1)},\vec{\mubar}^{(1)}\bigr)
\vcentcolon=
R_{\Sigma_{N,m,\vec{\zeta},\vec{\xi}}}
  \Big(
    -\rho^{-2}H^{\Sigma}
      +\vec{\disloc} \cdot \vec{\wbar}
  \Big)
\end{equation*}
is well-defined and satisfies
\begin{equation*}
\nm[\big]{u^{(1)}}_{2,\alpha,\beta}+\abs[\big]{\vec{\mu}^{(1)}}+\abs[\big]{\vec{\mubar}^{(1)}}
\leq C\tau_1
\end{equation*}
for some $C>0$ depending on just $N$. 
Next set
\begin{equation*}
B_{N,m}
\vcentcolon=
\left\{
  v
  \in
  C^{2,\alpha/2}_\G(\Sigma_{N,m,\vec{0},\vec{0}})
  \st
  \nm{v}_{2,\alpha,\beta} \leq \taubar^{1+\beta/3}
\right\}.
\end{equation*}
Then
for any $c>0$
and for all $m$ sufficiently large in terms of $c$
and $N$
we can define the map
\begin{equation*}
S_{N,m,c}
\colon
B_{N,m}
    \times
    \Bigl(
      \IntervaL{-c,c}^{N-1}_{\updownarrow+}
      \times
      \IntervaL{-c,c}^{N-1}_{\updownarrow-}
    \Bigr)
\to
C_\G^{2,\alpha}(\Sigma_{N,m,\vec{0},\vec{0}})
  \times
  \Bigl(
    \R^{N-1}_{\updownarrow+}
    \times
    \R^{N-1}_{\updownarrow-}
  \Bigr)
\end{equation*}
as follows,
where we use in particular the map $\ptofd$
identified in Lemma \ref{lem:coker_control}
in Appendix \ref{app:forces}
(and the reader can consult the cross-references
in the paragraph immediately following the below definition
of $S_{N,m,c}$ for reminders of the remaining notation employed): 
\begin{align*}
S_{N,m,c}\Bigl(v,(\vec{\zeta},\vec{\xi}\,)\Bigr)
&\vcentcolon=
\left(
  P_{\vec{\zeta},\vec{\xi}\,}^*u_v, ~
  \bigl(\vec{\zeta},\vec{\xi}\,\bigr)
  -\tau_1^{-1}\ptofd^{-1}
    \Bigl(
      \vec{\force}_{u^{(1)}+P_{\vec{\zeta},\vec{\xi}}^{-1*}v},~
      \vec{\disloc}+\vec{\mubar}^{(1)}+\vec{\mubar}_v
    \Bigr)
\right)
\shortintertext{with}
(u_v,\vec{\mu}_v,\vec{\mubar}_v)
  &\vcentcolon=
  R_{\Sigma_{N, m, \vec{\zeta},\vec{\xi}}}
    \biggl(
      -\rho^{-2}
        Q^{\vec{\zeta},\vec{\xi}}_{
          u^{(1)}+P_{\vec{\zeta},\vec{\xi}}^{-1*}v}
    \biggr).
\end{align*}

Crucially, 
the estimates we have for
$R_\Sigma$ (Lemma \ref{pro:global_inverse}),
$Q$ (Lemma \ref{lem:quadratic_estimate}),
and the second component
of the value of $S_{N,m,c}$
(Lemma \ref{lem:coker_control}
and Lemma \ref{lem:perturbed_force})
do not depend on the choice of $c$.
The estimates \eqref{eqn:diffeo_pullback_ests}
we have concerning $P_{\vec{\zeta},\vec{\xi}}$
do depend on $c$,
but for any given $c>0$ we may choose $m$
sufficiently large in terms of it
and (exploiting the definition of $B_{N,m}$
and the quadratic estimate
Lemma \ref{lem:quadratic_estimate})
sacrifice an appropriately small power
of $\taubar_1$
to secure a bound on the image of $S_{N,m,c}$
independent of $c$:
in summary we obtain
\begin{align*}
\nm{P_{\vec{\zeta},\vec{\xi}}^*u_v}_{2,\alpha,\beta}
  +\abs{\vec{\mu}_v}
  +\abs{\vec{\mubar}_v}
&\leq
C(c)\taubar_1^{1+\beta/2}
\leq
\taubar_1^{1+\beta/3},
\\
\abs[\Big]{
\bigl(\vec{\zeta},\vec{\xi}\,\bigr)
  -\tau_1^{-1}\ptofd^{-1}
    \Bigl(
      \vec{\force}_{u^{(1)}+P_{\vec{\zeta},\vec{\xi}}^{-1*}v},~
      \vec{\disloc}+\vec{\mubar}^{(1)}+\vec{\mubar}_v
    \Bigr)
}
&\leq C +C\tau_1^{-1} \abs{\vec{\mubar}^{(1)}+\vec{\mubar}_v}
\leq C,
\end{align*}
where $C(c)>0$ is a constant independent of $m$,
while each instance of $C$ is a positive real,
whose value may vary from instance to instance
but in each can be chosen independently
of $m$, $c$, $\vec{\zeta}$, $\vec{\xi}$ and~$v$. 
Consequently there exists some value $c>0$
for which the image of
the continuous map $S_{N,m,c}$
is contained in its domain.
We can now apply the Schauder fixed point theorem
(as in \cite[Theorem~11.1]{GilTru2001}
for instance)
to verify the existence of some
$v,\vec{\zeta},\vec{\xi}$ for which
$S\bigl(v,(\vec{\zeta},\vec{\xi}\,)\bigr)
=\bigl(v,(\vec{\zeta},\vec{\xi})\bigr)$.
The proof will be concluded (for $m$ sufficiently large)
with initial surface
$\Sigma=\Sigma_{N,m,\vec{\zeta},\vec{\xi}}$
for this choice
of $\vec{\zeta},\vec{\xi}$ (depending on $m$)
and with defining function
\[
u \vcentcolon=u^{(1)}+P_{\vec{\zeta},\vec{\xi}}^{-1*}v.
\]
Indeed,
starting with the definition \eqref{eqn:Q_def} of $Q_\cdot$,
we then have
\begin{align*}
H_u
&=
H^\Sigma + J_\Sigma u + Q_u
\\
&=
H^\Sigma
  - H^\Sigma
  +\rho^2\vec{\disloc} \cdot \vec{\wbar}
  +\rho^2 \vec{\mu}^{(1)} \cdot \vec{w}
  +\rho^2 \vec{\mubar}^{(1)} \cdot \vec{\wbar}
  +J_\Sigma P_{\vec{\zeta},\vec{\xi}}^{-1*}v
  +Q_u
\\
&=
\rho^2
  (
    \vec{\mu}^{(1)}+\vec{\mu}_v
  ) \cdot \vec{w}
  +
  \rho^2(
    \vec{\disloc}
    +\vec{\mubar}^{(1)}
    +\vec{\mubar}_v
  ) \cdot \vec{\wbar}.
\end{align*}
Here we have exploited just the first,
functional component of the fixed-point equation, 
which shows that $H_u$ lies in a finite-dimensional subspace.
The second, parametric component of the fixed-point equation implies
\begin{align}
\label{eqn:fpc_p}
\vec{\force}_u &= \vec{0},
&
\vec{\disloc}
  +\vec{\mubar}^{(1)}
  +\vec{\mubar}_v
&=0,
\end{align}
so that in fact
\begin{equation}
\label{eqn:H_in_w}
H_u =  \rho^2 (\vec{\mu}^{(1)}+\vec{\mu}_v ) \cdot \vec{w}.
\end{equation}
In this way the fixed-point condition
directly shows that the $\vec{\wbar}$
component of $H_u$ vanishes,
reflecting the transparent
influence that variation of $\vec{\xi}$
(via $\vec{\disloc}$)
exerts on this component.
On the other hand,
using the definition
\eqref{eqn:perturbed_force_def}
of $\vec{\force}_u$, for each $1 \leq i \leq N$
equation \eqref{eqn:H_in_w} and the first equation in \eqref{eqn:fpc_p}  yield
\begin{equation*}
0=\force_{i;u}
=\bigl(\mu^{(1)}_i+(\mu_v)_i\bigr)
\int_{D_i} \rho^2 w_i  \, \met{\R^3}(\nu^\Sigma, \, \partial_z|_\Sigma),
\end{equation*}
but, referring to the definition \eqref{eqn:w_ess_def}
of $w_i$, it is clear that the integrand has a sign,
implying $H_u=0$.

That
$
\iota_u(\partial \Sigma)
\subseteq
\Sp^2
$
with radial conormal
follows
(cf. \cite[Section 5]{KapouleasWiygul17}
or \cite[Section 5.1]{CSWnonuniqueness}
for a detailed explanation)
from the fact that these same conditions hold
for $\iota$ in place of $\iota_u$,
the fact that $\Sp^2$
is totally geodesic in $(\R^3,\auxmet)$,
and the boundary condition enforced
by $R_\Sigma$.
Finally, the maximum principle
(or alternatively the bound on $u$)
implies
$\iota_u(\Sigma) \subset \B^3$.
\end{proof}

For each $(N,m)$ as in
Proposition \ref{pro:final_existence_with_estimates}
we take $\Sigma$ and $u$ as in the same proposition,
and we define
\begin{equation*}
\Sigma_{N,m} \vcentcolon= \iota_u(\Sigma).
\end{equation*}

\begin{proposition}
[Embeddedness of $\Sigma_{N,m}$ and maximal symmetry group]
\label{pro:final_embd}
Let $N \geq 2$.
There exists $m_0>0$
such that for every integer $m>m_0$
the surface $\Sigma_{N,m}$ is embedded. 
Moreover, the maximal symmetry group of $\Sigma_{N,m}$ equals the prismatic group $\pri_m$ when $N$ is even, 
and the antiprismatic group $\apr_m$ when $N$ is odd. 
\end{proposition}

\begin{proof}
The embeddedness follows from
the bounds in Proposition \ref{pro:existence_of_fbms}
for the defining function and parameters
(using in particular the decay of the defining
functions on the $K_i$ regions)
and from the embeddedness and specifications
of the initial surfaces
(using in particular the facts
that each catenoid has waist
radius of order $\tau_1$
and that the spacing between the discs
is of order $m\tau_1$).

Turning to the symmetries,
by construction $\Sigma_{N,m}$
is invariant under $\G=\G_{N,m}$,
so it remains only to verify
that $\Sigma_{N,m}$ has no extra symmetries.
For this one could argue, somewhat as for the embeddedness,
by using the smallness of the defining function
and the easily verified
fact that the corresponding initial surface
has maximal symmetry group $\G$,
but we give instead a different proof
that proceeds by counting umbilics
(with multiplicity)
and depends very little on the construction.

Recall that in a tubular neighborhood of
$\{x=y=0\}$
each initial surface is a union
of exact horizontal discs
and $\auxmet$ agrees with the Euclidean metric.
It is then immediate from the construction
of $\Sigma_{N,m}$ as a $\auxmet$ graph
over an initial surface
that $\Sigma_{N,m} \cap \{x=y=0\}$ consists of
exactly $N$ points.
Moreover,
since $\Sigma_{N,m}$ is invariant under $\pyr_m$,
each of these points is
(by Lemma \ref{lem:umbilics_fixed_by_rots})
an interior umbilic of order at least $m-2$.

Suppose now that $\Sigma_{N,m}$
had a symmetry outside $\G$.
Note that whenever $N$ is odd,
$\Sigma_{N,m}$ contains the origin
and has horizontal tangent plane there.
Consequently, if our extra symmetry
does not preserve
$\{x=y=0\}$ as a set,
then $N$ must be even
and $\Sigma_{N,m}$
must contain at least
$(m+1)N$ interior umbilics
each of which has order at least $m-2$ due to Lemma \ref{lem:umbilics_fixed_by_rots},
so that by Proposition~\ref{pro:umbilic_count}
\begin{equation}
\label{eqn:z-axis_not_preserved}
8\gamma+4\beta-8
\geq
2(m-2)(m+1)N.
\end{equation}
If instead our extra symmetry preserves the set $\{x=y=0\}$,
then each point $\{x=y=0\} \cap \Sigma_{N,m}$
must in fact have order at least $2m-2$. (Indeed, this claim follows from the explicit classification of the elements in $\Ogroup(3)$, which allows -- possibly by composing with suitable elements of $\G$ -- to reduce to the case when the extra symmetry is a pure rotation around the $z$-axis, and then noting that since $\Sigma_{N,m}$ is not rotationally symmetric, such a symmetry must rotate by an integer divisor of $2\pi/m$; at that stage one simply appeals to Lemma \ref{lem:umbilics_fixed_by_rots} with $2m$ in lieu of $k$.)
In turn Proposition~\ref{pro:umbilic_count} implies
\begin{equation}
\label{eqn:z-axis_preserved}
8\gamma+4\beta-8 \geq 4(m-1)N.
\end{equation}
On the other hand, $\Sigma_{N,m}$ has the same topological type
as its corresponding initial surface,
so by Lemma \ref{lem:TopologicalType}
\begin{equation*}
8\gamma + 4\beta - 8
=
4(m-1)N-4m
\end{equation*}
(independently of the parity of $N$).
However, it is then readily seen that the preceding equation is incompatible with \eqref{eqn:z-axis_preserved}, as it is in fact with 
\eqref{eqn:z-axis_not_preserved}
provided $m \geq 3$.
\end{proof}

By virtue of the embeddedness in Proposition \ref{pro:final_embd} we can define the diffeomorphism
\begin{equation*}
\varpi_{N,m} \colon \Sigma_{N,m} \to \Sigma
\end{equation*}
such that $\varpi_{N,m} \circ \iota_u$
is the identity on $\Sigma$. That said, we wish to collect in the following statement some sharper geometric estimates that -- besides being informative in their own right -- are auxiliary to the study we perform in the next section and the resulting index estimates.

\begin{corollary}
[Geometric estimates for $\Sigma_{N,m}$]
\label{cor:final_geo}
Let $N \geq 2$.
There exist $\beta \in \interval{0,1}$ and $m>0$
such that for every integer $m>m_0$,
recalling the function $\rho$
defined by \eqref{eqn:rho_def}
on the initial surface
$
 \Sigma
 \vcentcolon=
 \varpi_{N,m}(\Sigma_{N,m})
$,
we have
\begin{align*}
\nm[\Big]{
    \rho^2\varpi_{N,m}^{-1*}\met{\Sigma_{N,m}}
    -\rho^2\met{\Sigma}
  }_{C^0(\Sigma, \rho^2\met{\Sigma})}
&\leq
  (m\tau_1)^\beta,
\\[1ex]
\nm[\Big]{
  \rho^{-2} \varpi_{N,m}^{-1*}
      \abs[\big]{
        \twoff{\Sigma_{N,m}}
      }_{\met{\Sigma_{N,m}}}^2
  -\rho^{-2}
      \abs[\big]{
        \twoff{\Sigma}
      }_{\met{\Sigma}}^2
}_{C^0(\Sigma, \rho^2\met{\Sigma})}
&\leq
  (m\tau_1)^\beta.
\end{align*}
\end{corollary}

\begin{proof}
This follows from the estimate
for the defining function
in Proposition \ref{pro:existence_of_fbms}
and from the estimates on the second fundamental form
of the initial surfaces
in Lemma \ref{lem:geo_ests_cat}
and Lemma \ref{lem:geo_ests_disc}.
\end{proof}

\section{Morse index and nullity bounds}\label{sec:Index}

We will partition each surface $\Sigma_{N,m}$ constructed above (cf. Proposition \ref{pro:final_existence_with_estimates} and Proposition \ref{pro:final_embd}) into three type of regions -- still to be formally introduced -- one category being that of ``catenoid regions'' and the other two being those of the ``disc regions'' and finally the ``intermediate regions'' in between. 
We are going to apply the Montiel--Ros type machinery, as developed in \cite{CSWSpectral}, so in the first part of this section we will collect some key facts from there, also with the goal of introducing the (somewhat heavy but) convenient notation we shall later employ. We will then proceed with the spectral estimates on such building blocks, and finally prove the global index and nullity bounds on the free boundary minimal disc stackings.

\paragraph{General setup.} We start by considering a Lipschitz domain $\lipdom$ 
of a smooth, compact $d$-dimensional manifold $M$
with (possibly empty) boundary $\partial M$;
by this we mean a nonempty, open subset of $M$
whose boundary is locally (around every point) representable
as the graph of a Lipschitz function, and note that
we do not require $\Omega$ to be connected. One wants to study the spectrum of a given Schr\"{o}dinger operator $\Delta_g + q$ on $\lipdom$ subject to Dirichlet, Neumann, and Robin conditions on pairwise disjoint open subsets $\dbdy \lipdom$, $\nbdy \lipdom$, $\rbdy \lipdom$ respectively, assuming that the union of their closures equals $\partial\lipdom$. We assume the Riemannian metric $g$, the Schr\"{o}dinger potential $\potential \colon \closure{\lipdom} \to \R$ and the Robin potential $\robinpotential \colon \closure{\lipdom} \to \R$ to be smooth (i.\,e., $C^{\infty}$) unless otherwise stated.

One can pose the problem by considering the Sobolev space
$\sob(\lipdom,g)$, with its standard inner product, then its subspace 
\begin{equation}
\label{eqn:H10}
\sobd{\dbdy \lipdom}(\lipdom,g)
\vcentcolon=
\{
  u \in \sob(\lipdom,g)
  \st
  u|_{\dbdy \lipdom}=0
\},
\end{equation}
understood in the sense of traces, and the quadratic form 
$\weakbf=\weakbf[\lipdom,g,\potential,\robinpotential,\dbdy\lipdom,\nbdy\lipdom,\rbdy\lipdom]$ defined by 
\begin{equation}
\label{eqn:bilinear_form_def}
\begin{aligned}
\weakbf\colon\sobd{\dbdy \lipdom}(\lipdom,g)
\times\sobd{\dbdy \lipdom}(\lipdom,g)
&\to\R
\\
(u,v)&\mapsto
  \int_\lipdom
    \Bigl(
      g(\nabla_g u, \nabla_g v)
      -\potential u v
    \Bigr)
    \, \hausint{d}{g}
  -\int_{\rbdy \lipdom}
    \robinpotential u v
    \, \hausint{d-1}{g}
\end{aligned}
\end{equation}
that is symmetric, bounded, coercive, and can thus be diagonalized (see \cite[Section 2]{CSWSpectral}).
Listing the eigenvalues as usual with repetitions in non-decreasing order, we denote the $k\textsuperscript{th}$ eigenvalue of $\weakbf$ by $\lambda_k(\weakbf)$.
Given $t\in\R$, we let $E^{<t}(\weakbf)$ respectively $E^{=t}(\weakbf)$ denote the span in $L^2(\lipdom,g)$ of all eigenfunctions of $\weakbf$ with eigenvalue $\lambda<t$ respectively $\lambda=t$, and $E^{\leq t}(\weakbf)$ be their direct sum. The \emph{index} and the \emph{nullity} of $\Sigma$ 
are then defined as 
\begin{align*}
\ind(\weakbf)\vcentcolon&=\dim E^{<0}(\weakbf), &
\nul(\weakbf)\vcentcolon&=\dim E^{=0}(\weakbf).
\end{align*}

Let $\grp$ be a finite group of smooth diffeomorphisms of $M$, each restricting to a surjective isometry of $(\closure{\lipdom},g)$. 
Given a group homomorphism 
$
 \twsthom
 \colon
 \grp
 \to
 \Ogroup(1)=\{-1,1\}
$
we define the action
\begin{equation*}
(\phi,u)
\mapsto
\twsthom(\phi)(u \circ \phi^{-1})
=
\twsthom(\phi)\phi^{-1*}u
\quad
\mbox{for all }
\phi \in \grp, ~
u \colon \lipdom \to \R,
\end{equation*}
and we call a function
$(\grp,\twsthom)$-invariant
if it is invariant under this action. As defined in \cite[(2.11)]{CSWSpectral}, we denote the orthogonal projection onto $(G,\twsthom)$-invariant functions by $\pi_{G,\twsthom}\colon L^2(\Omega,g)\to L^2(\Omega,g)$, and define the subspaces $\eigenspsym{\weakbf}{<t}{\grp}{\twsthom}, \eigenspsym{\weakbf}{=t}{\grp}{\twsthom}$ (as well as their direct sum $\eigenspsym{\weakbf}{\leq t}{\grp}{\twsthom})$ and in turn the integers 
\begin{align}
\label{eqn:symind}
\symind{G}{\twsthom}(T)
&\vcentcolon=\dim E_{G,\twsthom}^{<0}(T),
&
\symnul{G}{\twsthom}(T)
&\vcentcolon=\dim E_{G,\twsthom}^{=0}(T).
\end{align}
Of special utility, in our study, is the min-max characterization
\begin{equation}
\label{eqn:symminmax}
\eigenvalsym{\weakbf}{i}{\grp}{\twsthom}
=
\min
  \left\{
    \max
      \biggl\{
        \frac{\weakbf(w,w)}
        {\nm{w}_{L^2(\lipdom,g)}^2}
\st
0\neq w \in W
    \biggr\}
\st
    W
    \subspace
    \invproj{\grp}{\twsthom}\bigl(\sobd{\dbdy \lipdom}(\lipdom,g)\bigr),
    ~
    \dim W = i
\right\}.
\end{equation}

When $\abs{G}=2$, there are precisely two homomorphisms $G\to\Ogroup(1)$. 
This leads to defining the $G$-even and $G$-odd index (henceforth denoted $\symind{G}{+}$ and $\symind{G}{-}$ respectively), and likewise for nullity; we always have (cf. \cite[Example 2.3]{CSWSpectral})
\begin{align}
\label{evenOddDecomp}
\ind(\weakbf)
&=\symind{G}{+}({\weakbf})
+\symind{G}{-}(\weakbf), 
&
\nul(\weakbf)
&=\symnul{G}{+}(\weakbf)
+\symnul{G}{-}(\weakbf).
\end{align}  

\paragraph{Adjoining interior boundary conditions and main statement.}

If $\lipdom_1 \subset \lipdom$ is another Lipschitz domain of $M$ we define the sets 
\begin{equation}
\label{eqn:subdomainbdydecomps}
\begin{aligned}
\intbdy\lipdom_1
  &\vcentcolon=
  \partial\lipdom_1 \cap \lipdom,
\qquad
&\extbdy\lipdom_1
  &\vcentcolon=
  \partial \lipdom_1 
    \setminus
    \closure{\intbdy\lipdom_1},
\\[1ex]
\dint{\dbdy}\lipdom_1
  &\vcentcolon=
  (\extbdy\lipdom_1 \cap \dbdy\lipdom)
    \cup \intbdy\lipdom_1,
\qquad
&\nint{\dbdy}\lipdom_1
  &\vcentcolon=
  \extbdy\lipdom_1 \cap \dbdy\lipdom,
\\[1ex]
\dint{\nbdy}\lipdom_1
  &\vcentcolon=
  \extbdy\lipdom_1 \cap \nbdy\lipdom,
\qquad
&\nint{\nbdy}\lipdom_1
  &\vcentcolon=
  (\extbdy\lipdom_1 \cap \nbdy\lipdom)
    \cup \intbdy\lipdom_1,
\\[1ex]
\dint{\rbdy}\lipdom_1
  &\vcentcolon=
  \extbdy\lipdom_1 \cap \rbdy\lipdom,
\qquad
&\nint{\rbdy}\lipdom_1
  &\vcentcolon=
  \extbdy\lipdom_1 \cap \rbdy\lipdom.
\end{aligned}
\end{equation}
and the associated bilinear forms
\begin{equation}
\label{eqn:DirAndNeumInternalizationsOfBilinearForm}
\begin{aligned}
\dint{\weakbf}_{\lipdom_1}
  &\vcentcolon=
  \weakbf[
    \lipdom_1,g,\potential,\robinpotential,
    \dint{\dbdy}\lipdom_1,
    \dint{\nbdy}\lipdom_1,
    \dint{\rbdy}\lipdom_1
  ],
\\[1ex]
\nint{\weakbf}_{\lipdom_1}
  &\vcentcolon=
  \weakbf[
    \lipdom_1,g,\potential,\robinpotential,
    \nint{\dbdy}\lipdom_1,
    \nint{\nbdy}\lipdom_1,
    \nint{\rbdy}\lipdom_1
  ],
\end{aligned}
\end{equation}
defined, respectively, on the Sobolev spaces
$\sobd{\dint{\dbdy}\lipdom_1}(\lipdom_1,g)$
and $\sobd{\nint{\dbdy}\lipdom_1}(\lipdom_1,g)$.

Note that, if we further assume that each element of $\grp$ maps $\lipdom_1$ onto itself then it follows that each of the sets in \eqref{eqn:subdomainbdydecomps} is also invariant under the action of $\grp$.

\begin{proposition}
[\protect{\cite[Proposition 3.1]{CSWSpectral}}]
\label{prop:mr}
In the setting above, suppose we have open $\grp$-invariant Lipschitz subdomains
$\lipdom_1,\ldots,\lipdom_n \subset \lipdom$
which are pairwise disjoint
and whose closures cover $\closure{\lipdom}$. 
We assume further that $\grp$ acts transitively on the connected components of $\lipdom$ (which is automatically the case if $\lipdom$ is connected).
Then the following inequalities hold
for any $t \in \R$:
\begin{enumerate}[label={\normalfont(\roman*)}]
\item \label{mrLower}
$\displaystyle
 \dim \eigenspsym{\weakbf}{<t}{\grp}{\twsthom}
 \geq
 \dim \eigenspsym
       {\dint{\weakbf}_{\lipdom_1}}
       {<t}{\grp}{\twsthom}
   +\sum_{i=2}^n \dim
     \eigenspsym{\dint{\weakbf}_{\lipdom_i}}
                {\leq t}{\grp}{\twsthom}
$, 
\item \label{mrUpper}
$\displaystyle
 \dim \eigenspsym{\weakbf}{\leq t}{\grp}{\twsthom}
 \leq
 \dim \eigenspsym
      {\nint{\weakbf}_{\lipdom_1}}
      {\leq t}{\grp}{\twsthom}
   +\sum_{i=2}^n \dim
     \eigenspsym{\nint{\weakbf}_{\lipdom_i}}
                {<t}{\grp}{\twsthom}
$.
\end{enumerate}
\end{proposition}

\paragraph{The geometric scenario.} 
For our purposes it is convenient to specialize the discussion above to the geometric case of interest, which in particular warrants special notation. 
Accordingly, let $\B^3$ as usual denote the $3$-dimensional Euclidean unit ball, and let $\Sigma\subset\B^3$ be a properly embdedded (compact) surface. 
In this setting, its Morse index and nullity correspond to the Schr\"odinger operator determined by taking $g$ to be the metric induced by the embedding in $\B^3$, $q=\abs{A}^2$
(the squared norm of the second fundamental form of $\Sigma$) and $r=1$. 
Following the notation of \cite{CSWSpectral}, we define the Jacobi quadratic form 
\begin{align}\label{eqn:Jacobi_quadratic_form}
\begin{aligned}
\indexform{\Sigma}\colon H^1(\Sigma) \times H^1(\Sigma)&\to\R 
\\
(u,v)&\mapsto\int_{\Sigma}
\nabla u\cdot\nabla v-\abs{A}^2uv\,d\hsd^2 
-\int_{\partial\Sigma}u v\,d\hsd^1. 
\end{aligned}
\end{align} 
The (absolute) \emph{index} and the \emph{nullity} of $\Sigma$ 
are then (equivalently) defined as 
\begin{align*}
\ind(\Sigma)&\vcentcolon=\dim E^{<0}(\indexform{\Sigma}),
&
\nul(\Sigma)&\vcentcolon=\dim E^{=0}(\indexform{\Sigma}).
\end{align*} 

Among all possible group actions one can consider, one is especially important for our purposes. 
Recalling that any properly embedded surface in $\B^3$ is two-sided, let us pick a unit normal
$\nu$ on $\Sigma$
and thereby identify -- as usual -- sections of the normal bundle of $\Sigma$
with functions on $\Sigma$.
For $\grp$ again a subgroup of $\Ogroup(3)$ preserving $\Sigma$
we have a natural action given by
\begin{align*}
(\phi,u)
&\mapsto
\nrmlsgn{\nu}(\phi)
  (u \circ \phi^{-1})
\quad
\text{ for all }
\phi \in \grp,~
u \colon \Sigma \to \R,
\end{align*}
where $\nrmlsgn{\nu}(\phi)\vcentcolon=h(\phi_*\nu,\nu)$ is a constant in $\Ogroup(1)=\{1,-1\}$.
We shall further assume
that the action of $\grp$ on $\Sigma$
is faithful, meaning that only the identity
element fixes $\Sigma$ pointwise;
this assumption is always satisfied in our applications. In this context recall we say that a function
$u \colon \Sigma \to \R$
is $\grp$-invariant
if $u=u \circ \phi$
for all $\phi \in \grp$,
and we say rather that
$u$ is $\grp$-equivariant
if $u=\nrmlsgn{\nu}(\phi) u \circ \phi$
for all $\phi \in \grp$
(that is, noting the identity
$
 \nrmlsgn{\nu}(\phi)
 =
 \nrmlsgn{\nu}(\phi^{-1})
$,
provided $u$ is invariant under
the $\nrmlsgn{\nu}$-twisted $\grp$ action). 
The $\grp$-equivariant \emph{index} and the \emph{nullity} of $\Sigma$ are then
\begin{align}\label{eq:MeaningEquivariance}
\equivind{\grp}(\Sigma)
=\equivind{\grp}(\indexform{\Sigma})
&\vcentcolon=
  \symind{\grp}{\nrmlsgn{\nu}}(\indexform{\Sigma}),
&
\equivnul{\grp}(\Sigma)
=\equivnul{\grp}(\indexform{\Sigma})
&\vcentcolon=
  \symnul{\grp}{\nrmlsgn{\nu}}(\indexform{\Sigma});
\end{align}
note that, in the special case when $G$ is the trivial group, this notation is obviously consistent with what we defined above and thus we agree to employ just the former and write $\ind(\Sigma)$ and
$\nul(\Sigma)$ for the (standard, non-equivariant) index and nullity of $\Sigma$. We will also employ this notation, with the obvious changes, for the Jacobi operator of compact, connected, properly embedded minimal surfaces in a given regular subdomain of $\R^3$, subject to suitable boundary conditions specified whenever ambiguity is likely to arise.

\paragraph{Spectral estimates for the building blocks.}  
To apply Proposition \ref{prop:mr} we first need to prove bounds for the subdomains of $\Sigma_{N,m}$ of the three types
previewed at the beginning of this section.
For the following statement,
given one catenoidal ribbon 
\[
\catr_i\vcentcolon=\varpi_{N,m}^{-1}\bigl(K_i(m^{-4})\bigr)
\]
(for $i=1,\ldots,N-1$,
recalling \eqref{eqn:catr_init_def}),
we set
\begin{align*}
Q_{\catr_i}^\dir&=
  \dint{\bigl(Q^{\Sigma_{N,m}}\bigr)}_{\catr_i}, & 
Q_{\catr_i}^\neum&=
  \nint{\bigl(Q^{\Sigma_{N,m}}\bigr)}_{\catr_i},
\end{align*}
and, in addition,
\begin{equation*}
\G_{\catr_i}
\vcentcolon=
\{
  \mathsf{T} \in \G
  \st
  \mathsf{T}\catr_i = \catr_i
\},
\end{equation*}
the subgroup of the symmetry group of $\Sigma_{N,m}$
stabilizing $\catr_i$ as a set. In words, this group only consists of a reflection with respect to the vertical plane of symmetry of the ribbon if $N$ is odd or if $N$ is even and the ribbon in question is entirely contained in the half-space $\left\{z>0\right\}$ or $\left\{z<0\right\}$, and is larger (since it consists of two planar symmetries) if instead $N$ is even and the ribbon in question crosses the $\left\{z=0\right\}$ plane.
Note that, in either case, every element of every $\G_{\catr_i}$ preserves the sides of $\Sigma_{N,m}$ in $\B^3$
(that is: it preserves either choice of global unit normal).

\begin{lemma}[Index and nullity bounds on the catenoid regions]
\label{lem:cat_bounds}
Let $N \geq 2$ be an integer.
There exists $m_0>0$ such that for every integer $m>m_0$
we have
\begin{align*}
\ind(Q_{\catr_i}^\dir)
  \geq
  \equivind{\G_{\catr_i}}(Q_{\catr_i}^\dir)
&\geq
  1,
&
\ind(Q_{\catr_i}^\neum)+\nul(Q_{\catr_i}^\neum)
&\leq
  3,
&
\equivind{\G_{\catr_i}}(Q_{\catr_i}^\neum)
  +\equivnul{\G_{\catr_i}}(Q_{\catr_i}^\neum)
&\leq
  2
\end{align*}
for every catenoid region
$\catr_i \subset \Sigma_{N,m}$
and
\begin{equation*}
\equivind{\G_{\catr_n}}(Q_{\catr_n}^\neum)
  +\equivnul{\G_{\catr_n}}(Q_{\catr_n}^\neum)
  \leq
  1
  \quad \mbox{for $N$ even}.
\end{equation*}
\end{lemma}

\begin{proof}
By the conformal invariance of the index and nullity
in dimension two
(as proven, for example, in
\cite{CSWSpectral}*{Proposition 3.11})
\begin{align*} 
\ind(Q^\dir_{\catr_i})
&=\ind\bigl(
  Q^\dir_{\catr_i},
  (\varpi_{N,m}^*\rho^2)\met{\Sigma_{N,m}}
\bigr),
\\[.5ex]
    \ind\bigl(Q^\neum_{\catr_i}\bigr)+\nul\bigl(Q^\neum_{\catr_i}\bigr)
    &=\ind\bigl(
    Q^\neum_{\catr_i},
    (\varpi_{N,m}^*\rho^2)\met{\Sigma_{N,m}}
    \bigr)
    +\nul\bigl(
    Q^\neum_{\catr_i},
    (\varpi_{N,m}^*\rho^2)\met{\Sigma_{N,m}}
  \bigr).
\end{align*}
The claims now follow by applying
Corollary \ref{cor:final_geo}
and items \ref{itm:geo_ests_cat:met}
and \ref{itm:geo_ests_cat:twoff}
of Lemma \ref{lem:geo_ests_cat}
(which ensure closeness,
through the appropriate diffeomorphisms,
of the coefficients of $Q^{\,\cdot}_{\catr_i}$
to those of $Q_{\catdom_T}^{\,\cdot}$,
with the appropriate $T$),
Lemma \ref{lem:quant_spec_shift}
(which compares the spectrum
of the region in question with that of the corresponding model),
and analysis from Lemma \ref{lem:low_spec_cat} (which bounds the low spectrum of the model).
In applying Lemma \ref{lem:quant_spec_shift}
we make use of Lemma \ref{lem:cat_trace}
to ensure \eqref{eqn:trace_inequality}
holds with
$
\lipdom
=
\kappa_i^{-1}(\varpi_{N,m}(\catr_i))
$
uniformly in $m$
for a single choice of $C_\lipdom^{\mathrm{tr}}$.
Note in particular that under the conformal transformation
which we apply
the natural Robin potential $r\equiv 1$
becomes $(\varpi_{N,m}^*\rho)^{-1}$,
tending to zero as $m\to\infty$ and so justifying
the (homogeneous) Neumann condition we impose on the model problem.
\end{proof}

We further define
for each $\Sigma_{N,m}$ and each $D_i$
(as in \eqref{layers_with_half_cats})
on the corresponding initial surface
\begin{align*}
\interr_i
&\vcentcolon=
  \varpi_{N,m}^{-1}
    \bigl(
      D_i(m^{-4})
      \cap 
      \Phi(\{\sigma < 3/m\})
    \bigr),
&
\discr_i
&\vcentcolon=
  \varpi_{N,m}^{-1}
    \bigl(
      D_i(m^{-4})
      \setminus
      \Phi(\{\sigma \leq 3/m\})
    \bigr),
\\[1ex]
Q_{\interr_i}^\neum
&\vcentcolon=
  \nint{\bigl(Q^{\Sigma_{N,m}}\bigr)}_{\interr_i},
&
Q_{\discr_i}^\neum
&\vcentcolon=
  \nint{\bigl(Q^{\Sigma_{N,m}}\bigr)}_{\discr_i}
\end{align*}
(recalling \eqref{eqn:discr_init_def}).
Observe that each
$\discr_i$ is approximately a horizontal disc
whose boundary avoids $\Sp^2$
but rather is shared with one boundary
component of the corresponding $\interr_i$,
an approximately flat annulus
whose remaining boundary component
alternates between arcs on $\Sp^2$
and arcs through catenoidal ribbons,
along which $\interr_i$ borders a
component of some $\pyr_m \catr_j$.
In total we have $N$ disc regions $\discr_i$,
$N$ intermediate regions, $\interr_i$,
and $(N-1)m$ catenoid regions,
the components of $\bigcup_i \pyr_m \catr_i$.
The preceding regions are pairwise disjoint
and their closures cover $\Sigma_{N,m}$.

\begin{lemma}[Equivariant index and nullity upper bounds on the disc regions]
\label{lem:disc_bounds}
Let $N \geq 2$ be an integer.
There exists $m_0>0$
such that for every integer $m>m_0$
we have
\begin{equation*}
\equivind{\pyr_m}(Q_{\discr_i}^\neum)
  +\equivnul{\pyr_m}(Q_{\discr_i}^\neum)
=
\ind(Q_{\discr_i}^\neum)
  +\nul(Q_{\discr_i}^\neum)
\leq
1
\end{equation*}
for every disc region
$\discr_i \subset \Sigma_{N,m}$
and
\begin{equation*}
\equivind{\apr_m}(Q_{\discr_{n+1}}^\neum)
  +\equivnul{\apr_m}(Q_{\discr_{n+1}}^\neum)
=
0
\quad \mbox{for $N$ odd}.
\end{equation*}
\end{lemma}

\begin{proof}
Note that $\varpi_{N,m}(\discr_i)$
is an exact planar disc.
Since the Neumann Laplacian on a disc
has index zero and nullity one,
coming from the constants,
the claim follows by applying
Lemma \ref{lem:quant_spec_shift}
in conjunction with
Corollary~\ref{cor:final_geo}.
\end{proof}

\begin{lemma}[Index and nullity
upper bounds on the intermediate regions]
\label{lem:margin_bounds_equivariant}
Let $N \geq 2$ be an integer.
There exists $m_0>0$ such that for every integer $m>m_0$
we have
\begin{equation*}
\equivind{\pyr_m}(Q_{\interr_i}^\neum)
  +\equivnul{\pyr_m}(Q_{\interr_i}^\neum)
\leq
1   
\end{equation*}
for every intermediate region
$\interr_i \subset \Sigma_{N,m}$
and
\begin{equation*}
\equivind{\apr_m}(Q_{\interr_{n+1}}^\neum)
=
\equivnul{\apr_m}(Q_{\interr_{n+1}}^\neum)
=
0
\quad \mbox{for $N$ odd}. 
\end{equation*}
Furthermore
we have
\begin{equation*}
\ind(Q_{\interr_i}^\neum)
  +\nul(Q_{\interr_i}^\neum)
\leq
2m.
\end{equation*}
\end{lemma}

\begin{proof}
Let us first recall the general ``reduction principle'' given in \cite[Lemma 3.5]{CSWSpectral} and the related comments presented in Example 3.7 therein: computing $\equivind{\pyr_m}(Q_{\interr_i}^\neum)$
  and $\equivnul{\pyr_m}(Q_{\interr_i}^\neum)$ is equivalent to studying the index and nullity for the Jacobi operator on a fundamental domain for the $\pyr_m$-action, subject to adjoined Neumann boundary conditions in the interior boundary.
Since
\begin{equation*}
\equivind{\pyr_m}(Q_{\interr_i}^\neum)
+\equivnul{\pyr_m}(Q_{\interr_i}^\neum)
=
\equivind{\pyr_m}
    (
      Q_{\interr_i}^\neum,
      m^2\met{\Sigma_{N,m}} 
    )
    +\equivnul{\pyr_m}
  (
    Q_{\interr_i}^\neum, m^2\met{\Sigma_{N,m}}
  ),
\end{equation*}
the claim then follows by applying
Lemmata \ref{lem:low_spec_margin}, \ref{lem:perf_rect_trace} and \ref{lem:quant_spec_shift},
in view of Corollary~\ref{cor:final_geo} and items 
\ref{itm:geo_ests_disc:met_Phullback}
and \ref{itm:geo_ests_disc:twoff_Phullback} of Lemma \ref{lem:geo_ests_disc}
(and, for the final two claimed equalities in the statement,
the fact that the first eigenfunction
of $Q_{\interr_i}^\neum$
has a sign).
We point out again in particular
that the geometric Robin potential
$r \equiv 1$
is rescaled to $m^{-1}$,
justifying the comparison
with the Neumann problem on the model.

Let us turn to the final claim: if $\Omega_1, \ldots, \Omega_{2m}$
denote the interiors of the connected components of
\begin{equation*}
\interr_i\setminus \bigl\{(x,y,z)\st \exists j \in \Z\quad y=x \tan\bigl(\tfrac{2j+1}{2m}\pi\bigr)\bigr\} ,
\end{equation*}
where we are cutting
by the vertical planes of symmetry, the conclusion comes straight from employing the preceding argument, which precisely gives
\begin{equation*}
\ind\Bigl(\nint{\bigl(Q_{\interr_i}^\neum\bigr)}_{\Omega_j}\Bigr)
  +\nul\Bigl(\nint{\bigl(Q_{\interr_i}^\neum\bigr)}_{\Omega_j}\Bigr)
\leq 1 
\end{equation*}
for each $\Omega_j$ as input for applying Proposition \ref{prop:mr} \ref{mrUpper}.
\end{proof}

\paragraph{Conclusions.} 
We now have all necessary tools in place to prove the estimates on index and nullity claimed in Theorem \ref{thm:Main} and in Remark \ref{rem:IndexTop}. 
Let us start with the upper bounds.

\begin{proposition}
[Index and nullity upper bounds for $\Sigma_{N,m}$]\label{pro:Index_UpperBounds}
Let $N \geq 2$. There exists $m_0>0$
such that for every integer $m>m_0$
\begin{equation*}
\ind(\Sigma_{N,m}) + \nul(\Sigma_{N,m})
\leq
m(5N-3) + N.
\end{equation*}
\end{proposition}

\begin{proof}
We apply Proposition~\ref{prop:mr}~\ref{mrUpper}
(with trivial $\grp$),
relying upon Lemma \ref{lem:cat_bounds},
Lemma \ref{lem:disc_bounds},
and Lemma \ref{lem:margin_bounds_equivariant}.
To the resulting upper bound
the catenoid regions make total contribution
$3m(N-1)$,
the disc regions $N$,
and the intermediate regions $2mN$.
\end{proof}

\begin{proposition}
[Equivariant index and nullity upper bounds
for $\Sigma_{N,m}$]\label{pro:Index_Equiv_UpperBounds}
Let $N \geq 2$ be an integer. 
There exists $m_0>0$ such that for every integer $m>m_0$
\begin{equation*}
\equivind{\G}(\Sigma_{N,m})\leq
\begin{cases}
2N-1 & \mbox{for $N$ even,}
\\
2N-2 & \mbox{for $N$ odd}.
\end{cases}
\end{equation*}
\end{proposition}

\begin{proof}
We similarly apply
Lemma \ref{lem:cat_bounds},
Lemma \ref{lem:disc_bounds},
Lemma \ref{lem:margin_bounds_equivariant},
and
Proposition \ref{prop:mr} \ref{mrUpper},
but now with $\grp=\G$ (the maximal symmetry group of the surface); each $\Omega_i$ of the Montiel--Ros partition we wish to employ is 
a $\G$-invariant set that is obtained starting with either a disc region or an intermediate region or a catenoid region.
We note that, modulo the symmetries, we have:
\begin{itemize}
\item for $N$ even, 
exactly $n$ such $\Omega_i$ consisting of catenoidal components, $n$ consisting of intermediate regions,
and $n$ consisting of disc regions;
\item for $N$ odd, 
exactly $n$ such $\Omega_i$ consisting of catenoidal components, $n+1$ consisting of intermediate regions,
and $n+1$ consisting of disc regions.
\end{itemize}

Carefully accounting for the contribution of all such components, based on the aforementioned lemmata, gives the desired result.
\end{proof}

We now switch gears and move to studying \emph{lower} bounds on the Morse index of the free boundary minimal disc stackings. To begin with, it is instructive to first recall from \cite{CSWSpectral}
the following index lower bounds
in terms of symmetries of the surface in question.

\begin{proposition}[Index lower bounds under
pyramidal and prismatic symmetry {\cite[Prop. 4.2]{CSWSpectral}}]
\label{indBelowPyr}
Let $\Sigma$ be a connected, embedded free boundary minimal surface in $\B^3$.
Assume that $\Sigma$ is not a disc or critical catenoid, that $\Sigma$ is invariant under reflection through a plane $\Pi_1$, and that $\Sigma$ is also invariant under rotation through an angle $\alpha \in \interval{0,2\pi}$ about a line $\xi \subset \Pi_1$.
Then $\alpha$ is a rational multiple of $2\pi$, there is a largest integer $k \geq 2$ such that rotation about $\xi$ through angle $2\pi/k$ is also a symmetry of $\Sigma$, and
\begin{enumerate}[label={\normalfont(\roman*)}]
\item\label{prop:indBelowPyr-i}
  $\ind(\Sigma) \geq 2k-1$,
\item\label{prop:indBelowPyr-ii}
 $\symind{\cycgrp{\Pi_1}}{-}(\Sigma) \geq k-1$,
\item\label{indBelowPri}
  if $\Sigma$ is additionally invariant
under reflection through a plane $\Pi_\perp$ orthogonal to $\xi$,
then in fact
$\symind{\cycgrp{\Pi_\perp}}{+}(\Sigma) \geq 2k-1$.
\end{enumerate}
\end{proposition}

\begin{corollary}[Index lower bounds for $\Sigma_{N,m}$ from symmetries]\label{cor:lower}
Let $\Sigma_{N,m}$ be any embedded free boundary minimal surface in $\B^3$ with the topology and symmetry group given in Theorem \ref{thm:Main}. Then
\[
\ind(\Sigma_{N,m})\geq\begin{cases}
2m & \text{ for $N$ even, }\\
2m-1 & \text{ for $N$ odd. } 
\end{cases}
\]
\end{corollary}

\begin{proof}
For $N$ odd we apply Proposition \ref{indBelowPyr} \ref{prop:indBelowPyr-i} with $\Pi_1$ any vertical plane of symmetry, $k=m$ and $\xi=\{x=y=0\}$. 
For $N$ even we similarly apply Proposition \ref{indBelowPyr} \ref{indBelowPri} for $k=m$,
with $\Pi_1$ any vertical plane of symmetry and $\xi=\{x=y=0\}$, and then Proposition \ref{indBelowPyr} \ref{prop:indBelowPyr-ii} for $k=2$
with $\Pi_1$ the horizontal plane $\left\{z=0\right\}$ and $\xi=\{y=x\tan(\pi/(2m))\}$.
\end{proof}
 
In the case $N=2$, we expect the lower index bound stated in Corollary \ref{cor:lower} to be sharp, as explained in \cite[Conjecture 7.8 (iv)]{CSWnonuniqueness}. 
As anticipated in the introduction, we wish to refine our analysis so as to get a bound that also accounts for the number of layers, as seems natural. 
Indeed, Lemma \ref{lem:cat_bounds} allows us to derive the following lower bound, an improvement to the former whenever $N \geq 3$.

\begin{proposition}[Index lower bounds for $\Sigma_{N,m}$ depending on $N$]\label{pro:Index_LowerBounds}
Let $N \geq 2$ be an integer.
There exists $m_0>0$ such that
for every integer $m>m_0$
\begin{equation*}
\ind(\Sigma_{N,m}) \geq m(N-1).
\end{equation*}
\end{proposition}

\begin{proof}
We apply
Proposition \ref{prop:mr} \ref{mrLower}
(with trivial $\grp$) with a partition of $\Sigma_{N,m}$
into precisely $m(N-1)+1$ domains, namely all (connected) catenoid regions, plus the complement of their disjoint union.
Applying the local lower bound given in
Lemma \ref{lem:cat_bounds}, the proof is readily completed.
\end{proof}

We state also an equivariant counterpart.
\begin{proposition}
[Equivariant index lower bounds for $\Sigma_{N,m}$]\label{pro:Index_Equiv_LowerBounds}
Let $N \geq 2$ be an integer.
There exists $m_0>0$ such that for every integer
$m>m_0$
\begin{equation*}
\equivind{\G}(\Sigma_{N,m})
\geq
n=\floor{N/2}.
\end{equation*}
\end{proposition}

\begin{proof}
Here we just reconsider the very same partition employed in the proof of Proposition \ref{pro:Index_Equiv_UpperBounds},
but this time we apply
Lemma \ref{lem:cat_bounds}
and Proposition \ref{prop:mr} \ref{mrLower},
with $\grp=\G$.
\end{proof}

\appendix

\section{Vertical forces}
\label{app:forces}
This appendix contains the key estimates
establishing that the cokernel of the construction
can be controlled by variation
of the $\vec{\zeta},\vec{\xi}$ parameters.
The treatment is similar to
that in \cite[Section 3]{Wiygul2020},
but with symmetry constraints
and some simplifications.
Most notably we appeal
to the Perron--Frobenius theorem
to streamline the proof
of Lemma \ref{lem:limiting_waist_ratios}
(in comparison to \cite[Lemma 3.18]{Wiygul2020}),
and we prove uniqueness there
(whose analogue was not asserted
in \cite{Wiygul2020}). 
The use of forces as proxies
for components of the cokernel
follows \cite{KapouleasYangTorusDoubling},
while the introduction of dislocations
is motivated by the methodology of 
\cite{KapouleasWenteTori}.
The same underlying ideas
were applied in \cite{KapouleasWiygul17},
but Lemma \ref{lem:limiting_waist_ratios}
is trivial in that case
and the proof of
the corresponding special case
of Lemma \ref{lem:coker_control}
much simpler.

We begin with a lemma
which is needed for the definition
of the initial surfaces
and for the remainder of this appendix.

\begin{lemma}[Limiting waist ratios]
\label{lem:limiting_waist_ratios}
Let $N \geq 2$ be an integer 
and $n=\floor{N/2}$.
There exists
$\vec{x}=(x_1,\ldots,x_{N-1}) \in \R^{N-1}_{\updownarrow+}$
such that 
\begin{equation}
\label{eqn:x_ordering}
0 < x_1 < \cdots < x_{n-1} < x_n = 1,
\end{equation}
and
\begin{equation}
\label{eqn:x_balancing}
x_{i-1} + x_{i+1}
=
2\frac{x_{n-1}+N \bmod 2}{1+N \bmod 2}x_i
\quad
\mbox{for each }
1 \leq i \leq N-1,
\end{equation}
understanding $x_0=x_N=0$.
Furthermore,
$\vec{x}$
is the only vector in $\R^{N-1}$
satisfying \eqref{eqn:x_balancing}
and having $x_n=1$
and all entries strictly positive.
\end{lemma}

\begin{proof}
Let $B$ be the $(N-1) \times (N-1)$ matrix such that
\begin{equation*}
B_{ij}=
\begin{cases}
1 &\mbox{if } \abs{i-j}=1
\\
0 &\mbox{otherwise},
\end{cases}
\end{equation*}
and let $I$ be
the $(N-1) \times (N-1)$
identity matrix.
Then $(I+B)^{N-2}$
is entrywise strictly positive,
and by the Perron--Frobenius theorem
(see for example
\cite{MacCluerPF}
or the special case treated in \cite{NinioSimplePF},
which suffices for this application)
the maximum eigenvalue $\lambda$
of $B$ is strictly positive
and has eigenspace of dimension one
(in other words $\lambda$ is a simple eigenvalue)
which moreover includes a vector
with all entries strictly positive. In particular there exists a unique
vector $\vec{x} \in \R^{N-1}$
such that $x_n=1$
and $B\vec{x}=\lambda \vec{x}$,
and we then have also
$x_i>0$ for all $i$.
Since left multiplication by $B$
preserves
each of the two summands
in the decomposition
$
 \R^{N-1}
 =
 \R^{N-1}_{\updownarrow+}
   \oplus \R^{N-1}_{\updownarrow-}
$,
it follows from the simplicity of $\lambda$
and the positivity of the entries that
$\vec{x} \in \R^{N-1}_{\updownarrow+}$.

From $B\vec{x}=\lambda \vec{x}$
and the definition of $B$
we have
\begin{equation*}
x_{i-1}+x_{i+1}
=
\lambda x_i
\quad
\mbox{for each } 1 \leq i \leq N-1.
\end{equation*}
In particular,
by taking $i=n$
and applying the choice $x_n=1$
we find
\begin{equation*}
x_{n-1} + x_{n+1}
=
\lambda,
\end{equation*}
but since $\vec{x} \in \R^{N-1}_{\updownarrow+}$ one has in particular
\begin{equation*}
x_{n+1}
=
\begin{cases}
x_{n-1} & \mbox{if } N \bmod 2 = 0
\\
x_n=1 & \mbox{if } N \bmod 2 = 1,
\end{cases}
\end{equation*}
and so we conclude
\begin{equation*}
\lambda
=
2\frac{x_{n-1} + N \bmod 2}{1 + N \bmod 2}.
\end{equation*}
Consequently
\eqref{eqn:x_balancing} is satisfied.

If $\vec{y} \in \R^{N-1}$
satisfies \eqref{eqn:x_balancing}
(with $y$ in place of $x$),
then $\vec{y}$ is an eigenvector
of $B$.
If the corresponding eigenvalue
differs from $\lambda$,
then, by symmetry of $B$,
$\vec{y}$ is orthogonal to $\vec{x}$,
so cannot have all entries strictly positive.
If instead
the corresponding eigenvalue
agrees with $\lambda$,
then 
by the latter's simplicity and the assumption
$y_n=1$ imply $\vec{y}=\vec{x}$.
Thus we have verified the uniqueness assertion.

Finally, we observe that
\begin{equation*}
\lambda
=
\frac{{\vec{x}\vphantom{\big\vert}}^\intercal B \vec{x}}
  {\abs{\vec{x}}^2}
=
\frac{2}{\abs{\vec{x}}^2}
  \sum_{i=1}^{N-2} x_ix_{i+1}
<2,
\end{equation*}
where $\abs{{}\cdot{}}$
is the Euclidean norm on $\R^{N-1}$
and the strict inequality
follows from the Cauchy--Schwarz inequality
and the fact that $x_1=x_{N-1}>0$ (because of the known positivity of all entries, as mentioned above).
Therefore we have
\begin{equation*}
x_{n-1}
=
\frac{1+N \bmod 2}{2}\lambda - N \bmod 2
<
1=x_n
\end{equation*}
and
\begin{equation*}
x_{i-1}
=
\lambda x_i - x_{i+1}
<
x_i + (x_i - x_{i+1})
\quad
\mbox{for each }
2 \leq i \leq n-1,
\end{equation*}
whence, by downward induction starting at $i=n-1$, follows
\eqref{eqn:x_ordering},
completing the proof.
\end{proof}

Next, for each initial surface
$\Sigma=\Sigma_{N,m,\vec{\zeta},\vec{\xi}}$ (as designed in the first half of Section \ref{sec:Stackings})
we define
$
 \vec{\force}
 =
 \vec{\force}_{N,m,\vec{\zeta},\vec{\xi}}
 \in
 \R^N_{\updownarrow-}$
by
\begin{equation}
\label{eqn:force_def}
\force_i
=
\force_{i;N,m,\vec{\zeta},\vec{\xi}}
\vcentcolon=
\int_{\partial D_i}
  \met{\R^3}(\eta^{D_i}, \partial_z|_\Sigma)
  \, \hausint{1}{\met{\R^3}}
\end{equation}
for each region
$D_i \subset \Sigma$
as given in \eqref{layers_with_half_cats},
with $\eta^{D_i}$ denoting the outward unit conormal.
We will refer to the $\force_i$ as forces,
since if $D_i$ is interpreted physically
as an idealized film (not in equilibrium),
then $-\force_i$ represents
the $z$ component of the net force
that $D_i$ exerts, through surface tension,
on $\partial D_i$.
(Integrals of type \eqref{eqn:force_def}, also called fluxes,
provide a standard tool in the study of (exactly) minimal surfaces; see for example \cite[Chapter 1, \S3.1]{CM11}.)

By the first variation formula 
and the fact $\partial_z$ is Killing,
$\force_i$ is equivalently
the integral over $D_i$
of the $z$ component of the (vector-valued)
mean curvature of $D_i$,
which perspective motivates
our interest in $\vec{\force}$:
by manipulating
(via Lemma~\ref{lem:coker_control})
the components
of the initial mean curvature that
the $\force_i$ represent
we will obtain control
(in Proposition \ref{pro:existence_of_fbms},
with the aid of Lemma \ref{lem:perturbed_force})
over the subspace of the approximate cokernel
(that is the span of the $w_i$ and $\wbar_i$
functions as defined in
\eqref{eqn:w_ess_def}
and \eqref{eqn:wbar_def} respectively,
and as applied in
Proposition \ref{pro:global_inverse})
spanned by the $w_i$.
The span of the $\wbar_i$
will instead be controlled,
much more directly,
by adjusting the dislocations
$\disloc_i$
(introduced in \eqref{eqn:disloc_def}).
We will see below that the assumptions of
Lemma \ref{lem:limiting_waist_ratios}
can be understood as enforcing
approximate balancing
(that is vanishing of the $\force_i$).
Lemma \ref{lem:coker_control}
represents finer control
over the $\force_i$,
as well as the $\disloc_i$.

We proceed to estimate the $\force_i$.
In doing so we will make repeated use
of \eqref{eqn:tau_size},
without any further comment.
We will also apply,
throughout this appendix,
the notation
$a=b+O(c)$
(and simply $a=O(c)$ in case $b=0$)
as follows.
Here $a,b \in \R$ and $c>0$ are real numbers
which may depend on initial-surface data
$(N,m,\vec{\zeta},\vec{\xi}\,)$.
Then we stipulate that 
$a=b+O(c)$ if and only if 
\begin{equation*}
\exists
  m_0
  =
  m_0(N,\abs{\vec{\zeta}}+\abs{\vec{\xi}})
  >
  0
\quad
\exists
  C
  =
  C(N)
  >
  0
\quad
\forall m\geq m_0
: \quad
\frac{\abs{a-b}}{c} \leq C.
\end{equation*}
From the definitions \eqref{eqn:force_def} of $\force_i$ and \eqref{layers_with_half_cats} of $D_i$
(using in particular
the orthogonality of $\Sigma$
to $\partial \B^3$)
we find
\begin{align*}
\force_i
&=
\int_{\partial D_i \cap \Sp^2}
  z
  \, \hausint{1}{\met{\R^3}}
+\int_{
    \partial D_i \cap \Phi(\{\omega=h^K_i\})
  }
  \met{\R^3}(\eta^{D_i},\partial_z|_\Sigma)
  \, \hausint{1}{\met{\R^3}}
\\
&\hphantom{{}={}}
  +\int_{
    \partial D_i \cap \Phi(\{\omega=h^K_{i-1}\})
  }
  \met{\R^3}(\eta^{D_i},\partial_z|_\Sigma)
  \, \hausint{1}{\met{\R^3}},
\end{align*}
where we understand
$\{\omega=h^K_0\}=\{\omega=h^K_N\}=\emptyset$,
so that the corresponding integrals terms
for $i=1$ and $i=N$,
are omitted.
With the aid of \eqref{eqn:Phi_pullback_metric}
and then \eqref{eqn:height_bound}
we estimate
\begin{equation*}
\abs[\Big]{
  (
    \Phi^*\met{\R^3}
    -\met{\dom \Phi}
  )|_{
    \partial D_i \cap \{\omega=h^K_j\}
  }
}_{\met{\dom \Phi}}
=
O\bigl(\tau_j^2+(h^K_j)^2\bigr)
=
O(m^2\tau_1^2),
\end{equation*}
whence
\begin{equation}
\label{eqn:force_est_waists}
\force_i
=
\int_{\partial D_i \cap \Sp^2}
  z
  \, \hausint{1}{\met{\R^3}}
  + m\pi\tau_i - m\pi\tau_{i-1}
  +O(m^3\tau_1^3),
\end{equation}
where we understand $\tau_0=\tau_N=0$.

\begin{figure}%
\providecommand{\bdry}{
--++(\tauA,0)
..controls+(0,0.04)and+(-3:-0.1)..
++(pi/2/\mpar-\tauA,\hB)
..controls+(-3:0.1)and+(0,-0.04)..
++(pi/2/\mpar-\tauB,\hK-\hB)
--++(2*\tauB,0)
..controls+(0,-0.04)and+(3:-0.1)..
++(pi/2/\mpar-\tauB,\hB-\hK)
..controls+(3:0.1)and+(0,0.04)..
++(pi/2/\mpar-\tauA,-\hB)--++(\tauA,0)
}
\pgfmathsetmacro{\mpar}{8}
\pgfmathsetmacro{\tauA}{0.0387}
\pgfmathsetmacro{\tauB}{0.0274}
\pgfmathsetmacro{\hK}{0.1315}
\pgfmathsetmacro{\hB}{0.0713}
\pgfmathsetmacro{\dashheight}{0.15}
\begin{tikzpicture}[line cap=round,line join=round,scale={\mpar/2/pi*\textwidth/2.01cm}] 
\begin{scope}[very thick]
\clip(-2*pi/\mpar-pi/2/\mpar,\hB-\dashheight)rectangle++(4*pi/\mpar,2*\dashheight);
\draw[Blau](-4*pi/\mpar,0)\bdry\bdry\bdry;
\foreach\k in {-2,-1,...,2}{
\begin{scope}[shift={(\k*pi/\mpar,\hB)}]
\draw[dotted,thin](pi/4/\mpar,-\dashheight)--++(0,2*\dashheight)(-pi/4/\mpar,-\dashheight)--++(0,2*\dashheight);;
\draw[dashed,semithick](0,-\dashheight)--++(0,2*\dashheight)node[below right,inner sep=0]{~$L_i^{\Sp^2}$};
\end{scope}
}
\begin{scope}[Orange]
\clip
(2*pi/\mpar-pi/4/\mpar,\hB-\dashheight)rectangle++(pi/2/\mpar,2*\dashheight)
(pi/\mpar-pi/4/\mpar,\hB-\dashheight)rectangle++(pi/2/\mpar,2*\dashheight)
(-pi/4/\mpar,\hB-\dashheight)rectangle++(pi/2/\mpar,2*\dashheight) 
(-pi/\mpar-pi/4/\mpar,\hB-\dashheight)rectangle++(pi/2/\mpar,2*\dashheight)
(-2*pi/\mpar-pi/4/\mpar,\hB-\dashheight)rectangle++(pi/2/\mpar,2*\dashheight);
\draw(-4*pi/\mpar,0)\bdry\bdry\bdry;
\end{scope}
\draw[gray]
(-\tauA,0)--++(2*\tauA,0)
(-\tauA+2*pi/\mpar,0)--++(2*\tauA,0)
(-\tauA-2*pi/\mpar,0)--++(2*\tauA,0)node[near end,inner sep=0.5pt](H-){}
;
\draw[gray]
(-\tauB-pi/\mpar,\hK)--++(2*\tauB,0)
(-\tauB+pi/\mpar,\hK)--++(2*\tauB,0)node[near start,inner sep=0.5pt](H+){}
;
\end{scope}
\draw[latex-latex](pi/\mpar,0)--++(-pi/4/\mpar,0)node[midway,below]{$\frac{\pi}{4m}$};
\draw[Blau](-pi/2/\mpar,\hB)node[below]{$\Gamma^B_i$};
\draw[Orange](pi/4/\mpar-pi/\mpar,\hB)node[below left]{$\Gamma^K_i$};
\draw(-3*pi/2/\mpar,\hB-\dashheight)node[above=-.5pt,rectangle,fill=white](h-){$\scriptstyle\partial D_i \cap \Phi(\{\omega=h^K_{i-1}\})$};
\draw(pi/3.9/\mpar,\hB+\dashheight)node[below right=-.5pt,rectangle,fill=white](h+){$\scriptstyle\partial D_i \cap \Phi(\{\omega=h^K_{i}\})$};
\draw[->](h-.west)to[in=-90,out=180](H-);
\draw[->](h+.east)to[in=90,out=0](H+);
\end{tikzpicture}
\caption{Schematic visualization of $\partial D_i$ (for $i \not \in \{1,N\}$, cf. Figure \ref{fig:Initialsurface}) 
and the curves $\Gamma^K_i$ and $\Gamma^B_i$.}%
\label{fig:decomposition}%
\end{figure}

To estimate the remaining integral
we will decompose its domain of integration
into two (disconnected) curves,
$\Gamma^K_i$ and $\Gamma^B_i$,
which,
recalling
\eqref{eqn:axes_def},
we specify in terms of
the union
$L_i^{\Sp^2} \vcentcolon= \Phi(L_i)$
of meridians on $\Sp^2$
whose inverse images under $\Phi$
are axes of catenoids employed
in the assembly of $D_i$.
Thus $L_i^{\Sp^2}$ is a union of $m$ meridians
when $D_i$ is outermost and otherwise
a union of $2m$ meridians.
Now let $\Gamma^K_i$ be the intersection
of $\partial D_i$ with a tubular neighborhood
in $\Sp^2$ of $L_i^{\Sp^2}$ of (longitudinal) radius
$\frac{\pi}{4m}$,
and let
$
 \Gamma^B_i
 \vcentcolon=
 \partial D_i \cap \Sp^2 \setminus \Gamma^K_i
$.
See Figure \ref{fig:decomposition}.

For each $j \in \{i-1,i\}$
we fix a component
$\gamma^K_{i,j}$ of $\Gamma^K_i$
such that $\gamma^K_{i,j}$ has a boundary point
on $\Phi(\{\omega=h^K_j\})$;
we understand
$\gamma^K_{1,0}=\gamma^K_{N,N}=\emptyset$.
We fix also a component $\gamma^B_i$
of $\Gamma^B_i$.
In the estimates to follow
we will make use of
parametrizations of the $\gamma^K_{i,j}$
and $\gamma^B_i$
that can be read off
by inspection of
\eqref{layers_with_half_cats}
(defining the $D_i$)
and the supporting definitions,
at the heart of which are
the defining functions
\eqref{angular_height_function_one_bridge}
and \eqref{angular_height_function_two_bridges}
near $\Phi^{-1}\partial\Sigma$
in the domain of $\Phi$.

Referring to \eqref{def:Phi}
and again using
both \eqref{eqn:Phi_pullback_metric}
and \eqref{eqn:height_bound},
we continue by estimating
\begin{equation}
\label{eqn:bdy_met_and_height_est}
\biggl(
  \abs[\Big]{
    (
      \Phi^*\met{\R^3}
      -\met{\dom \Phi}
    )
  }_{\met{\dom \Phi}}
  +(\Phi^*z-\omega)
\biggr)
\bigg|_{\Phi^{-1}(\partial D_i \cap \Sp^2)}
=
O(\omega^2)
=
O(m^2\tau_1^2).
\end{equation}
Using this last estimate,
the additional estimate
(for the integral of the height function
along a catenary arc)
\begin{equation*}
\int_0^{\arcosh \frac{\pi}{4m\tau_j}}
  \Bigl(h^K_j - (-1)^{i-j} \tau_j t\Bigr)
  \, \tau_j \cosh t \, dt
=
\frac{\pi}{4m}
  \Bigl(
    h^K_j
    -(-1)^{i-j}(\delta h^K)_j/2
  \Bigr)
  +O(m^{-1}\tau_1)
\end{equation*}
(using in particular
\eqref{eqn:cat_extent}
and
$
\arcosh \frac{\pi}{4m\tau_j}
=
\arcosh \frac{1}{m\tau_j} + O(1)
$),
and \eqref{eqn:dislocated_matching},
we in turn obtain
(for $\gamma^K_{i,j} \neq \emptyset$)
\begin{equation}
\label{eqn:K_bdy_int}
\int_{\gamma^K_{i,j}} z \, \hausint{1}{\met{\R^3}}
=
\frac{\pi}{4m}
\Bigl(
    h^B_i
    +(-1)^{i-j}\disloc_i
\Bigr)
  +O(m^{-1}\tau_1).
\end{equation}
We next observe
(referring ultimately
to the defining functions
\eqref{angular_height_function_one_bridge}
and \eqref{angular_height_function_two_bridges})
the estimate
\begin{equation*}
\frac{d\omega|_{\Phi^{-1}(\gamma^B_i)}}
  {d\theta|_{\Phi^{-1}(\gamma^B_i)}}
=
O(m^2\tau_1)
\end{equation*}
and from here, again using
\eqref{eqn:bdy_met_and_height_est}
and \eqref{eqn:dislocated_matching},
we obtain also
\begin{equation}
\label{eqn:B_bdy_int}
\int_{\gamma^B_i} z \, \hausint{1}{\met{\R^3}}
=
\begin{cases}
\frac{\pi}{2m}h^B_i
  +O(m^{-1}\tau_1)
  &\mbox{for } 2 \leq i \leq N-1
\\
\frac{3\pi}{2m}h^B_i
  +O(m^{-1}\tau_1)
  &\mbox{for } i \in \{1,N\},
\end{cases}
\end{equation}
noting in particular
that,
to leading order,
the contribution
of $\disloc_i$
to $z$ is antisymmetric in $\theta$
so integrates to zero.

Using
\eqref{eqn:force_est_waists},
\eqref{eqn:K_bdy_int},
and \eqref{eqn:B_bdy_int}
(and recalling
from \eqref{eqn:disloc_def}
that $\disloc_1=\disloc_N=0$),
we conclude
\begin{equation}
\label{eqn:force_est}
\force_i
=
2\pi h^B_i + m\pi\tau_i - m\pi\tau_{i-1}
  +O(\tau_1),
\end{equation}
noting in particular that the $\disloc_i$
contributions from the components
of $\Gamma^K_i$
(as identified in \eqref{eqn:K_bdy_int})
cancel in pairs
(or vanish identically from the start
in case $i \in \{1,N\}$).

For the proof of Lemma \ref{lem:coker_control} below
it is useful to introduce and examine
$
 \vec{\widetilde{\force}}
 \in
 \R^{N-1}_{\updownarrow+}
$,
which,
recalling
$\vec{\disloc} \in \R^{N,0,0}_{\updownarrow+}$
from \eqref{eqn:disloc_def},
we define
by
\begin{equation}
\label{eqn:Ftilde_def}
\widetilde{\force}_i
\vcentcolon=
\frac{
  (\force_{i+1}-\force_i)
  -2(\disloc_i + \disloc_{i+1})
  }
  {\pi \tau_n}.
\end{equation}
From \eqref{eqn:force_est}
and
\eqref{eqn:disloc_def},
and using
\eqref{eqn:dislocated_matching}
and
\eqref{eqn:cat_extent}
(after expressing $\arcosh$ therein
in terms of $\log$)
to estimate $h^B_{i+1}-h^B_i$,
we find
\begin{equation*}
\widetilde{\force}_i
=
4\frac{\tau_i}{\tau_n} \log \frac{1}{m\tau_i}
  +m\frac{\tau_{i+1}}{\tau_n}
  -2m\frac{\tau_i}{\tau_n}
  +m\frac{\tau_{i-1}}{\tau_n}
  +O(1).
\end{equation*}
Applying the definition
\eqref{eqn:tau_def}
of the $\tau_j$
(and the supporting definition
\eqref{eqn:taubar_def}
of the $\taubar_j$)
with $x_1,\ldots,x_{N-1}$
as in Lemma \ref{lem:limiting_waist_ratios}
(and understanding
$x_0=\zeta_0=x_N=\zeta_N=0$),
we
further estimate
\begin{equation}
\label{eqn:Fdiff_est}
\begin{aligned}
\widetilde{\force}_i
&=
m
  \left(
    x_{i-1}
        e^{\frac{\zeta_{i-1}-\zeta_n}{m}}
    -2\frac{x_{n-1} + N \bmod 2}{1 + N \bmod 2}x_i
        e^{\frac{\zeta_i-\zeta_n}{m}}
    +x_{i+1}
        e^{\frac{\zeta_{i+1}-\zeta_n}{m}}
  \right)
  -4x_i\zeta_n
  +O(1)
\\
&=
-x_{i-1}(\zeta_i-\zeta_{i-1})
  +x_{i+1}(\zeta_{i+1}-\zeta_i)
  -4x_i\zeta_n
  +O(1).
\end{aligned}
\end{equation}

We are now ready to state and prove
the following lemma,
which,
as anticipated above,
plays an indispensable role
in managing certain apparent obstructions
(that originate from the cokernel
of the model linear problem on the disc)
in the proof
of the existence
(Proposition \ref{pro:existence_of_fbms})
of an exactly free boundary minimal
graph over each initial surface
with $m$ sufficiently large.

\begin{lemma}[Control of forces and dislocations]
\label{lem:coker_control}
Let $N \geq 2$ be an integer.
There exist
$C=C(N)>0$
and
an invertible linear map
\begin{equation*}
\ptofd=\ptofd_N
\colon
\R^{N-1}_{\updownarrow+}
  \oplus
  \R^{N-1}_{\updownarrow-}
\to
\R^N_{\updownarrow-}
  \oplus
  \R^{N,0,0}_{\updownarrow+},
\end{equation*}
such that for every $c>0$
there exists
$m_0=m_0(N,c)$ such that
for every integer $m>m_0$
and for every
$
 (\vec{\zeta},\vec{\xi}\,)
 \in
 \IntervaL{-c,c}^{N-1}_{\updownarrow+}
  \times
  \IntervaL{-c,c}^{N-1}_{\updownarrow-}
$
we have
\begin{equation*}
\nm{\ptofd}+\nm{\ptofd^{-1}}
+
\abs[\big]{
  \tau_n^{-1}
  (
    \vec{\force}_{N,m,\vec{\zeta},\vec{\xi}},~
    \vec{\disloc}_{N,m,\vec{\zeta},\vec{\xi}}
  )
  -\ptofd(\vec{\zeta},\vec{\xi})
}
\leq
C,
\end{equation*}
where $\nm{\cdot}$
is the operator norm
with the Euclidean norm
on the corresponding domain and target.
\end{lemma}

\begin{proof}
We begin by defining some linear maps:
\begin{align*}
\begin{aligned}
S
  \colon
  \R^{N,0,0}
  &\to
  \R^{N-1}
  &&\mbox{ by } &
  (S\vec{v}\,)_i
    &\vcentcolon=
    v_i + v_{i+1},
\\
T=T^{(d)}
  \colon
  \R^{d+1}
  &\to
  \R^d
  &&\mbox{ by } &
  (T\vec{v}\,)_i
    &\vcentcolon=
    v_{i+1}-v_i,
\\
T_{0,0}=T^{(d+2)}_{0,0}
  \colon
  \R^{d+1}
  &\to
  \R^{d+2,0,0}
  &&\mbox{ by } &
  T_{0,0}(v)_i
    &\vcentcolon=
    \begin{cases}
      0
        &\mbox{if } i \in \{1,d+2\},
      \\
      v_i-v_{i-1}
        &\mbox{if } 2 \leq i \leq d+1,
    \end{cases}
\\
F
  \colon
  \R^{N-2}
  &\to
  \R^{N-1}
  &&\mbox{ by } &
  (Fv)_i
   &\vcentcolon=
    -x_{i-1}v_{i-1} + x_{i+1}v_i,
\end{aligned}
\end{align*}
where $d$ is any nonnegative integer,
$x_1,\ldots,x_{N_1}$
are as in Lemma \ref{lem:limiting_waist_ratios},
and we understand $x_0=v_0=x_N=v_{N-1}=0$ in the definition of $F$.
Then, $T|_{\R^N_{\updownarrow-}}=T^{(N-1)}|_{\R^N_{\updownarrow-}}$
is a bijection onto $\R^{N-1}_{\updownarrow+}$ and 
\begin{align*} 
T_{0,0}(\R^{N-1}_{\updownarrow-})
&\subseteq\R^{N,0,0}_{\updownarrow+},
&
T(\R^{N-1}_{\updownarrow+})
&\subseteq\R^{N-2}_{\updownarrow-},
&
S(\R^{N,0,0}_{\updownarrow+})
&\subseteq\R^{N-1}_{\updownarrow+},
&
F(\R^{N-2}_{\updownarrow \pm})
&\subseteq\R^{N-1}_{\updownarrow \mp}.
\end{align*}
Thus we can define
$
\ptofd=\ptofd_N
\colon
\R^{N-1}_{\updownarrow+}
  \oplus
  \R^{N-1}_{\updownarrow-}
\to
\R^N_{\updownarrow-}
  \oplus
  \R^{N,0,0}_{\updownarrow+},
$
by
\begin{equation*}
\ptofd(\vec{\zeta},\vec{\xi}\,)
\vcentcolon=
\left(
  T|_{\R^N_{\updownarrow-}}^{-1}
\Bigl(
\pi FT\vec{\zeta}-4\pi\zeta_n\vec{x}+S T_{0,0}\vec{\xi}\,
\Bigr), ~
  \frac{1}{2}T_{0,0}\vec{\xi}
\right),
\end{equation*}
which is obviously
(in view of the bounds in Lemma \ref{lem:limiting_waist_ratios})
a bounded linear map.
The asserted bound on
$
\tau_n^{-1}(\vec{\force},\vec{\disloc})
-\ptofd(\vec{\zeta},\vec{\xi})
$
follows by the further observations
\begin{align*}
\vec{\disloc}
&=
  \frac{\tau_n}{2}  T_{0,0}\vec{\xi},
&
\pi \tau_n \vec{\widetilde{\force}}
&=
  T\vec{\force}
  -2S\vec{\disloc},
&
\vec{\widetilde{\force}}
&=
  FT\vec{\zeta}
  -4\zeta_n\vec{x}
  +O(1),
\end{align*}
referring to the definition \eqref{eqn:disloc_def}
of $\vec{\disloc}$ for the first equality,
to the definition $\eqref{eqn:Ftilde_def}$
of $\vec{\widetilde{\force}}$
for the second equality,
and to estimate \eqref{eqn:Fdiff_est} for the final estimate.

Finally, to conclude the invertibility of $\ptofd$
we use the easily verified facts that 
$T_{0,0}|_{\R^{N-1}_{\updownarrow-}}$ is a bijection onto $\R^{N,0,0}_{\updownarrow+}$, 
that $\ker F^\intercal = \R\vec{x}$, 
where $F^\intercal$ is the transpose of $F$, 
and that the map 
\begin{equation*}
\R^{N-1}_{\updownarrow+}
\ni \vec{v}
\mapsto
(T\vec{v},v_n)
\in
\R^{N-2}_{\updownarrow-} \oplus \R
\end{equation*}
is bijective. 
For the bound on $\ptofd^{-1}$ we appeal again to the bounds in Lemma \ref{lem:limiting_waist_ratios}.
\end{proof}

\section{Quantitative geometry of the initial surfaces}
\label{app:geo_ests}
Here we collect some calculations which can be used to verify fundamental estimates used for both the construction and the index estimate.

\paragraph{Method of computation.}
Let $g$ be a Riemannian metric on $\R^3$,
$\varphi \colon S \to \R^3$ a two-sided embedding
of a compact smooth surface,
and $\nu \colon S \to \R^3$
a nowhere vanishing normal field,
not necessarily unit.
We consider the corresponding
scalar-valued second fundamental form
\begin{equation*}
A=
-\frac{1}{2}
  \varphi^*
  \mathscr{L}_{\widetilde{\nu}/\abs{\widetilde{\nu}}}g,
\end{equation*}
where $\mathscr{L}$ denotes Lie differentiation,
$\widetilde{\nu} \colon \R^3 \to \R^3$
is a smooth vector field such that
$\nu=\widetilde{\nu} \circ \varphi$
but is otherwise arbitrary,
and
$
 \abs{\widetilde{\nu}}
 \vcentcolon=
 \abs{\widetilde{\nu}}_g
$.
In this context
the above equality
can be written somewhat more concretely:
for any vector fields $V,W$ on $\Sigma$
\begin{equation*}
-2\abs{\nu}A(V,W)
=
(\nu g)(V\varphi,W\varphi)
  +g(V\varphi,W\nu)
  +g(W\varphi,V\nu),
\end{equation*}
where each of the vector fields
$\nu,V,W$
acts on the indicated argument
($g$, $\varphi$, or $\nu$)
by componentwise differentiation,
relative to the standard coordinates on $\R^3$,
and
of course we have also
\begin{equation*}
(\varphi^*g)(V,W)
=
(g \circ \varphi)(V\varphi,W\varphi).
\end{equation*}

In the following applications
we will have a global coordinate system
$(s^1,s^2)$
on $S$,
and we will construct $\nu$ as
$g$ cross product of the coordinate vector fields,
rescaled as convenient.
Moreover, we will have
$g=\Phi^*(dx^2+dy^2+dz^2)$,
with $\Phi$ as in \eqref{def:Phi},
so that
\begin{equation}
\label{eqn:Phi_pullback_metric}
g
=
g^\Phi
\vcentcolon=
\Phi^*(dx^2+dy^2+dz^2)
=
d\sigma^2
  + (1-\sigma)^2 \cos^2 \omega \, d\theta^2
  + (1-\sigma)^2 \, d\omega^2.
\end{equation}
In particular we will take
\begin{equation}
\label{eqn:cross_prod_normal}
\nu=
\bigl(
  \varphi^\theta_{,1}\varphi^\omega_{,2}
  -\varphi^\theta_{,2}\varphi^\omega_{,1}
\bigr) \, \partial_\sigma
+
\frac{
  \varphi^\omega_{,1}\varphi^\sigma_{,2}
  -\varphi^\sigma_{,1}\varphi^\omega_{,2}
}
  {(1-(\sigma \circ \varphi))^2 \cos^2 (\omega \circ \varphi)}
  \, \partial_\theta
+
\frac{
  \varphi^\sigma_{,1}\varphi^\theta_{,2}
  -\varphi^\theta_{,1}\varphi^\sigma_{,2}
}
  {(1-(\sigma \circ \varphi))^2}
  \, \partial_\omega
\end{equation}
up to arbitrary rescalings, which we will declare at due course.

\paragraph{$K_i$ regions in $(t,\vartheta)$ coordinates.}
In correspondence with each $\kappa_i$
defined in \eqref{def:kappa}
we define
$\widehat{\kappa}_i$ -- having
the same domain as $\kappa_i$
but target instead the domain of $\Phi$ -- by
the same rule as for $\kappa_i$
but without applying $\Phi$,
so that
$\kappa_i=\Phi \circ \widehat{\kappa}_i$.
We compute
\begin{align*}
\partial_t\widehat{\kappa}_i
&=
  (\tau_i \sinh t)
  (
    \cos \vartheta \, \partial_\sigma
    +\sin \vartheta \, \partial_\theta
    +\csch t \, \partial_\omega
  ),
\\
\partial_\vartheta\widehat{\kappa}_i
&=
  (\tau_i \cosh t)
  (
    -\sin \vartheta \, \partial_\sigma
    +\cos \vartheta \, \partial_\theta
  ).
\end{align*}
Abbreviating
\begin{align*}
r_i
&\vcentcolon=
  \tau_i \cosh t,
&
\sigma_i
&\vcentcolon=
  \widehat{\kappa}_i^*\sigma
  =
  \tau_i \cosh t \cos \vartheta,
&
\omega_i
&\vcentcolon=
  \widehat{\kappa}_i^*\omega
  =
  h^K_i + \tau_i t,
\end{align*}
we find
(using
$
 \kappa_i^*\met{\Sigma}
 =
 \widehat{\kappa}_i^*g^\Phi
$)
\begin{align}
\label{eqn:cat_met_calc}
r_i^{-2}\kappa_i^*\met{\Sigma}
&=
 dt^2+d\vartheta^2
 -\sigma_i(2-\sigma_i) \sech^2 t \, dt^2 
 \\\notag
  &\hphantom{{}={}}
 -\bigl(\sigma_i(2-\sigma_i) \cos^2 \omega_i +\sin^2 \omega_i\bigr)
 \cdot
  \bigl(
    \tanh^2 t \sin^2 \vartheta \, dt^2
    +\tanh t \sin 2\vartheta \, dt \, d\vartheta
    +\cos^2 \vartheta \, d\vartheta^2
  \bigr),
\end{align}
and we take for \eqref{eqn:cross_prod_normal}
(after a rescaling)
\begin{align*}
\nu_{\widehat{\kappa_i}}
&=
(1-\sigma_i) \sech t \cos \vartheta
    \, \partial_\sigma
  +(1-\sigma_i)^{-1} \sech t \sec^2 \omega_i 
      \sin \vartheta \, \partial_\theta
  -(1-\sigma_i)^{-1} \tanh t \, \partial_\omega.
\end{align*}
Then
(recalling in particular that $\nu_{\widehat{\kappa}_i}g^\Phi$
denotes the $(\sigma, \theta, \omega)$-componentwise normal derivative
of the ambient metric)
\begin{equation*}
\begin{aligned}
\abs{\nu_{\widehat{\kappa}_i}}
  &=
  \sqrt{
    1+\sech^2 t \tan^2 \omega_i \sin^2 \vartheta
    -\sigma_i(2-\sigma_i) \sech^2 t \cos^2 \vartheta
  },
\\[1ex]
\nu_{\widehat{\kappa}_i}g^\Phi
  &=
  2(1-\sigma_i) \cos \omega_i
  [
    \tanh t \sin \omega_i
    -(1-\sigma_i) \sech t
        \cos \omega_i
        \cos \vartheta
  ] \, d\theta^2
  \\
  &\hphantom{{}={}}
  -2(1-\sigma_i)^2 \sech t \cos \vartheta
    \, d\omega^2,
\\[1ex]
\partial_t \nu_{\widehat{\kappa}_i}
  &=
  -\tanh t \cos \vartheta
  [
    \tau_i \cos \vartheta
    +(1-\sigma_i) \sech t
  ] \, \partial_\sigma
  \\
  &\hphantom{{}={}}
  +(1-\sigma_i)^{-1}
    \sech t
    \sec^2 \omega_i
    \sin \vartheta
    [
      \tau_i(1-\sigma_i)^{-1} \sinh t \cos \vartheta
      -\tanh t
      +2 \tau_i \tan \omega_i
    ] \, \partial_\theta
  \\
  &\hphantom{{}={}}
    -(1-\sigma_i)^{-2} \sech t
    [
      \tau_i \sinh^2 t \cos \vartheta
      -(1-\sigma_i) \sech t
    ] \, \partial_\omega,
\\[1ex]
\partial_\vartheta \nu_{\widehat{\kappa}_i}
  &=
  \sin \vartheta
  [
    \tau_i \cos \vartheta
    -(1-\sigma_i) \sech t
  ] \, \partial_\sigma
  +
  \tau_i(1-\sigma_i)^{-2}
    \sinh t \sin \vartheta
    \, \partial_\omega
  \\
  &\hphantom{{}={}}
  +(1-\sigma_i)^{-1} \sec^2 \omega_i
  [
    \sech t \cos \vartheta
    -\tau_i(1-\sigma_i)^{-1} \sin^2 \vartheta
  ] \, \partial_\theta.
\end{aligned}
\end{equation*}
We then obtain
(recalling the choice of direction
for $\twoff{\Sigma}$ declared
below \eqref{eqn:init_surf_def})
\begin{equation}
\label{eqn:cat_twoff_calc}
\begin{aligned}
\frac{
  \abs{
    \nu_{\widehat{\kappa}_i}
    }
  }
  {\tau_i}
    \kappa_i^*
    \twoff{\Sigma}
  &=
  (-1)^{i-1}
  \Bigl[
    (1-\sigma_i)
    +2\tau_i
        \sinh t
        \tanh t
        \cos^3 \vartheta
    -2\tau_i(1-\sigma_i)
        \tan \omega_i
        \tanh t
        \sin^2 \vartheta
  \\
  &\hphantom{{}=+\Bigl[}
    +\tau_i (1-\sigma_i)^2
        \cos^2 \omega_i
        \sinh t \tanh t
        \sin^2 \vartheta
        \cos \vartheta
    +\tau_i (1-\sigma_i)^2
        \sech t
        \cos \vartheta
  \\
  &\hphantom{{}=+\Bigl[}
    -\tau_i (1-\sigma_i)
        \sin \omega_i
        \cos \omega_i
        \sinh^2 t \tanh t
        \sin^2 \vartheta
  \Bigr] \, dt^2
  \\
  &\hphantom{{}={}}
  +\Bigl[
    \tau_i(1-\sigma_i)^2
        \cos^2 \omega_i
        \sinh t
        \sin 2\vartheta 
        \cos \vartheta
      -2\tau_i
      \sinh t
      \sin 2\vartheta
      \cos \vartheta
    \\
  &\hphantom{{}=+\Bigl[}
      -\tau_i (1-\sigma_i)
          \tan \omega_i
          (\cos^2 \omega_i \sinh^2 t + 1)
          \sin 2\vartheta
  \Bigr] \, dt \, d\vartheta
  \\
  &\hphantom{{}={}}
  +\Bigl[ 
   (\sigma_i-1)
   +2\tau_i
       \cosh t
       \sin^2 \vartheta 
       \cos \vartheta
    +\tau_i(1-\sigma_i)^2
        \cos^2 \omega_i
        \cosh t
        \cos^3 \vartheta
  \\ 
  &\hphantom{{}=+\Bigl[}
    -\tau_i(1-\sigma_i)
        \sin \omega_i
        \cos \omega_i
        \sinh t \cosh t
        \cos^2 \vartheta
  \Bigr] \, d\vartheta^2.
\end{aligned}
\end{equation}
As a result, by taking $m$ sufficiently large in terms of
$N$, $\abs{\vec{\zeta}}$,
and $\abs{\vec{\xi}}$
we can ensure
the existence of some $C(N)>0$,
independent of $m$, $\vec{\zeta}$, and $\vec{\xi}$,
so that for any $R \in \intervaL{0,(2m)^{-1}}$
\begin{equation}
\label{eqn:cat_aux_fun_ests}
\begin{aligned}
\nm{r_i}_{
  C^2(\kappa_i^{-1}(K_i(R)), ~dt^2+d\vartheta^2)
  }
  +\nm{\sigma_i}_{
    C^2(\kappa_i^{-1}(K_i(R)), ~dt^2+d\vartheta^2)
    }
&\leq
C(N)R,
\\
\nm{\sin \omega_i}_{
    C^2(\kappa_i^{-1}(K_i(R)), ~dt^2+d\vartheta^2)
    }
&\leq
C(N)m\tau_1,
\\
\nm{
 \abs{\nu_{\widehat{\kappa}_i}} 
 -1
}_{C^1(\kappa_i^{-1}(K_i(R)), ~ dt^2+d\vartheta^2)}
&\leq
C(N)\tau_1,
\end{aligned}
\end{equation}
where
for the middle and final inequalities
we have used
\eqref{eqn:tau_size}
\eqref{eqn:height_bound},
and
\eqref{eqn:catr_init_def}.

\paragraph{$D_i$ regions in $(\sunder,\thunder)$ coordinates.}
Each $D_i$,
as defined in
\eqref{layers_with_half_cats},
is constructed in such a way
that $\Phi^{-1}(D_i)$ is a graph
over a subset of the $\{\omega=0\}$ plane.
We set
$h_i \vcentcolon= \omega|_{\Phi^{-1}(D_i)}$,
an appropriately symmetric extension of
a function from
\eqref{angular_height_function_one_bridge}
and \eqref{angular_height_function_two_bridges}
with the corresponding parameter values
as set from the data for the construction.
We therewith parametrize $\Phi^{-1}(D_i)$ by
\begin{equation*}
\varphi_i
\colon
(\sunder,\thunder)
\mapsto
\bigl(\sunder,\thunder,h_i(\sunder,\thunder)\bigr).
\end{equation*}
Then $\partial_\sunder \varphi_i=\partial_\sigma + h_{i,\sunder} \, \partial_\omega$ 
and $\partial_\thunder \varphi_i=\partial_\theta + h_{i,\thunder} \, \partial_\omega$,  
we have the metric
\begin{equation}
\label{eqn:dsc_met_calc}
\begin{aligned}
(\Phi \circ \varphi_i)^*(\met{\Sigma}|_{D_i})
  &=
  d\sunder^2
    + (1-\sunder)^2 \cos^2 h_i \, d\thunder^2
  \\
  &\hphantom{{}={}}
  +(1-\sunder)^2
\Bigl(
    h_{i,\sunder}^2
      \, d\sunder^2
    + 2h_{i,\sunder}h_{i,\thunder}
      \, d\sunder \, d\thunder
    +h_{i,\thunder}^2
      \, d\thunder^2
\Bigr),
\end{aligned}
\end{equation}
and we take  
\begin{align}\notag
\nu_{\varphi_i}
&=(1-\sunder)h_{i,\sunder} \, \partial_\sigma
  +\frac{h_{i,\thunder}}{(1-\sunder) \cos^2 h_i}
    \, \partial_\theta
  -(1-\sunder)^{-1} \, \partial_\omega.
\intertext{Thus
(recalling the ambient metric $g^\Phi$ from \eqref{eqn:Phi_pullback_metric}
and the fact that $\nu_{\varphi_i}g^\Phi$
denotes its $(\sigma,\theta,\omega)$-componentwise normal derivative)} 
\label{dsc_nu_norm}
\abs{\nu_{\varphi_i}}
&=\sqrt{
    1+\abs{dh_i}_{d\sunder^2+d\thunder^2}^2
    -\sunder(2-\sunder)h_{i,\sunder}^2
    +h_{i,\thunder}^2 \tan^2 h_i
  },
\\[1ex]\notag
\nu_{\varphi_i}g^\Phi
  &=
  (1-\sunder)
  \bigl[
    \sin 2h_i
    -2(1-\sunder)h_{i,\sunder} \cos^2 h_i
  \bigr] \, d\theta^2
  -2(1-\sunder)^2h_{i,\sunder} \, d\omega^2,
\\[1ex]\notag
\partial_\sunder \nu_{\varphi_i}
  &=
  [
    (1-\sunder)h_{i,\sunder\sunder}
    -h_{i,\sunder}
  ] \, \partial_\sigma
  -(1-\sunder)^{-2} \, \partial_\omega
  \\\notag
  &\hphantom{{}={}}
  +\left[
    \frac{h_{i,\sunder\thunder}}
      {(1-\sunder)\cos^2 h_i}
    +\frac{h_{i,\thunder}}
      {(1-\sunder)^2 \cos^2 h_i}
    +\frac{2h_{i,\sunder}h_{i,\thunder} \tan h_i}
      {(1-\sunder)\cos^2 h_i}
  \right] \, \partial_\theta,
\\[1ex]\notag
\partial_\thunder \nu_{\varphi_i}
  &=
  (1-\sunder)h_{i,\sunder\thunder} \, \partial_\sigma
  +
  \left[
    \frac{h_{i,\thunder\thunder}}
      {(1-\sunder) \cos^2 h_i}
    +\frac{2h_{i,\thunder}^2 \tan h_i}
      {(1-\sunder) \cos^2 h_i}
  \right] \, \partial_\theta,
\end{align}
and, recalling the direction chosen in defining
$\twoff{\Sigma}$
below \eqref{eqn:init_surf_def},
\begin{equation}
\label{eqn:dsc_twoff_calc}
\begin{aligned}
\abs{\nu_{\varphi_i}}
    (\Phi \circ \varphi_i)^*
    \twoff{\Sigma}|_{D_i}
  &=
  (-1)^i
  \Bigl[
    (\sunder-1)h_{i,\sunder\sunder}
    +2h_{i,\sunder}
    +(1-\sunder)^2 h_{i,\sunder}^3
  \Bigr] \, d\sunder^2
  \\
  &\hphantom{{}={}}
  +2\Bigl[ 
    (\sunder-1) h_{i,\sunder\thunder}
    +(1-\sunder)^2 h_{i,\sunder}^2 h_{i,\thunder}
    -(1-\sunder) h_{i,\sunder} h_{i,\thunder} \tan h_i
  \Bigr] \, d\sunder \, d\thunder
  \\
  &\hphantom{{}={}}
  +\Bigl[
    (\sunder-1)h_{i,\thunder\thunder}
    +(1-\sunder)^2 h_{i,\sunder} \cos^2 h_i
    -\frac{1}{2}(1-\sunder) \sin 2h_i
    \\
    &\hphantom{{}=+\Bigr[}
    -2(1-\sunder)h_{i,\thunder}^2 \tan h_i
    +(1-\sunder)^2 h_{i,\sunder}h_{i,\thunder}^2
  \Bigr] \, d\thunder^2.
\end{aligned}
\end{equation}
Referring to
\eqref{eqn:disc_def_fun_Phi}
and \eqref{eqn:height_bound},
we can take
$m$ sufficiently large in terms of
$\abs{\vec{\zeta}},\abs{\vec{\xi}}$
to ensure
\begin{equation*}
\nm[\Big]{\omega^B_{m,h^B_i}}_{
  C^3(
    \Lambda^\Phi_m, ~
    d\sunder^2+d\thunder^2
  )
  }
\leq
C(N)m\tau_1
\end{equation*}
for some $C(N)>0$,
independent of $m,\vec{\zeta},\vec{\xi},i$.
Similarly,
now recalling
\eqref{eqn:tau_size}
and
\eqref{eqn:cat_def_fun_Phi},
for each $q>1$
and each integer $k \geq 0$
we can take $m$ large enough in terms
of $q,\vec{\zeta},\vec{\xi}$
to ensure
\begin{equation*}
\nm[\Big]{\omega^{K,\pm}_{m,\tau_j,h^K_j,h^B_i}}_{
  C^k(
    \Lambda^\Phi_m
      \cap
      (\Phi \circ \varphi_i)^{-1}(D_i(m^{-q}))
      , ~
    d\sunder^2+d\thunder^2
  )
}
\leq
C(k,N)m^{\max \{kq, 1\}}\tau_1
\end{equation*}
for each $j$
such that $K_j \cap D_i \neq \emptyset$
and for some $C(k,N)>0$,
independent of $m,\vec{\zeta},\vec{\xi},i$.
In turn,
again
recalling \eqref{angular_height_function_one_bridge}
and \eqref{angular_height_function_two_bridges}
and taking $m$ sufficiently large
in terms of $q,\vec{\zeta},\vec{\xi}$,
we ensure
\begin{equation}
\label{eqn:dsc_def_fun_est}
\nm{h_i}_{
  C^k(
    (\Phi \circ \varphi_i)^{-1}(D_i(m^{-q})), ~
    d\sunder^2+d\thunder^2
  )
}
\leq
C(k,N)m^{k+\max \{kq,1\}}\tau_1.
\end{equation}
Last we focus on the region
\begin{equation*}
\Omega_i
\vcentcolon=
(\Phi \circ \varphi_i)^{-1}
\Bigl(
    D_i(\tfrac{1}{4m})
    \cap\bigl\{\textstyle\sqrt{x^2+y^2}\geq 1-\tfrac{5}{m}\bigr\}\Bigr).
\end{equation*}
From
\eqref{eqn:dsc_met_calc},
\eqref{dsc_nu_norm},
\eqref{eqn:dsc_twoff_calc},
and
\eqref{eqn:dsc_def_fun_est}
we find
\begin{equation}
\label{eqn:dsc_mc_est}
\nm[\big]{
  (\Phi \circ \varphi_i)^*H^\Sigma
  -
  (-1)^{i-1}(\partial_\sunder^2 + \partial_\thunder^2)h_i
}_{
  C^1(
    \Omega_i, ~
    m^2(d\sunder^2+d\thunder^2)
  )
}
\leq
Cm^2\tau_1,
\end{equation}
where we have also used \eqref{eqn:tau_size}
and we again assume
$m$ sufficiently large
in terms of
$N,m,\abs{\vec{\zeta}},\abs{\vec{\xi}}$.
Recalling 
\eqref{eqn:dislocated_matching}
and referring again to
\eqref{angular_height_function_one_bridge},
\eqref{angular_height_function_two_bridges},
\eqref{eqn:dsc_def_fun_est},
and \eqref{layers_with_half_cats}
we conclude with the estimate
\begin{equation}
\label{eqn:disloc_mc_est}
\nm[\big]{
  (\partial_\sunder^2 + \partial_\thunder^2)h_i
  -(-1)^{i-1}\disloc_i(\partial_\sunder^2+\partial_\thunder^2)
  [
    \psi_m^\sigma(\psi_m^+ - \psi_m^-)
  ]
}_{C^1(\Omega_i \cap \Lambda^\Phi_m, m^2(d\sunder^2+d\thunder^2))}
\leq
C(N)m^2\tau_1.
\end{equation}

\section{Analysis of the Jacobi operator on the models}

\paragraph{Catenoid.}
Given $T \in \R \cup \{\infty\}$,
we identify in the $(t,\vartheta)$-plane
the rectangle
\begin{align*}
\catdom_T&\vcentcolon=\{\abs{t} < T\} \cap \{\abs{\vartheta} < \tfrac{\pi}{2}\},
\intertext{and we define thereon the differential operator}
L_K&\vcentcolon=\partial_t^2 + \partial_\vartheta^2 + 2 \sech^2 t, 
\end{align*}
(where $\sech t=1/\cosh t$). 
If we parametrize half of
the standard (unit waist) catenoid
by $\catdom_\infty$ via
\[
 (t,\vartheta)
 \mapsto
 (
   \cosh t \, \cos \vartheta,
   \cosh t \, \sin \vartheta,
   t
 ),
 \]
then it has induced metric
$(\cosh^2 t)(dt^2+d\vartheta^2)$
and Jacobi operator equal to
$(\sech^2 t)L_K$. 
We also define the associated bilinear form 
\begin{align*}
Q_{\catdom_T}^\neum(u,v)
&\vcentcolon=
\int_{\catdom_T}
  \Bigl((\partial_t u)(\partial_tv)
    +(\partial_\vartheta u)
     (\partial_\vartheta v)
   -2uv \sech^2 t\Bigr) \, dt \, d\vartheta
\end{align*}
for $u,v \in \sob(\catdom_T, dt^2+d\vartheta^2)$ 
and, assuming $T$ finite
and recalling \eqref{eqn:H10} concerning how we encode Dirichlet boundary conditions,
its restriction
\begin{equation*}
Q_{\catdom_T}^\dir
\vcentcolon=
\left.Q_{\catdom_T}^\neum\right|_{
  \sobd{\{\abs{t}=T\}}(\catdom_T, dt^2+d\vartheta^2)
  \times
  \sobd{\{\abs{t}=T\}}(\catdom_T, dt^2+d\vartheta^2)
}.
\end{equation*}
Focusing, for instance, on the first quadratic form we say (consistently with the conventions stipulated in \cite[Section2]{CSWSpectral}) that $\lambda\in\R$ is an eigenvalue of $Q_{\catdom_T}^\neum$ with eigenfunction $u\in \sob(\catdom_T, dt^2+d\vartheta^2)$ if  
\[
Q_{\catdom_T}^\neum(u,v)=\lambda \langle u, v\rangle_{L^2} \quad \forall v\in \sob(\catdom_T, dt^2+d\vartheta^2),
\]
which happens if and only if
\[\left\{
\begin{aligned}
L_K u &= -\lambda u
  &&\text{ in } \catdom_T,
\\
\partial_t u &= 0
  &&\text{ on } \{\abs{t}=T\},
\\
\partial_\vartheta u &= 0
  &&\text{ on } \{\abs{\vartheta}=\tfrac{\pi}{2}\}.
\end{aligned}\right.
\]
We write 
\begin{align*}
\cycgrp{\{t=0\}}
&=
\sk[\big]{
  (t,\vartheta) \mapsto (-t,\vartheta)
},
\\
\cycgrp{\{\vartheta=0\}}
&=
\sk[\big]{
  (t,\vartheta) \mapsto (t,-\vartheta)
},
\\
\cycgrp{\{t=\vartheta=0\}}
&=
\sk[\big]{
  (t,\vartheta) \mapsto -(t,\vartheta)
}
\end{align*}
for the order-$2$ groups of isometries
of $(\catdom_T,dt^2+d\vartheta^2)$
generated by reflection through,
respectively,
$\{t=0\}$,
$\{\vartheta=0\}$
and the origin.

\begin{lemma}[Low spectrum of $L_K$]
\label{lem:low_spec_cat}
There exist $c,T_0>0$ such that for each finite $T>T_0$
\begin{leqnoalign} 
\tag*{(i)}\label{itm:low_spec_cat:dir}
\lambda_1(Q_{\catdom_T}^\dir)
  =
  \eigenvalsym{Q_{\catdom_T}^\dir}{1}
      {\cycgrp{\{t=0\}}}{+} 
  &=\mathrlap{
  \eigenvalsym{Q_{\catdom_T}^\dir}{1}
      {\cycgrp{\{\vartheta=0\}}}{+} 
  =
  \eigenvalsym{Q_{\catdom_T}^\dir}{1}
      {\cycgrp{\{t=\vartheta=0\}}}{+}
  \xrightarrow[T \to \infty]{}
  -1}
\\[1ex]
\tag*{(ii)}\label{itm:low_spec_cat:neum_t_even}
\eigenvalsym{Q_{\catdom_T}^\neum}{2}{\cycgrp{\{t=0\}}}{+}
<0&<\eigenvalsym{Q_{\catdom_T}^\neum}{3}{\cycgrp{\{t=0\}}}{+},
& 
cT^{-2}
&<\eigenvalsym{Q_{\catdom_T}^\neum}{3}{\cycgrp{\{t=0\}}}{+},
\\
\tag*{(iii)}\label{itm:low_spec_cat:neum_t_odd}
0&<\eigenvalsym{Q_{\catdom_T}^\neum}{1}{\cycgrp{\{t=0\}}}{-},
&
cT^{-2}
&<\eigenvalsym{Q_{\catdom_T}^\neum}{2}{\cycgrp{\{t=0\}}}{-},
\\
\tag*{(iv)}\label{itm:low_spec_cat:neum_theta_even}
\eigenvalsym{Q_{\catdom_T}^\neum}{1}{\cycgrp{\{\vartheta=0\}}}{+}
<0&<\eigenvalsym{Q_{\catdom_T}^\neum}{2}{\cycgrp{\{\vartheta=0\}}}{+},
&
cT^{-2}
&<\eigenvalsym{Q_{\catdom_T}^\neum}{3}{\cycgrp{\{\vartheta=0\}}}{+},
\\
\tag*{(v)}\label{itm:low_spec_cat:neum_origin_even}
\eigenvalsym{Q_{\catdom_T}^\neum}{1}{\cycgrp{\{t=\vartheta=0\}}}{+}
<0&<\eigenvalsym{Q_{\catdom_T}^\neum}{2}{\cycgrp{\{t=\vartheta=0\}}}{+},
&
cT^{-2}
&<\eigenvalsym{Q_{\catdom_T}^\neum}{2}{\cycgrp{\{t=\vartheta=0\}}}{+}.
\end{leqnoalign}
\end{lemma}

\begin{proof}
By Fourier decomposition in $\vartheta$
we can reduce the problem
(in both the Dirichlet and Neumann cases)
to the study of the eigenvalues
of the $\vartheta$-invariant eigenfunctions:
more precisely, we will appeal to the fact
that each eigenfunction of
$Q_{\catdom_T}^\dir$ or $Q_{\catdom_T}^\neum$
with eigenvalue $\lambda$
is a linear combination
of (finitely many) functions of the form
$u_{\lambda,n}(t)e^{i n\vartheta}$,
where
$i$ is the usual imaginary unit
(rather than an integral index,
as elsewhere in this article),
$n \in \Z$,
and $u_{\lambda,n}(t)$ is a $\vartheta$-invariant
function satisfying
\begin{equation*}
L_K u_{\lambda,n} = (n^2 - \lambda)u
\end{equation*}
and the correspondingly Dirichlet or Neumann
boundary conditions.
Moreover, we can further split
each $\vartheta$-invariant eigenspace
into its $t$-even and $t$-odd subspaces.

To proceed we write
$\lambda^{\dir,t,+}_k(T)$
for the $k$\textsuperscript{th} eigenvalue
of $Q_{\catdom_T}^\dir$
restricted to 
$\vartheta$-invariant functions
that are even in $t$,
and we similarly write
$\lambda^{\neum,t,\pm}_k(T)$
for the $k$\textsuperscript{th} eigenvalue
of $Q_{\catdom_T}^\neum$
restricted to $\vartheta$-invariant functions
that are either even in $t$, for the $+$ choice in $\pm$,
or odd in $t$, for the $-$ choice. 
We will show that for some $c,T_0>0$ and all $T>T_0$
\begin{align}
\label{eqn:D+}
\lim_{T\to \infty}\lambda^{\dir,t,+}_1(T)
&=-1,
\\[.5ex]
\label{eqn:N+<}
-2<\lambda^{\neum,t,+}_1(T)
&<-1<0<\lambda^{\neum,t,+}_2(T),
\\
\label{eqn:N-<}
\lambda^{\neum,t,-}_1(T)
&>0,
\\[1.5ex]
\label{eqn:N+>}
\lambda^{\neum,t,+}_2(T)
&>cT^{-2},
\\
\label{eqn:N->}
\lambda^{\neum,t,-}_2(T)
&>cT^{-2}.
\end{align}
Provisionally assuming the claims
\eqref{eqn:D+}--\eqref{eqn:N->},
we now observe
how they imply the lemma.
First, item \ref{itm:low_spec_cat:dir}
follows from \eqref{eqn:D+},
given that any first eigenfunction
must be $\vartheta$-invariant
and $t$-even.
Next,
\eqref{eqn:N+<} and \eqref{eqn:N-<}
imply
(assuming $T>T_0$)
that the sum of the eigenspaces
of $Q^\neum_{\catdom_T}$ of nonpositive eigenvalue
is spanned by
$u^{\neum,t,+}_1$
and $u^{\neum,t,+}_1 \sin \vartheta$,
where $u^{\neum,t,+}_1$
is the first $t$-even, $\vartheta$-invariant
eigenfunction of $Q_{\catdom_T}^\neum$
(noting that
$u^{\neum,t,+}_1 \cos \vartheta$
violates the homogeneous Neumann condition
along $\{\abs{\vartheta}=\pi/2\}$,
while
$u^{\neum,t,+}_1 e^{in\vartheta}$
has eigenvalue
$\lambda_1^{\neum,t,+}(T) + n^2 > 0$
whenever $\abs{n} \geq 2$).
This implies the first part of
\ref{itm:low_spec_cat:neum_t_even} and \ref{itm:low_spec_cat:neum_t_odd},
and, since $\sin \vartheta$ is $\vartheta$-odd,
also of \ref{itm:low_spec_cat:neum_theta_even} and \ref{itm:low_spec_cat:neum_origin_even}.
By additionally applying
\eqref{eqn:N+>} and \eqref{eqn:N->}
we obtain also
the second part 
for each of claims \ref{itm:low_spec_cat:neum_t_even}--\ref{itm:low_spec_cat:neum_origin_even}.

It remains to prove claims \eqref{eqn:D+}--\eqref{eqn:N->}.
The $\vartheta$-invariant solutions of $L_K u = \gamma^2 u$ with $\gamma^2 \in \R$ are spanned by
\begin{equation}
\label{eqn:cat_rot_inv_efs_wo_bcs}
\begin{aligned}
u_\gamma^+(t)
&\vcentcolon=
\begin{cases}
1 - t \tanh t
  &\mbox{for } \gamma=0
\\
\sech t
  &\mbox{for } \gamma=1
\\
\gamma \cosh \gamma t - \sinh \gamma t \, \tanh t
  &\mbox{for }
    \gamma \in \interval{0,1} \cup \interval{1,\infty}
\\
\frac{\gamma}{i}
    \cos \frac{\gamma}{i} t
  - \sin \frac{\gamma}{i} t \, \tanh t
  &\mbox{for } \gamma \in i \interval{0,\infty},
\end{cases}
\\[1ex]
u_\gamma^-(t)
&\vcentcolon=
\begin{cases}
\tanh t
  &\mbox{for } \gamma=0
\\
\sinh t + t \sech t
  &\mbox{for } \gamma=1
\\
\gamma \sinh \gamma t - \cosh \gamma t \, \tanh t
  &\mbox{for }
    \gamma \in \interval{0,1} \cup \interval{1,\infty}
\\
\frac{\gamma}{i}
    \sin \frac{\gamma}{i} t
  + \cos \frac{\gamma}{i} t \, \tanh t
  &\mbox{for } \gamma \in i \interval{0,\infty},
\end{cases}
\end{aligned}
\end{equation}
where $i$ continues, in this proof, to denote the usual imaginary unit.
Note that $u_\gamma^\pm$ is a $t$-even/$t$-odd, $\vartheta$-invariant eigenfunction
of $Q_{\catdom_T}^\dir$ or $Q_{\catdom_T}^\neum$
whenever it satisfies, respectively,
the homogeneous Dirichlet or Neumann boundary condition
on $\{\abs{t}=T\}$,
which we proceed to impose.

Considering first $\gamma=0$ and $\gamma=1$,
it is easy to see that
for all $T$ sufficiently large
neither $u_\gamma^+$ nor $u_\gamma^-$
satisfies either the homogeneous Dirichlet
or homogeneous Neumman condition,
and so
neither
$L_K u = 0$ nor $L_K u = u$
has a nontrivial $\vartheta$-invariant solution
satisfying either $u(T)=u(-T)=0$
or $(\partial_t u)(T)=(\partial_t u)(-T)=0$,
and we accordingly henceforth assume $\gamma \not \in \{0,1\}$.

Suppose now that $\gamma \in \interval{0,1} \cup \interval{1,\infty}$.
From \eqref{eqn:cat_rot_inv_efs_wo_bcs} we find that 
\begin{alignat}{3}
\label{eqn:u+condition}
u^+_\gamma(T)&=0=u^+_\gamma(-T) &
&\Leftrightarrow &
\tanh\gamma T&=\gamma\coth T,
\\[.5ex]\label{eqn:u-condition}
u^-_\gamma(T)&=0=u^-_\gamma(-T) &
\quad&\Leftrightarrow\quad & 
\gamma\tanh\gamma T&=\tanh T.
\end{alignat} 
Condition \eqref{eqn:u-condition} for $u_\gamma^-$ holds only for the excluded value $\gamma=1$,
while condition \eqref{eqn:u+condition} for $u_\gamma^+$ is equivalent to
$\tanh s = s T^{-1} \coth T$ (with $s=\gamma T$) and so,
assuming $T$ sufficiently large
in terms of a universal constant,
has exactly one strictly positive solution,
$\gamma=\tanh T \, \tanh \gamma T < 1$,
which moreover converges to $1$ as $T \to \infty$.
We have in particular established
\eqref{eqn:D+}.

Continuing to assume $\gamma \in \interval{0,1} \cup \interval{1,\infty}$, 
we next find from \eqref{eqn:cat_rot_inv_efs_wo_bcs} that
\begin{align*}
(\partial_t u^\pm_\gamma)(T)&=0=(\partial_t u^\pm_\gamma)(-T) 
\quad \Leftrightarrow \quad
f_{\pm,T}(\gamma)=0,
\shortintertext{where} 
f_{+,T}(\gamma)
&\vcentcolon=
  \gamma^2
  -\gamma \coth \gamma T \, \tanh T
  -\sech^2 T,
\\
f_{-,T}(\gamma)
&\vcentcolon=
  \gamma^2
  -\gamma \tanh \gamma T \, \tanh T
  -\sech^2 T.
\end{align*} 
Since $f_{+,T}|_{\interval{0,1}}<0$, 
in order for $f_{+,T}(\gamma)$ to vanish
we must have $\gamma>1$, but
\begin{equation*}
f_{+,T}'(\gamma)
=
2\gamma
  -\coth \gamma T \, \tanh T
  +\gamma T \csch^2 \gamma T \, \tanh T
>
0
\quad\mbox{ on } \interval{1,\infty},
\end{equation*}
so $f_{+,T}(\gamma)$
has exactly one positive zero,
$\gamma>1$,
which must then moreover
tend to $1$ as $T \to \infty$.
This establishes
\eqref{eqn:N+<}.
On the other hand,
$f_{-,T}(\gamma)=0$
is equivalent to
\begin{equation*}
s \tanh s \, \tanh T
  =
  \frac{s^2}{T} - T \sech^2 T,
\quad
s=\gamma T,
\end{equation*}
which has exactly one positive solution,
but 
$f_{-,T}(1)=0$
and $\gamma=1$ has been excluded,
yielding
\eqref{eqn:N-<}.

Finally,
now assuming $\gamma \in \interval{0,\infty}$,
from \eqref{eqn:cat_rot_inv_efs_wo_bcs} we also find
\begin{alignat}{3}
\label{eqn:ui+condition}
(\partial_t u^+_{i\gamma})(T)&=0=(\partial_t u^+_{i\gamma})(-T) 
&&\Leftrightarrow&
\gamma^2 + \gamma \cot \gamma T \, \tanh T + \sech^2 T
&=0,
\\[.5ex]\label{eqn:ui-condition}
(\partial_t u^-_{i\gamma})(T)&=0=(\partial_t u^-_{i\gamma})(-T) 
&\quad&\Leftrightarrow\quad&
\gamma^2-\gamma \tan \gamma T \, \tanh T + \sech^2 T
&=0.
\end{alignat}
Condition \eqref{eqn:ui+condition} for $u_{i\gamma}^+$ is equivalent to
\begin{equation*}
-s \cot s \tanh T=\frac{s^2}{T} + T \sech^2 T,
\quad
s=\gamma T,
\end{equation*}
whose smallest positive solution $s$
tends to $\frac{\pi}{2}$,
the smallest positive zero of $s \cot s$,
as $T \to \infty$, establishing \eqref{eqn:N+>},
while condition \eqref{eqn:ui-condition} for $u_{i\gamma}^-$ is equivalent to
\begin{equation*}
s \tan s \tanh T = \frac{s^2}{T} + T \sech^2 T,
\quad
s=\gamma T,
\end{equation*}
whose smallest positive solution
tends to $0$ as $T \to \infty$
and whose next smallest positive solution
must be at least $\pi$, the least positive zero of $s \tan s$.
This completes the proof of \eqref{eqn:N->}. 
\end{proof}

For each $\alpha,\beta \in \interval{0,1}$,
each integer $k \geq 0$,
and any functions $u \colon \catdom_\infty \to \R$,
$f \colon \partial \catdom_\infty \to \R$
we define
\begin{align*}
\nm{u}_{k,\alpha,\beta}
&\vcentcolon=
  \nm{u \sech \beta t}_{k,\alpha},
&
\nm{f}_{k,\alpha,\beta}
&\vcentcolon=
  \nm{f \sech \beta t}_{k,\alpha}
\end{align*}
and set
\begin{alignat*}{2}
C^{k,\alpha,\beta}_{\mathrm{sym}}(\catdom_\infty)
  &\vcentcolon=
  \{
    u \in C^{k,\alpha}(\catdom_\infty)
    & &\st
    u(t,\vartheta) \equiv u(t,-\vartheta)
     \text{ and } 
    \nm{u}_{k,\alpha,\beta}<\infty
  \},
\\[1ex]
C^{k,\alpha,\beta}_{\mathrm{sym}}(\partial\catdom_\infty)
  &\vcentcolon=
  \{
    f \in C^{k,\alpha}(\partial \catdom_\infty)
    & &\st
    f(t,\vartheta) \equiv f(t,-\vartheta)
     \text{ and } 
    \nm{f}_{k,\alpha,\beta}<\infty
  \}.
\end{alignat*}
We also write $\eta^{\catdom_\infty}$
for the outward unit conormal
on $\partial \catdom_\infty$.

\begin{lemma}[Invertibility of $L_K$]\label{lem:cat_model_sol}
Given $\alpha, \beta \in \interval{0,1}$ there exists a bounded linear map
\begin{equation*}
R_K
\colon
C^{0,\alpha,\beta}_{\mathrm{sym}}(\catdom_\infty)
  \oplus
  C^{1,\alpha,\beta}_{\mathrm{sym}}(\partial\catdom_\infty)
\to
C^{2,\alpha,\beta}_{\mathrm{sym}}(\catdom_\infty)
\end{equation*}
such that
for all
$(E,f)$ in the domain of $R_K$
\begin{equation*}
\left\{
\begin{aligned}
L_K R_K(E,f)
&=E
&&\text{ in } \catdom_\infty
\\
\eta^{\catdom_\infty}R_K(E,f)
&=f
&&\text{ on } \partial\catdom_\infty.
\end{aligned}
\right.
\end{equation*}
\end{lemma}

\begin{proof}
By constructing 
$v \in C^{2,\alpha}_{\mathrm{sym}}(\catdom_\infty)$
with
$
 (\partial_\vartheta v)|_{
     \{
       \abs{\vartheta}
       =
       \pi/2
     \}
  }
=
f
$
we reduce the problem to the case
of homogeneous boundary data.
Using separation of variables
one can directly solve the problem and prove the
$C^0$ estimate with the required decay;
all claims then follow by elliptic regularity cast in weighted spaces.
Full details can be found in
the proof of
\cite[Proposition 6.9]{KapouleasWiygul17}.
\end{proof}

\begin{lemma}
[Partial trace inequality on $\catdom_T$]
\label{lem:cat_trace}
There exists a constant $C>0$
(independent of $T$)
such that for any finite $T>0$
and all
$u \in \sob(\catdom_T, dt^2+d\vartheta^2)$
\begin{equation*}
\int_{\{\abs{\vartheta}=\pi/2\} \cap \partial\catdom_T}
  u^2 \, dt
\leq
C\int_{\catdom_T}
\Bigl(\abs{u}\abs{du}_{dt^2+d\vartheta^2}+u^2\Bigr)
\, dt \, d\vartheta.
\end{equation*}
\end{lemma}

\begin{proof}
Apply the divergence theorem on $\catdom_T$ to the vector field
$u^2 \sin \vartheta \, \partial_\vartheta$.
\end{proof}

\paragraph{Perforated rectangles.}
Given $0<r<\pi/2$, we define the following rectangle and perforated rectangles in the $(\sunder,\thunder)$-plane, as visualized in Figure \ref{fig:perforated}. 
\begin{align*}
\margindom
&\vcentcolon=
  \{0 < \sunder < 3\}
    \cap
    \{-\pi/2 < \thunder < \pi/2\},
\\
\margindom_r
&\vcentcolon=
  \margindom
    \setminus
    \{
        \sunder^2
          +(
            \thunder-\pi/2
          )^2
        \leq
        r^2
    \},
\\
\margindom_{r,r}
&\vcentcolon=
  \margindom_r
    \setminus
    \{
        \sunder^2
          +(
            \thunder+\pi/2
          )^2
        \leq
        r^2
    \}.
\end{align*}
For $\Omega$ any of the preceding three domains
we denote by $Q_\Omega^\neum$
the bilinear form
\begin{equation*}
Q_\Omega^\neum(u,v)
\vcentcolon=
\int_\Omega
\Bigl(
   (\partial_\sunder u)(\partial_\sunder v)
    +(\partial_\thunder u)(\partial_\thunder v)
\Bigr)\, d\sunder \, d\thunder
\end{equation*}
defined for $u,v \in \sob(\Omega, d\sunder^2+d\thunder^2)$, 
corresponding to the Neumann Laplacian
on $\Omega$ in the sense of \eqref{eqn:bilinear_form_def}. 

\begin{figure}
\pgfmathsetmacro{\rpar}{0.75}
\begin{tikzpicture}[line cap=round,line join=round]
\fill[black!33](0,-pi/2)rectangle(3,pi/2);	
\draw[->](0,0)--(3.5,0)node[right]{$\sunder$};	
\draw[->](0,-pi/2)--(0,2.2)node[below right,inner sep=0]{~$\thunder$};	
\draw plot[plus](0,0)node[below left]{$0$};
\draw plot[vdash](3,0)node[below]{$3$};
\draw plot[hdash](0,pi/2)node[left]{$\frac{\pi}{2}$};
\draw plot[plus](0,-pi/2)node[left]{$-\frac{\pi}{2}$};
\path(1.66,0.5)node[anchor=base]{$\margindom$};
\end{tikzpicture}
\hfill
\begin{tikzpicture}[line cap=round,line join=round]
\fill[black!33]
(0,pi/2-\rpar)arc(-90:0:\rpar)-|(3,-pi/2)-|cycle;
\draw[->](0,0)--(3.5,0)node[right]{$\sunder$};	
\draw[->](0,-pi/2)--(0,2.2)node[below right,inner sep=0]{~$\thunder$};	
\draw plot[plus](0,0)node[below left]{$0$};
\draw plot[vdash](3,0)node[below]{$3$};
\draw plot[hdash](0,pi/2)node[left]{$\frac{\pi}{2}$};
\draw plot[plus](0,-pi/2)node[left]{$-\frac{\pi}{2}$};
\path(1.66,0.5)node[anchor=base]{$\margindom_{r}$};
\draw[latex-latex](0, pi/2)--++(0:\rpar) node[midway,below]{$r$};
\end{tikzpicture}
\hfill
\begin{tikzpicture}[line cap=round,line join=round]
\fill[black!33]
(0,pi/2-\rpar)arc(-90:0:\rpar)-|(3,-pi/2)-|(\rpar,-pi/2)arc(0:90:\rpar)--cycle;
\draw[->](0,0)--(3.5,0)node[right]{$\sunder$};	
\draw[->](0,-pi/2)--(0,2.2)node[below right,inner sep=0]{~$\thunder$};	
\draw plot[plus](0,0)node[below left]{$0$};
\draw plot[vdash](3,0)node[below]{$3$};
\draw plot[hdash](0,pi/2)node[left]{$\frac{\pi}{2}$};
\draw plot[plus](0,-pi/2)node[left]{$-\frac{\pi}{2}$};
\path(1.66,0.5)node[anchor=base]{$\margindom_{r,r}$};
\draw[latex-latex](0, pi/2)--++(0:\rpar) node[midway,below]{$r$};
\draw[latex-latex](0,-pi/2)--++(0:\rpar) node[midway,above]{$r$};
\end{tikzpicture}
\caption{The open domains $\margindom$, $\margindom_{r}$ and $\margindom_{r,r}$.}%
\label{fig:perforated}%
\end{figure}

\begin{lemma}
[Low spectrum of Neumann Laplacian
on $\margindom_r$ and $\margindom_{r,r}$]
\label{lem:low_spec_margin}
There exist
$c, r_0>0$
such that for each $0<r<r_0$
and each
$
 \Omega
 \in 
 \{\margindom_r,\margindom_{r,r}\}
$
\begin{align*}
\ind(Q_\Omega^\neum)&=0, &
\nul(Q_\Omega^\neum)&=1, &
\lambda_2(Q_\Omega^\neum)&>c.
\end{align*}
In particular $\lambda_2(Q_\Omega^\neum)$
is the lowest positive eigenvalue
of $Q_\Omega^\neum$.
\end{lemma}

\begin{proof}
The asserted values
of the index and nullity
hold for the Neumann Laplacian on any connected domain
with Lipschitz boundary.
Then suppose the remaining claim failed,
say on $\margindom_r$
(the proof for $\margindom_{r,r}$ being essentially the same).
Then there would exist sequences
of strictly positive reals
$r_n \downarrow 0$,
and
$\lambda^{(n)} \downarrow 0$
and a sequence of functions
$u_n \colon \margindom_{r_n} \to \R$
such that
$u_n$ is an $L^2$-normalized
eigenfunction of $Q_{\margindom_{r_n}}^\neum$
with eigenvalue $\lambda^{(n)}$;
in particular each $u_n$
is $L^2$-orthogonal to the constants
on $\margindom_{r_n}$.

Fix, independently of $n$,
a smooth radial cutoff function $\psi$
on the $(\sunder,\thunder)$ plane
that takes the constant value $1$
on the closed disc of radius $10^{-2}$
and center $(0,\frac{\pi}{2})$
and the constant value $0$
outside the open disc of radius $10^{-1}$
and the same center.
We take $n$ large enough that $r_n<10^{-2}$.
By conformal inversion $E_n$
across the circle
containing the boundary of the perforation
we extend each $u_n$ to
$\overline{u}_n$
on all of $\margindom$
by
\begin{equation*}
\overline{u}_n|_{
    \margindom \setminus \margindom_{r_n}
}(\sunder,\thunder)
\vcentcolon=
\begin{cases}
(\psi u_n)(E_n(\sunder,\thunder))
  &\mbox{for }
  (\sunder,\thunder) \neq (0,\frac{\pi}{2})
\\
0
  &\mbox{for } (\sunder,\thunder)=(0,\frac{\pi}{2}),
\end{cases}
\end{equation*}
where $\psi u_n$ is defined on the entire plane,
vanishing outside the support of $\psi$.
Using the conformal invariance in dimension two
of the Dirichlet energy,
it follows that
$
 \nm{\overline{u}_n}_{H^1(\margindom)}
 \leq
 \sqrt{2}\nm{\psi u_n}_{H^1(\margindom_{r_n})}
$,
and so the $\overline{u}_n$
are bounded in $H^1(\margindom)$,
uniformly in $n$,
by a constant depending on $\psi$.

We may assume (by tacitly passing to a subsequence)
that $\{\overline{u}_n\}$ converges weakly in $H^1$
and strongly in $L^2$ to a function
$\overline{u} \in H^1(\margindom)$.
Writing $\chi_n$ for the indicator function
of $\margindom_{r_n}$,
we in turn have
\begin{align*}
0
  =
  \sk{u_n,1}_{L^2(\margindom_{r_n})}
  =
  \sk{\chi_n \overline{u}_n,1}_{L^2(\margindom)}
&\xrightarrow[n \to \infty]{}
  \sk{\overline{u},1}_{L^2(\margindom)},
\\
1
  =
  \nm{u_n}_{L^2(\margindom_{r_n})}^2
  =
  \sk{\chi_n \overline{u}_n, \overline{u}_n}_{
      L^2(\margindom)}
&\xrightarrow[n \to \infty]{}
  \nm{\overline{u}}_{L^2(\margindom)}^2,
\\
\lambda^{(n)}
    \sk{
      u_n,
      \overline{u}|_{\margindom_{r_n}}
    }_{L^2(\margindom_{r_n})}
  =
  \sk{
    \nabla u_n,
    (\nabla \overline{u})|_{\margindom_{r_n}}
  }_{L^2(\margindom_{r_n})}
  =
  \sk{
    \chi_n \nabla \overline{u}_n,
    \nabla \overline{u}
  }_{L^2(\margindom)}
&\xrightarrow[n \to \infty]{}
  \nm{\nabla \overline{u}}_{L^2(\margindom)}^2.
\end{align*}
Consequently
$\overline{u}$ is nonconstant
with vanishing weak gradient,
which contradiction concludes the proof.
\end{proof}

\begin{lemma}
[Partial trace inequality on
$\margindom_r$ and $\margindom_{r,r}$]
\label{lem:perf_rect_trace}
There exists $C>0$
such that
for any $r<\pi/2$,
for each
$
 \Omega
 \in
 \{\margindom_r, \margindom_{r,r}\}
$,
and for all
$u \in \sob(\Omega, d\sunder^2 + d\thunder^2)$
\begin{equation*}
\int_{\{\sunder=0\} \cap \partial \Omega}
    u^2 \, d\thunder
\leq
C\int_\Omega
\Bigl(\abs{u}\abs{du}_{d\sunder^2+d\thunder^2}+u^2\Bigr)
\, d\sunder \, d\thunder.
\end{equation*}
\end{lemma}

\begin{proof}
Apply the divergence theorem
on $\Omega$
to the vector field
$u^2(\frac{\sunder}{3}-1) \, \partial_\sunder$,
which is never inward
(has nonnegative inner product
with the outward unit normal)
along $\partial \Omega$.
\end{proof}

\section{Bounds on spectral shift}
Let $(\Omega,\dbdy\Omega,\nbdy\Omega,\rbdy\Omega)$
and $(\grp,\twsthom)$
be as recalled at the beginning of Section \ref{sec:Index} (see \cite{CSWSpectral}),
and in place of $(g,q,r)$ assume we are given two triples of data
$(g_i,q_i,r_i)$, for $i=1,2$.
Recalling \eqref{eqn:bilinear_form_def}, set
\begin{equation*}
T_i=T[\Omega,g_i,q_i,r_i,\dbdy\Omega,\nbdy\Omega,\rbdy\Omega];
\end{equation*}
we shall always associate with $T_i$,
when speaking of its eigenvalues and eigenfunctions,
the corresponding $g_i$.
We assume further that 
$C_\Omega^{\mathrm{tr}}>0$
is a constant such that
\begin{equation}
\label{eqn:trace_inequality}
\int_{\rbdy\Omega} u^2 \, \hausint{d-1}{g_1}
\leq
C_\Omega^{\mathrm{tr}}
\int_\Omega
\left(
 \abs{u}\abs{du}_{g_1}
 +u^2
\right) \, \hausint{d}{g_1}
\quad
\forall
u \in \sob(\Omega,g_1).
\end{equation}
Although such $C_\Omega^{\mathrm{tr}}$
exists for any Lipschitz domain $\Omega$,
in our applications
it will be provided by either
Lemma \ref{lem:cat_trace} or Lemma \ref{lem:perf_rect_trace}.
In fact each of these lemmata
concerns a certain infinite family of domains
and asserts
that a single constant can be chosen
for which \eqref{lem:cat_trace} holds
over the entire family
(in other words,
$C^{\mathrm{tr}}_\lipdom$ can be chosen
independently of $\lipdom$ in the family);
since we will also have uniform control
over the metrics and Schr\"{o}dinger
and Robin potentials appearing in these applications, 
Lemma \ref{lem:quant_spec_shift}
will thereby deliver uniform bounds on
the corresponding spectral shift.
(We will also apply Lemma \ref{lem:quant_spec_shift}
with $(\lipdom,g_1)$ a Euclidean disc,
but in this case we will have
$\rbdy\lipdom=\emptyset$,
so that \eqref{eqn:trace_inequality}
holds vacuously.)

For the statement of the following lemma
we also introduce the abbreviations
\begin{align*}
\widetilde{g}
&\vcentcolon=
  \nm{g_2-g_1}_{C^0(\Omega,g_1)},
\\
\widetilde{q}
&\vcentcolon=
  \widetilde{g}\nm{q_1}_{C^0(\Omega)}
    +\nm{q_2-q_1}_{C^0(\Omega)},
\\
\widetilde{r}
&\vcentcolon=
  C_\Omega^{\mathrm{tr}}
    (
      1
      +C_\Omega^{\mathrm{tr}}\nm{r_1}_{C^0(\Omega)}
    )
    (1+\nm{q_1}_{C^0(\Omega)})
    (\widetilde{g} 
      +\nm{r_2-r_1}_{C^0(\partial\Omega)}
    ).
\end{align*}

\begin{lemma}
\label{lem:quant_spec_shift}
With notation and assumptions
as in the preceding paragraph,
there exist
$\epsilon(d),C(d)>0$
(depending on just the dimension
of $\Omega$)
such that
for each
integer $k \geq 1$
\begin{align*}
\widetilde{g}&<\epsilon
\quad\Rightarrow\quad
\eigenvalsym{T_2}{k}{\grp}{\twsthom}
\leq
\eigenvalsym{T_1}{k}{\grp}{\twsthom}
  +C(\widetilde{q}+\widetilde{r})
  +C\eigenvalsym{T_1}{k}{\grp}{\twsthom}
    (\widetilde{g}+\widetilde{r}).
\end{align*}
\end{lemma}

\begin{proof}
Let $k \geq 1$ be a given integer, and let $\{u_1,\ldots,u_k\}$
be an $L^2(\lipdom,g_1)$ orthonormal collection of eigenfunctions of $T_1$
with respective eigenvalues
$\eigenvalsym{T_1}{1}{\grp}{\twsthom},\ldots,\eigenvalsym{T_1}{k}{\grp}{\twsthom}$.
Then the span $X$ of $u_1,\ldots,u_k$
is a $k$-dimensional subspace
of
$
\invproj{\grp}{\twsthom}(
  \sobd{\dbdy\lipdom}(\lipdom,g_2)
)
$,
so by the min-max characterization of eigenvalues
we have
\begin{equation*}
\eigenvalsym{T_2}{k}{\grp}{\twsthom}
\leq
\max_{0 \neq u \in X} \frac{T_2(u,u)}{\nm{u}_{L^2(\lipdom, g_2)}^2},
\end{equation*}
and we will accordingly complete the proof
by arranging a suitable upper bound for
the right-hand side of this last inequality.
Throughout the proof
$C=C(d)$ and $\epsilon=\epsilon(d)$
will denote constants
whose value may change from line to line
but at each step
(and so also after the finitely many steps of the proof)
can be chosen in terms of $d$ alone
to ensure the various asserted inequalities
whenever $\widetilde{g}<\epsilon$.

Now suppose we are given some
$0 \neq u \in X$.
We then have
\begin{equation*}
\abs[\Big]{
  \nm{u}_{L^2(\lipdom,g_2)}^2
  -\nm{u}_{L^2(\lipdom,g_1)}^2
}
\leq
C\widetilde{g}\nm{u}_{L^2(\lipdom,g_1)}^2,
\end{equation*}
and so
\begin{equation*}
\frac{T_2(u,u)}{\nm{u}_{L^2(\lipdom,g_2)}^2}
\leq
\frac{T_2(u,u)}{\nm{u}_{L^2(\lipdom,g_1)}^2}
  (1+C\widetilde{g}).
\end{equation*}
(For both inequalities
one can use the fact that
$g_1+t(g_2-g_1)$
is a metric,
whose volume form is moreover a smooth function of $t$,
say for all $\abs{t}<1$
and for $\epsilon$,
so $\widetilde{g}$, sufficiently small.)

Similarly,
\begin{align*}
\abs{
  \nm{du}_{L^2(\lipdom,g_2)}^2
  -\nm{du}_{L^2(\lipdom,g_1)}^2
}
&\leq
C\widetilde{g}\nm{u}_{\sob(\lipdom,g_1)}^2,
\\[1ex]
\abs{
  \sk{q_2 u, u}_{L^2(\lipdom,g_2)}
  -\sk{q_1 u, u}_{L^2(\lipdom,g_1)}
}
&\leq
C\widetilde{q}\nm{u}_{L^2(\lipdom,g_1)}^2,
\\[1ex]
\abs{
  \sk{r_2 u, u}_{L^2(\rbdy\lipdom, g_2)}
  -\sk{r_1 u, u}_{L^2(\rbdy\lipdom, g_1)}
}
&\leq
C(
  \widetilde{g}\nm{r_1}_{C^0(\lipdom)}
  +\nm{r_2-r_1}_{C^0(\rbdy\lipdom)}
)\nm{u}_{L^2(\rbdy\lipdom,g_1)}^2,
\end{align*}
and, using \eqref{eqn:trace_inequality},
\begin{equation*}
\nm{u}_{L^2(\rbdy\lipdom,g_1)}^2
\leq
C_\Omega^{\mathrm{tr}}\nm{u}_{H^1(\lipdom,g_1)}^2.
\end{equation*}
Finally, recalling again \eqref{eqn:bilinear_form_def},
\begin{align*}
\nm{u}_{\sob(\lipdom,g_1)}^2
&=
T_1(u,u)
  +\sk{q_1u,u}_{L^2(\lipdom,g_1)}
  +\sk{r_1u,u}_{L^2(\rbdy\lipdom,g_1)}
\\[1ex]
&\leq
\eigenvalsym{T_1}{k}{\grp}{\twsthom}
    \nm{u}_{L^2(\lipdom,g_1)}^2
  +\nm{q_1}_{C^0(\lipdom)}
      \nm{u}_{L^2(\lipdom,g_1)}^2
\\&\hphantom{{}\leq{}}
  +C_\Omega^{\mathrm{tr}}
      \nm{r_1}_{C^0(\rbdy\lipdom)}
      \Bigl(
        \sk{u,\abs{du}_{g_1}}_{L^2(\lipdom,g_1)}
        +\nm{u}_{L^2(\lipdom,g_1)^2}
      \Bigr),
\end{align*}
where in particular we have again applied
\eqref{eqn:trace_inequality}.
The proof is now concluded
by applying the Cauchy--Schwarz inequality
(with an absorption factor, to the term
with the trace inequality)
in order to obtain a bound
on $\nm{u}_{\sob(\lipdom,g_1)}^2$
in terms of $\nm{u}_{L^2(\lipdom,g_1)}^2$
and then combining the foregoing
inequalities.
\end{proof}

\section{Umbilics of free boundary minimal surfaces in the ball}
Following H. Hopf's proof
\cite{HopfCMCsphere}
(presented also in the lecture notes
\cite[Chapter VI]{HopfLectures})
that every immersed topological sphere 
of constant mean curvature in $\R^3$
is round,
the Hopf differential
(defined below in the setting of the present article)
has proved a valuable tool
in the study of constant-mean-curvature
surfaces in space forms.
In the context of
free boundary minimal surfaces
in $\B^n$ it has been exploited
(see \cites{Nitsche1985, FraserSchoenUniqueness})
to establish uniqueness among solutions
having the topological type of a disc.
For surfaces in $\B^3$ of higher topology
it can be used to count umbilics,
as we briefly now explain;
cf. \cite[Proposition 1.5]{Law70}
concerning the umbilics of closed surfaces
in the round $3$-sphere.
A careful and detailed treatment
of umbilics of constant-mean-curvature surfaces
with boundary can be found in \cite{Choe2002}.
A generalization of Proposition \ref{pro:umbilic_count}
can be found in \cite{DomingosSantosVitorio}.

Let
$\varphi \colon \Sigma \to \B^3$
be a free boundary minimal immersion
of a compact, connected surface with boundary.
By the regularity theory for such maps
(we can appeal specifically to
\cite{GrueterHildebrandtNitscheBoundaryBehavior})
$\Sigma$ then inherits (from $\varphi$) the structure of a bordered Riemann surface.
We write $\gamma$ for the genus of $\Sigma$
and $\beta$ for the number of its boundary components.
Then the double $\underline{\Sigma}$ of $\Sigma$
is a closed Riemann surface of genus
$2\gamma + \beta - 1$.
(The notions of a bordered Riemann surface
and its double, sometimes called Schottky double,
can be found in
\cite[II.3]{AhlforsSario} and many other
textbooks treating Riemann surfaces.)

We define,
as a section of $(T^*\Sigma \otimes \C)^{\odot 2}$
(the symmetric product of $T^*\Sigma \otimes \C$
with itself),
the Hopf differential
$\Phi$ of $\varphi$ by
\begin{equation*}
\Phi(X,Y)
=
\frac{1}{4}A(\varphi_{\ast}X,\varphi_{\ast}Y)
-\frac{1}{4}A(\varphi_{\ast}JX,\varphi_{\ast}JY)
  -\frac{1}{4}iA(\varphi_{\ast}JX,\varphi_{\ast}Y)-\frac{1}{4}iA(\varphi_{\ast}X,\varphi_{\ast}JY),
\end{equation*}
where $X$ and $Y$ are sections of $T\Sigma \otimes \C$,
$J$ is the complex structure of $\Sigma$,
$A$ is the second fundamental form of $\varphi$
extended to be $\C$-bilinear and the differential of $\varphi$ is similarly understood extended by $\C$-linearity.
Equivalently, if $w$ is a local complex coordinate
on $\Sigma$, then
\begin{equation*}
\Phi = A\bigl(\varphi_\ast(\partial_w),\varphi_{\ast}(\partial_w)\bigr) \, dw^2.
\end{equation*}
By the Codazzi equation
(and using the minimality of $\varphi$
and the flatness of $\R^3 \supset \B^3$)
$\Phi$ is holomorphic inside $\Sigma$.
Furthermore, the free boundary condition
and the fact that $\partial \B^3$ is
totally umbilic
imply that $\partial \Sigma$
is a line of curvature,
whence it follows that
$\Phi$ is real along $\partial \Sigma$.
Consequently $\Phi$ extends
(by complex reflection)
uniquely to a holomorphic quadratic differential $\underline{\Phi}$
on the double $\underline{\Sigma}$.
The zeros of $\underline{\Phi}$ on $\Sigma$
(including on its boundary) are precisely
the umbilics of $\varphi$,
and we define the \emph{order} 
$\ord p$ of each umbilic $p$
to be the multiplicity of the corresponding zero
of $\underline{\Phi}$.
(If $\underline{\Phi}$ vanishes identically,
then every point of $\Sigma$ is an umbilic
of infinite order;
otherwise, by the holomorphicity of $\underline{\Phi}$,
each umbilic is isolated and has finite order.)

\pagebreak[2]

\begin{proposition}[Count of umbilics; cf. \protect{\cite[Theorem 4.2]{DomingosSantosVitorio}}]
\label{pro:umbilic_count}
Let $\varphi \colon \Sigma \to \B^3$
be a free boundary minimal immersion
of a compact, connected surface
having genus $\gamma$ and $\beta$ boundary components.
If $\Sigma$ is not a disc,
then the set $Z$ of umbilics of $\varphi$ is finite and
\begin{equation*}
\sum_{p \in Z \cap \partial \Sigma} \ord p
  +2\sum_{p \in Z\setminus \partial \Sigma}
    \ord p
=
8\gamma + 4\beta - 8.
\end{equation*}
\end{proposition}

\begin{proof}
We define $\underline{\Sigma}$
and $\underline{\Phi}$
as in the preceding paragraph.
If $\Phi$ vanished identically on $\Sigma$,
then $\Sigma$ would be totally umbilic, hence totally geodesic by minimality, and so in particular a disc.
Therefore $\underline{\Phi}$
is a nonzero
holomorphic quadratic differential
on the closed Riemann surface $\underline{\Sigma}$.
Consequently its zeros are isolated,
so (by compactness of $\Sigma$) constitute a finite set,
and moreover
(see for example
\cite[Remark IV.13.5]{Freitag2011})
the sum of their multiplicities
is $4\underline{\gamma}-4$,
where $\underline{\gamma}=2\gamma+\beta-1$
is the genus of $\underline{\Sigma}$.
The proof is now completed by the observation
that the left-hand side of the equation
in the proposition is precisely
the sum of the multiplicities
of the zeros of $\underline{\Phi}$.
\end{proof}

Naturally the symmetries of $\varphi$ constrain its umbilics and their orders.
In particular we have the following result; here we will assume embeddedness, which is essential but serves merely to simplify its statement.

\begin{lemma}[Umbilics fixed by a nontrivial rotation]
\label{lem:umbilics_fixed_by_rots}
Let $\Sigma$ be an embedded free boundary minimal surface in $\B^3$,
let $\rot \in \SOgroup(3)$ be a rotation through angle $2\pi/k$ for some integer $k \geq 3$,
and assume that $\rot \Sigma = \Sigma$.
Then any point of $\Sigma$  fixed by $\rot$ is an umbilic of order at least $k-2$.
\end{lemma}

\begin{proof}
Suppose $p \in \Sigma$
satisfies $\rot p = p$.
Note that $p$ must be an interior point,
since $\rot$ is a symmetry of $\Sigma$
and $k>2$.
We may assume, without loss of generality, that $\rot$ fixes 
the $z$-axis pointwise,
and we can choose $\epsilon>0$
and a complex coordinate $w$
on a neighborhood of $p$ such that
$w$ takes
the component of
$\{x^2+y^2<\epsilon^2\} \cap \Sigma$
containing $p$
onto the open unit disc in $\C$
with $w(p)=0$.
(This can be done essentially by hand
using the assumption of the minimality of $\Sigma$
in $\R^3$, or one can appeal more generally
to the existence of isothermal coordinates for surfaces
and to the Riemann mapping theorem.)
Then $w \circ \rot \circ w^{-1}$
is a conformal automorphism of the unit disc
fixing its center, so (by standard classification results) is a rotation;
since its differential at $0$
is rotation through angle $2\pi/k$
we conclude that in fact
\begin{equation}
\label{rot_in_complex_coords}
w \circ \rot = e^{2 \pi i / k}w.
\end{equation}
Now let $\Phi$ be the Hopf differential of $\Sigma$.
Working with respect to the above complex coordinate $w$,
by the holomorphicity of $\Phi$ we have 
\begin{equation*}
\Phi(w)=\sum_{n=0}^\infty c_n w^n \, dw^2
\end{equation*}
for some coefficients $\{c_n\} \subset \C$.
The assumption $\rot \Sigma = \Sigma$
and the equality \eqref{rot_in_complex_coords}
imply $\Phi(w)=\Phi(w \circ \rot)=\Phi(e^{2 \pi i / k}w)$, and so 
\begin{equation*}
c_n e^{2n \pi i / k}e^{4 \pi i/k} = c_n
\end{equation*}
for all $n \geq 0$, whence $c_n=0$ for all $n<k-2$, completing the proof.
\end{proof}

\setlength{\parskip}{1ex plus 1pt minus 1pt}
\bibliography{fbms-bibtex}

\printaddress
\end{document}